%% file: document.tex
\documentclass[11pt,a4paper]{article}

\usepackage[T1]{fontenc}
\usepackage[utf8]{inputenc}
\usepackage[bottom=1in, top=1.5in,right=2cm,left=2cm]{geometry}
\usepackage{amssymb}
\usepackage{csquotes}
\usepackage{amsfonts}

\usepackage[labelfont=bf]{caption}
\usepackage{subcaption}
\usepackage{microtype}
\usepackage{lmodern}
\usepackage[USenglish]{babel}
\usepackage{hyperref}
\usepackage{flafter}
\usepackage{amsthm}
\usepackage{amsmath}
\usepackage{amssymb}
\usepackage{madefs}%
\usepackage{algorithm}
\usepackage{algpseudocode}
\usepackage[mathscr]{eucal}
\usepackage{pgfplots,tikz}
\pgfplotsset{compat=newest}
\usetikzlibrary{plotmarks}
\usepackage{tcolorbox}
\usetikzlibrary{matrix,arrows}
\usepackage{bm}
\usepackage{enumerate} 
\usepackage[capitalize]{cleveref}
\crefname{definition}{Definition}{Definitions}
\crefname{rem}{Remark}{Remarks}
\crefname{cor}{Corollary}{Corollaries}
\crefname{thm}{Theorem}{Theorems}
\crefname{alg}{Algorithm}{Algorithms}
\crefname{lem}{Lemma}{Lemmas}
\crefname{exa}{Example}{Examples}
\crefname{pro}{Proposition}{Propositions}
\crefname{defpro}{Definition and Proposition}{}
\crefname{equation}{Equation}{Equations}
\theoremstyle{plain}
\newtheorem{definition}{Definition}[section]
\newtheorem{rem}{Remark}
\newtheorem{exa}{Example}
\theoremstyle{plain}

\newtheorem{thm}[definition]{Theorem}

\newtheorem{pro}[definition]{Proposition}

\newlength\figureheight
\newlength\figurewidth

\hypersetup{
    colorlinks,
    linkcolor={blue!50!black},
    citecolor={green!50!black},
    urlcolor={red!80!black}
}

\date{\today}
\author{
	{\large Abdul-Lateef Haji-Ali\textsuperscript{a}, Fabio Nobile\textsuperscript{b}, Raúl Tempone\textsuperscript{c}, Sören Wolfers\textsuperscript{c}\footnote{Corresponding author. Email address: \texttt{soeren.wolfers@kaust.edu.sa}}}\\[0.5em]
{\small \textsuperscript{a} Oxford University}\\
{\small \textsuperscript{b}  École Polytechnique Fédérale de Lausanne (EPFL), CSQI-MATH}\\
{\small \textsuperscript{c} King Abdullah University of Science and Technology (KAUST), Applied Mathematics and Computational Science}\\
}

\title{Multilevel weighted least squares polynomial approximation}

\newcommand{\work}{\text{Work}}
\newcommand{\mix}{\text{mix}}

\newcommand{\smol}{\mathcal{S}}

\newcommand{\norm}[2]{\|#1\|_{#2}}

\newcommand{\cond}{c}

\newcommand{\expect}{\mathbb{E}}
\newcommand{\prob}{\mathbb{P}}
\newcommand{\measure}{\mu}
\newcommand{\samplingmeasure}{\nu}
\newcommand{\density}{\rho}
\newcommand{\optimaldistribution}{\samplingmeasure^{*}}
\newcommand{\optimaldensity}{\density^{*}}
\newcommand{\weight}{w}
\newcommand{\optimalweight}{\weight^{*}}
\newcommand{\arcsine}{p^{\infty}}

\newcommand{\p}{v}
\newcommand{\pss}{\mathcal{P}}
\newcommand{\onb}{B}
\newcommand{\leg}{P}
\newcommand{\her}{H}
\newcommand{\vsp}{V}
\newcommand{\dvsp}{m}
\newcommand{\NS}{N}
\renewcommand{\P}{\Pi}
\newcommand{\coeff}{c}

\newcommand{\neighbors}{\mathcal{N}}
\newcommand{\mi}{\mathbf{\k}} 
\newcommand{\mis}{\mathcal{I}}
\newcommand{\mia}{\mathcal{A}}
\newcommand{\mip}{\eta} 
\newcommand{\mipp}{\eta'} 

\newcommand{\ps}{\gamma}
\newcommand{\psmi}{\bm{\ps}}
\newcommand{\domPS}{\Gamma}
\newcommand{\dps}{d}

\newcommand{\rs}{f}
\newcommand{\f}{g}
\newcommand{\n}{n}
\newcommand{\QoI}{Q}
\newcommand{\pde}{u}
\newcommand{\rhs}{g}

\newcommand{\dpde}{D}

\newcommand{\genericmeasure}{\pi}
\newcommand{\gap}{g}
\newcommand{\mathup}[1]{\text{\textup{#1}}}
\newcommand{\dimp}{\sigma}

\renewcommand{\k}{k}
\renewcommand{\L}{L}

\newcommand{\F}{F}

\newcommand{\rate}{\lambda}
\newcommand{\lrate}{t}

\newcommand{\uqsummability}{\alpha}

\renewcommand{\sc}{\beta_{s}}
\newcommand{\wc}{\beta_{w}}

\algdef{SE}[DOWHILE]{Do}{doWhile}{\algorithmicdo}[1]{\algorithmicwhile\ #1}%

\begin{document}
\maketitle
\abstract{
Weighted least squares polynomial approximation uses random samples to determine projections of functions onto spaces of polynomials. It has been shown that, using an optimal distribution of sample locations, the number of samples required to achieve quasi-optimal approximation in a given polynomial subspace scales, up to a logarithmic factor, linearly in the dimension of this space. However, in many applications, the computation of samples includes a numerical discretization error. Thus, obtaining polynomial approximations with a single level method can become prohibitively expensive, as it requires a sufficiently large number of samples, each computed with a sufficiently small discretization error. As a solution to this problem, we propose a multilevel method that utilizes samples computed with different accuracies and is able to match the accuracy of single-level approximations with reduced computational cost. We derive complexity bounds under certain assumptions about polynomial approximability and sample work. Furthermore, we propose an adaptive algorithm for situations where such assumptions cannot be verified a priori. Finally, we provide an efficient algorithm for the sampling from optimal distributions and an analysis of computationally favorable alternative distributions. Numerical experiments underscore the practical applicability of our method.\\

 }
\pagenumbering{roman}
\pagestyle{headings}
\pagenumbering{arabic}

\input{./content/introduction}
\input{./content/dpls}
\input{./content/sampling}

\input{./content/nonadaptive}
\input{./content/adaptive}
\input{./content/UQ}
\input{./content/numerics}

\input{./content/conclusion}
\input{./content/acknlowledgments}


\bibliography{library.bib}{}
\bibliographystyle{plain}
\end{document}

%% file: content/introduction.tex
\section{Introduction}

A common goal in uncertainty quantification \cite{LeMaitreKnio2010} is the approximation of response surfaces
\begin{equation*}
\psmi\mapsto \rs(\psmi):=\QoI(\pde_{\psmi})\in\R,
\end{equation*}
which describe how a quantity of interest $\QoI$ of the solution $\pde_{\psmi}$ to some partial differential equation (PDE) depends on parameters $\psmi\in\domPS\subset\R^\dps$ of the PDE. The non-intrusive approach to this problem is to evaluate the response surface for finitely many values of $\psmi$ and then to use an interpolation method, such as (tensor-)spline interpolation \cite{deb2001solution}, kernel-based approximation (kriging) \cite{WolfersNobileTempone2016a,GriebelRieger2015}, or (global) polynomial approximation \cite{LeMaitreKnio2010}.

 In this work, we study a variant of polynomial approximation in which least squares projections onto finite-dimensional polynomial subspaces are computed using values of $\rs$ at finitely many random locations. More specifically, given a probability measure $\measure$ on the parameter space $\domPS$ and a polynomial subspace $\vsp\subset L^2_{\measure}(\domPS)$, the approximating polynomial is determined as
\begin{equation}\label{eq:P0}
\P_\vsp \rs:=\argmin_{\p\in \vsp}\norm{\rs-\p}{\NS},
\end{equation}
where $\norm{\cdot}{\NS}$ is a discrete approximation of the $L^2_{\measure}(\domPS)$ norm that is based on evaluations in finitely many randomly chosen sample locations $\psmi_j\in\domPS$, $j\in\{1,\dots,\NS\}$ and a weight function $\weight\colon\domPS\to\R$. 

The case where equally weighted samples are drawn independently and identically distributed from the underlying probability measure itself, $\psmi_j\sim\measure$, has been popular among practitioners for a long time and has been given a thorough theoretical foundation in the past decade  \cite{ChkifaCohenMiglioratiEtAl2015,Cohen2013,MiglioratiNobileTempone2015}. 
More recently, the use of alternative sampling distributions and non-constant weights was studied in \cite{narayan2014christoffel,cohen2016optimal,hampton2015coherence}. In particular, \cite{hampton2015coherence} presented a sampling distribution $\optimaldistribution_{\vsp}$ and a corresponding weight function for which the number of samples required to determine quasi-optimal approximations within $\vsp$ is bounded by $\dim \vsp$ up to a logarithmic factor. (This result was proved in \cite{hampton2015coherence} for total degree polynomial spaces and generalized in \cite{cohen2016optimal} to more general function spaces.) Since this distribution depends on $\vsp$, it is natural to ask how samples can be efficiently obtained from it and whether there is an alternative that works equally well for all polynomial subspaces $\vsp$. 
To address the first question, we present and analyze an efficient algorithm to generate samples from $\optimaldistribution_{\vsp}$ in the case where $\domPS$ is a product domain and $\measure$ is a product measure. For more general cases, we also study Markov chain methods for sample generation and analyze the effect of small perturbations of the sampling distribution on the convergence estimates of \cite{hampton2015coherence,cohen2016optimal}. To address the second question, we provide upper and lower bounds on $\optimaldistribution_{\vsp}$ in the case where $\domPS$ is a hypercube. The lower bound allows us to make the error estimates obtained in \cite{cohen2016optimal} more explicit. The upper bound shows that the arcsine distribution, which was proposed in \cite{narayan2014christoffel}, performs just as well as $\optimaldistribution_{\vsp}$ up to a constant that is independent of $\vsp$. This is advantageous for adaptive algorithms in which the polynomial subspace and the corresponding optimal sampling distribution vary during the iterations. However, the constant mentioned above does depend exponentially on the dimension $\dps$ of $\domPS$. 
\\

 To motivate the main contribution of this work, namely the multilevel weighted least squares polynomial approximation method, we note that the response surface $\rs$ from the beginning of this introduction cannot be evaluated exactly. Indeed, in most cases, the computation of $Q(\pde_{\psmi})$ requires the numerical solution of a PDE. Thus, we can only compute approximations of $\rs$ whose accuracy and computational work are determined by the PDE discretization. If we simply applied  polynomial least squares approximation using a sufficiently fine discretization of the PDE for all evaluations, then we would quickly face prohibitively long runtimes. For this reason, we introduce a multilevel method that combines lots of cheap samples using coarse discretizations with relatively few more expensive samples using fine discretizations of the PDE. In the recent decade, such multilevel algorithms have been studied intensely for the approximation of expectations \cite{harbrecht2013multilevel,heinrich2001multilevel,KuoScheichlSchwabEtAl2015,Haji-AliNobileTamelliniEtAl2015}. The goal of this paper is to extend this earlier work to the reconstruction of the full response surface, using global polynomial approximation and estimating the resulting error in the $L^2_{\measure}$ norm.
 
To describe the multilevel method, assume that we want to approximate a function $\rs$. Assume furthermore that we can only evaluate functions $\rs_{l}$ with $\rs_{l}{\rightarrow}\rs$ as $l\to\infty$ in a suitable sense and that the cost per evaluation increases as $l\to\infty$. A straightforward approach to this situation is to apply least squares approximation to some $\rs_{L}$ that is sufficiently close to $\rs$. The theory of (weighted) polynomial least squares approximation then provides conditions on the number of samples required to achieve quasi-optimal approximation of $\rs_{L}$ within a given space of polynomials $\vsp_{L}$. However, this approach can be computationally expensive, as each evaluation of $\rs_{L}$ requires the numerical solution of a PDE using a fine discretization. As an alternative, our proposed multilevel algorithm starts out with a least squares approximation of $\rs_{0}$ using a relatively large polynomial subspace $\vsp_0$ and correspondingly many samples. To correct for the committed error $\rs-\rs_{0}$, the algorithm then adds polynomial approximations of $\rs_{l}-\rs_{l-1}$ that lie in subspaces $\vsp_l$, $l\in\{1,\dots,\L\}$. 
Since we assume that $\rs_{l}\to\rs$ in an appropriate sense, the differences $\rs_{l}-\rs_{l-1}$ may be approximated using smaller polynomial subspaces for $l\to\infty$. Exploiting this fact, it is possible to obtain approximations with significantly reduced computational work. 
Indeed, we show that under certain conditions the work that the multilevel method requires to attain an accuracy of  $\epsilon>0$ is the same as the work that regular least squares polynomial approximation would require if $\rs$ could be evaluated exactly. It is clear that such a result is not always possible. For example, if $\rs$ were constant, then polynomial least squares approximations in any fixed polynomial subspace would yield the exact solution given a sufficiently large sample size. This means that the work required to achieve an accuracy $\epsilon>0$ would be bounded as $\epsilon\to 0$, which can clearly not be true for an algorithm that uses evaluations from approximate functions $\rs_l$ that become more expensive to evaluate as $l\to\infty$. Instead, the optimal computational work required for an accuracy of $\epsilon>0$ by such an algorithm would be determined by the convergence of $\rs_l\to \rs$ and by the work that is required for evaluations of $\rs_l$. 
Our results show that for many problems, the two cases described above are dichotomous. This means that the total computational work is determined either by the convergence and work associated with $\rs_{l}$ or by the convergence of polynomial least squares approximation using exact evaluations (see \Cref{thm:main} for a more formal statement). \\

The remainder of this paper is structured as follows. In \Cref{sec:dpls}, we review the theoretical analysis of weighted least squares approximation. In \Cref{sec:sampling}, we discuss different sampling strategies. We propose algorithms to sample the optimal distribution and we discuss the consequences of using perturbed distributions. In \Cref{sec:nonadaptive}, we introduce a novel multilevel algorithm and prove our main results concerning the work and convergence of this algorithm. For situations in which the regularity of $\rs$ and the convergence of $\rs_l$ are not known, we propose an adaptive algorithm in \Cref{sec:adaptive}. We discuss the applicability of our method to problems in uncertainty quantification in \Cref{sec:uq}. Finally, we present numerical experiments in \Cref{sec:numerics}.

%% file: content/dpls.tex
\section{Weighted least squares polynomial approximation}

\label{sec:dpls}
In this section, we provide a short summary of the theory of weighted discrete least squares polynomial approximation, closely following \cite{cohen2016optimal}.

Assume that we want to approximate a function $\rs\in L^2_{\measure}(\domPS)$, where $\domPS\subset\R^\dps$ and $\measure$ is a probability measure on $\domPS$. 
The strategy of weighted discrete least squares polynomial approximation is to 
\begin{itemize}
	\item choose a finite-dimensional space $\vsp\subset L^2_{\measure}(\domPS)$ of polynomials on $\domPS$
	\item choose a function $\density\colon\domPS\to\R$ that satisfies $\int_{\domPS}\density(\psmi) \measure(d\psmi)=1$ and $\density>0$
	\item generate $\NS>0$ independent random samples from the \emph{sampling distribution} $\samplingmeasure$ defined by $\frac{\text{d}\samplingmeasure}{\text{d}\measure}:=\density$, 
	$$
	\psmi_j\sim \samplingmeasure, \quad j\in\{1,\dots,\NS\}.
	$$
	Here, $\frac{\text{d}\samplingmeasure}{\text{d}\measure}$ denotes the density, or Radon-Nikodym derivative, of the probability measure $\samplingmeasure$ with respect to the reference measure $\measure$.
	\item evaluate $\rs$ at $\psmi_j$, $j\in\{1,\dots,\NS\}$
	
	\item define the \emph{weight function} $\weight:=\frac{1}{\density}\colon \domPS\to\R$
	\item  and finally define the \emph{weighted discrete least squares approximation}
	\begin{equation}
	\label{eq:dpls:def}
	\P_{\vsp}\rs:=\argmin_{\p\in \vsp} \norm{\rs-\p}{\NS},
	\end{equation}
	where 
	\begin{equation}
	\label{eq:dpls:discretenorm}
	\norm{\f}{\NS}^2:=\langle \f,\f\rangle_{\NS}:=\frac{1}{\NS}\sum_{j=1}^{\NS} w(\psmi_j)|\f(\psmi_j)|^2\quad \forall \f\colon\domPS\to\R.
	\end{equation}
\end{itemize}
It is straightforward to show that the coefficients $\mathbf{\p}$ of $\P_{\vsp}\rs$ with respect to any basis $(\onb_j)_{j=1}^{\dvsp}$ of $\vsp$ are given by
\begin{equation}
\label{eq:dpls:computation}
	\mathbf{G}\mathbf{\p}=\mathbf{\coeff},
\end{equation}
with $\mathbf{G}_{ij}:=\langle \onb_i,\onb_j\rangle_{\NS}$, and $\coeff_j:=\langle \rs,\onb_j\rangle_{\NS}$, $i,j\in\{1,\dots,\dvsp\}$, assuming that $\mathbf{G}$ is invertible. If $\mathbf{G}$ is not invertible, then \Cref{eq:dpls:def} has multiple solutions and we define $\Pi_{\vsp}\rs$ as the one with the minimal $L^2_{\measure}(\domPS)$ norm. 
\begin{rem}
	\label{rem:matvec}
	Assembling the matrix $\mathbf{G}$ requires $\mathcal{O}(m^2N)$ operations. However, using the fact that $\mathbf{G}=\mathbf{M}^{\top}\mathbf{M}$ for $\mathbf{M}_{ij}:=N^{-1/2}\sqrt{\weight(\psmi_i)}\onb_j(\psmi_i)$, matrix vector products with $\mathbf{G}$ can be computed at the lower cost $\mathcal{O}(mN)$ as $\mathbf{G}\mathbf{x}=\mathbf{M}^{\top}(\mathbf{M}\mathbf{x})$. 
\end{rem}

Since $\weight\density=1$, the semi-norm defined in \Cref{eq:dpls:discretenorm} is a Monte Carlo approximation of the $L^2_{\measure}(\domPS)$ norm. Therefore, we may expect that the error $\norm{\rs-\P_{\vsp}\rs}{L^2_{\measure}(\domPS)}$ is close to the optimal one,
\begin{equation}\label{eq:best2}
e_{\vsp,2}(\rs):=\min_{\p\in\vsp}\norm{\rs-\p}{L^2_{\measure}(\domPS)}.
\end{equation}
Part (iii) of \Cref{thm:dpls} below shows that this is true in expectation, provided that the number of samples $\NS$ is coupled appropriately to the dimension $\dvsp=\dim\vsp$ of the approximating polynomial subspace and provided that we ignore outcomes where $\mathbf{G}$ is ill-conditioned.  For results in probability, we need to replace the best $L^2_{\measure}(\domPS)$ approximation by the best approximation in a weighted supremum norm,
\begin{equation}
\label{eq:bestlinf}
e_{\vsp,\weight,\infty}(\rs):=\inf_{\p\in\vsp}\sup_{\psmi\in\domPS}|\rs(\psmi)-\p(\psmi)|\sqrt{\weight(\psmi)}.
\end{equation}

\begin{thm}[\textbf{Convergence of weighted least squares, \cite[Theorem 2]{cohen2016optimal}}]
	\label{thm:dpls}
For arbitrary $r>0$, define 
$$\kappa:=\frac{1/2-1/2\log 2}{1+r}.$$ Assume that for all $\psmi\in\domPS$ there exists $\p\in\vsp$ such that $\p(\psmi)\not =0$ and denote by $(\onb_j)_{j=1}^{\dvsp}$ an $L^2_{\measure}$-orthonormal basis of $\vsp$. Finally, assume that 
\begin{equation}
\label{eq:K}
K_{\vsp,\weight}:=\norm{\weight \sum_{j=1}^{\dvsp}\onb_j^2}{L^\infty(\domPS)}\leq \kappa\frac{\NS}{\log \NS}.
\end{equation}
\begin{enumerate}[(i)]
	\item With probability larger than $1-2\NS^{-r}$, we have
	\begin{equation}
	\|\mathbf{G}-\mathbf{I}\|\leq\frac{1}{2},
	\end{equation}
	where $\mathbf{G}$ is the matrix from \Cref{eq:dpls:computation}, $\mathbf{I}$ is the identity matrix, and $\|{\cdot}\|$ denotes the spectral matrix norm.
	\item If $\|{\mathbf{G}-\mathbf{I}}\|\leq 1/2$, then for all $\rs$ with $\sup_{\psmi\in\domPS}|\rs(\psmi)|\sqrt{\weight(\psmi)}<\infty$, we have
	\begin{equation*}
		\norm{\rs-\P_{\vsp}\rs}{L^2_{\measure}(\domPS)}\leq (1+\sqrt{2})e_{\vsp,\weight,\infty}(\rs).
	\end{equation*}
\item If $\rs\in L^2_{\measure}(\domPS)$, then 
\begin{equation*}
\expect\norm{\rs-\P^{\cond}_{\vsp}\rs}{L^2_{\measure}(\domPS)}^2\leq (1+\frac{4\kappa}{\log \NS})e^2_{\vsp,2}(\rs)+2\norm{\rs}{L^2_{\measure}(\domPS)}^2\NS^{-r},
\end{equation*}
where $\expect$ denotes the expectation with respect to the $\NS$-fold draw from the sampling distribution $\samplingmeasure$ and 
\begin{equation*}
\P^{\cond}_{\vsp}\rs:=\begin{cases}
\P_{\vsp}\rs\quad\text{if }\|{\mathbf{G}-\mathbf{I}}\|\leq\frac{1}{2},\\
0\quad\text{otherwise}.
\end{cases}
\end{equation*}
\end{enumerate}
\end{thm}
\begin{proof}
	It is proved in \cite[Theorem 2]{cohen2016optimal} that the bound in part (ii) holds for a fixed $\rs$ with probability larger than $1-2\NS^{-r}$. A look at the proof reveals that the bound only depends on the event $\|{\mathbf{G}-\mathbf{I}}\|\leq 1/2$ and not on the specific choice of $\rs$. The remaining claims are exactly as in \cite{cohen2016optimal}.
\end{proof}


%% file: content/sampling.tex

\section{Sampling strategies}
\label{sec:sampling}
It was observed in \cite{cohen2016optimal} that the constant $K_{\vsp,\weight}$ in \Cref{eq:K} satisfies
\begin{equation}
\label{eq:optimalK}
\aligned
m  = & \int \sum_{j=1}^{\dvsp} \onb_j(\psmi)^2 \measure(d\psmi)\\
\le & \left( \int  w^{-1}(\psmi) \measure(d\psmi) \right)  \left\| \sum_{j=1}^{\dvsp} w \onb_j^2 \right\|_{L^\infty(\domPS)} \\
 = & K_{\vsp,\weight}
\endaligned
\end{equation}
and that the inequality becomes an equality for the weight $\weight^{*}_{\vsp}={\optimaldensity_{\vsp}}^{-1}$ that is associated with the density
\begin{equation}
\label{eq:optimalmeasure}
\optimaldensity_{\vsp}(\psmi):=\frac{1}{\dvsp}\sum_{j=1}^{\dvsp}|\onb_j(\psmi)|^2.
\end{equation}
For this choice, \Cref{thm:dpls} roughly asserts that the number of samples required to determine a near-optimal approximation of $\rs$ in an  $\dvsp$-dimensional space $\vsp$ is smaller than $C\dvsp \log \dvsp$ for some $C>0$. In the remainder of this work, we refer to $\optimalweight_{\vsp}$, $\optimaldensity_{\vsp}$, and
\begin{equation}
\label{eq:optimaldistribution}
\optimaldistribution_{\vsp}\colon \quad \frac{\mathrm{d}\optimaldistribution_{\vsp}}{\text{d}\measure}:=\optimaldensity_{\vsp}
\end{equation}
as the \emph{optimal} weight, density, and distribution, respectively.
Since the optimal distribution $\optimaldistribution_{\vsp}$ depends on $\vsp$, practical implementations need to address the question how to obtain samples from $\optimaldistribution_{\vsp}$
 for general subspaces $\vsp$. Furthermore, since $\optimaldensity_{\vsp}$ depends on $\vsp$, the weight in $e_{\vsp,\infty}(\rs)$ in part (ii) of \Cref{thm:dpls} does as well. To address these issues, we present two types of results in this section.

First, we discuss how to obtain samples from $\optimaldistribution_{\vsp}$.
For the case where $\domPS$ is a $d$-dimensional hypercube and $\measure=\bigotimes_{j=1}^{d}\measure_{j}$ with $\measure_{j}$ satisfying certain assumptions, we propose a method for the generation of $\NS$ samples whose computational work is bounded in expectation by the product $Kd\NS$ with a constant $K$ that depends only on the measures $\measure_{j}$.
For non-product domains or measures, we briefly discuss how to use Markov chain Monte Carlo (MCMC) sampling for the generation of samples from approximate distributions and how perturbations of the sampling distributions affect the error estimates.

Second, we prove for the case $\domPS=[0,1]^d$ and under a rather permissive assumption on $\measure$ (in particular, $\measure$ must be absolutely continuous with respect to the Lebesgue measure) that for any polynomial subspace $\vsp$ the density of the optimal distribution $\optimaldistribution_{\vsp}$ with respect to the Lebesgue measure $\lambda$ satisfies
\begin{equation}
\label{eq:temparcsine}
C^{-d}<\frac{\text{d}\optimaldistribution_{\vsp}}{\text{d}\lambda}\leq C^d \arcsine_{d},
\end{equation}
where $0<C<\infty$ is independent of $\vsp$, and $\arcsine_{d}$ is the Lebesgue density of the $d$-dimensional arcsine distribution,
\begin{equation}
\label{eq:arcdef}
\arcsine_{d}(\psmi):=\prod_{j=1}^{d}\frac{1}{\pi\sqrt{\psmi_j(1-\psmi_j)}}.
\end{equation}

The lower bound in \Cref{eq:temparcsine} implies that the optimal weight $\optimalweight_{\vsp}$ is bounded above by $C^{d}\frac{d\measure}{d\lambda}$, which can be used to make the error estimate in part (ii) of \Cref{thm:dpls} more explicit.

By the upper bound, we may use samples from the $d$-dimensional arcsine distribution instead of the optimal distribution.
	Indeed, the upper bound implies that the weight function $\weight$ associated with the arcsine distribution satisfies $K_{\vsp,\weight}\leq C^dm$. Thus, the required number of samples is increased by at most a factor that is independent of $\vsp$. The advantages are that samples from the arcsine distribution can be generated efficiently, that we can use samples from the same distribution for all polynomial subspaces, and that the weight $\weight$ is easy to analyze and independent of $\vsp$.

\subsection{Sampling from the optimal distribution}\label{sec:optimal-sampling}
We now describe an efficient algorithm to obtain samples from $\optimaldistribution_{\vsp}$ in the case when $\domPS$ is a Cartesian product,  $\measure$ is a product measure, and $\vsp$ is downward closed.
\begin{definition}[\textbf{Downward closedness}]
	Let $\N:=\{0,1,\dots\}$.
	A set $\mis\subset\N^d$ is called \emph{downward closed} if $\mip \in\mathcal{I}$ implies $\mipp\in\mathcal{I}$ for any $\mipp\in\N^{\dps}$ with $\mipp\leq \mip$ componentwise.

	A space $\vsp$ of polynomials on a Cartesian product domain $\domPS=\prod_{j=1}^{d}I_j$
	with $I_j\subset\R$ is called \emph{downward closed} if
	it is the span of monomials,
	 $$
	\vsp=\operatorname{span}\{\psmi^{\mip}=\prod_{j=1}^{\dps}\ps_{j}^{\mip_j}:\mip\in\mathcal{I}\},
	$$
	for some downward closed set $\mis\subset \N^d$.
\end{definition}
\begin{rem}
	Observe that any non-trivial downward closed polynomial space $\vsp$ includes the constant functions and thus satisfies the assumption of \Cref{thm:dpls} that for all $\psmi\in\domPS$ there exists $\p\in\vsp$ with $\p(\psmi)\not =0$.
\end{rem}
We first discuss the case $\domPS=[0,1]^d$ and $\measure=\lambda$ the Lebesgue measure. For any downward closed subspace
\begin{equation*}
\vsp=\vspan\{\psmi^{\mip}:\mip\in\mis\}\subset L^2_{\lambda}([0,1]^d)
\end{equation*} with $\mis\subset\N^{d}$ and $|\mis|=\dim \vsp=\dvsp$, 
an orthonormal basis is then given by
\begin{equation*}
(\leg_{\mip})_{\mip\in\mis}
\end{equation*}
where
\begin{equation*}
\leg_{\mip}(\psmi):=\prod_{j=1}^{d}\leg_{\mip_j}(\psmi_j)
\end{equation*}
and $(\leg_n)_{n\in \N}$ are the Legendre polynomials on $[0,1]$, which are orthonormal with respect to the one-dimensional Lebesgue measure. By orthonormality, each $\leg_{\mip}^2$ may be interpreted as a probability density with respect to the Lebesgue measure. Thus,
\begin{equation*}
\frac{\text{d}\optimaldistribution_{\vsp}}{\text{d}\lambda}=\optimaldensity_{\vsp}=\frac{1}{\dvsp}\sum_{\mip\in\mis}\leg_{\mip}^2
\end{equation*} may be interpreted as mixture of $\dvsp$ probability densities. An efficient strategy to obtain samples from $\optimaldistribution_{\vsp}$ is therefore to first choose $\mip\in\mis$ at random and then generate a sample from the distribution with Lebesgue density $\leg^2_{\mip}$. Since $\leg^2_{\mip}=\prod_{j=1}^{d}L^2_{\mip_j}$, samples from this distribution can be generated componentwise. Finally, to obtain samples from the univariate distributions with Lebesgue densities $\leg^2_{n}$, $n\in\N$, we use a rejection sampling method with the arcsine proposal density $\arcsine_{1}$. By \cite[Theorem 1]{nevai1994generalized} the Legendre polynomials satisfy
 \begin{equation}
 \label{eq:rejectionbound}
|\leg_{n}(\psmi)|^2\leq 4e\arcsine_{1}(\psmi)\quad\forall \psmi\in[0,1]\;\forall n\in\N.
 \end{equation}

 Therefore, the theory of rejection sampling \cite[Chapter 4.5]{MR2151519} ensures that if we repeatedly generate $\psmi\sim \arcsine_{1}$ and $U\sim\text{Unif}(0,1)$ until $U\leq |\leg_{n}(\psmi)|^2/(4e\arcsine_{1}(\psmi))$ holds, then the resulting sample is exactly distributed according to $\leg_n^2$ and the required number of iterations until acceptance has a geometric distribution with mean $4e$.
The total expected computational work for the generation of $\NS$ samples from $\optimaldistribution_{\vsp}$ is thus $4e\NS d$, if we assume that the computation of $\leg^2_n(\psmi)$ is $O(1)$. In practice, a 3-term recurrence formula whose work is bounded by $3n$ can be used to compute $\leg_{n}(\psmi)$. This increases  the upper bound for the expected  work to $12e\NS \frac{1}{\dvsp}\sum_{\mip\in\mis}|\mip|_{1}$.

\Cref{eq:rejectionbound} holds more generally for probability measures on $[0,1]$ with Lebesgue densities of the form $\frac{\text{d}\measure}{\text{d}\lambda}=C(\alpha,\beta)\psmi^{\alpha}(1-\psmi)^{\beta}$, $\alpha,\beta\geq -1/2$  \cite[Theorem 1]{nevai1994generalized}. The bound on the associated orthogonal polynomials $(\leg^{\alpha,\beta}_{n})_{n\in\N}$, which are commonly called Jacobi polynomials, is
\begin{equation*}
|\leg^{\alpha,\beta}_n(\psmi)|^2\frac{\text{d}\measure}{\text{d}\lambda}\leq 2e(2+\sqrt{\alpha^2+\beta^2})\arcsine_{1}(\psmi)\quad\forall\psmi\in[0,1]\;\forall n\in\N.
\end{equation*}

Even more generally, the same inequality holds with a constant $C_{\measure}$ independent of $\psmi$ and $n$ for orthogonal polynomials with respect to a wide class of measures $\mu$ that are absolutely continuous with respect to the Lebesgue measure on $[0,1]$ \cite[Theorem 12.1.4]{MR0372517}.
When $C_{\measure}$ is unknown, however, rejection sampling cannot be applied. As a substitute, we could use MCMC sampling (which we also discuss below as an alternative method to sample directly from $\optimaldistribution$ in cases when no product structure of $\domPS$ or $\measure$ can be exploited). The error resulting from the fact that the resulting samples would not be distributed exactly according to $|\leg_n|^2$ can be controlled using \Cref{pro:stability} below.

For orthonormal polynomials $(\her_n)_{n\in\N}$ with respect to rapidly decaying measures supported on the whole real line, such as Gaussian measures, it is shown in  \cite{levin1992christoffel} that $|\her_n(\psmi)|^2\frac{\text{d}\measure}{\text{d}\lambda}$ is exponentially concentrated in an interval $[-a_n,a_n]$ with $C^{-1}n^{b}\leq a_n\leq Cn^b$ for some $b>0$ and $C>0$ depending on $\measure$, and  that for some $C_{\measure}$
\begin{equation*}
|\her_n(\psmi)|^2\frac{\text{d}\measure}{\text{d}\lambda}\leq C_{\measure}\frac{a_n}{4}|1-\frac{\psmi}{a_n}|^{-1/2}\quad\forall \psmi\in[-a_n,a_n]\;\forall n\in\N.
\end{equation*}

 Together with the stability result in \Cref{pro:stability} below, this shows that the previous results can be transfered to measures on the real line, if we simply ignore the mass outside $[-a_n,a_n]$ and apply rejection sampling or Markov chain methods with the proposal density $\frac{a_n}{4}|1-\frac{\psmi}{a_n}|^{-1/2}$.
 Alternatively, a different result in \cite{levin1992christoffel} shows that on $[-a_n,a_n]$ the density $|\her_n(\psmi)|^2\frac{\text{d}\measure}{\text{d}\lambda}$ is bounded by the uniform probability density up to a factor that grows sublinearly in the polynomial degree $n$.
 \\

For the remainder of this subsection, we assume that a polynomial space $\vsp$ is fixed and use the simplified notation  $\optimaldistribution=\optimaldistribution_{\vsp}$, $\optimaldensity=\optimaldensity_{\vsp}$,  $\optimalweight=\optimalweight_{\vsp}$.

We assume that we cannot use exact samples from the optimal distribution because, for example, $\domPS$ is not a hypercube or $\measure$ not a product measure. We therefore discuss the use of Markov chain methods for the generation of samples. Since the resulting samples are not exactly distributed according to the optimal distribution, we need the following stability result.

\begin{pro}[\textbf{Stability with respect to perturbations of the sampling density}]
	\label{pro:stability}
	All results in \Cref{thm:dpls} that are valid for the optimal choice $\optimaldistribution$ with $\frac{\mathrm{d}\optimaldistribution}{\mathrm{d}\measure}=\optimaldensity$ of the sampling distribution hold true if we instead use samples from $\tilde{\samplingmeasure}$ with $\frac{\mathrm{d}\tilde{\samplingmeasure}}{\mathrm{d}\measure}:=\tilde{\density}$ (but keep the weight function $\optimalweight=1/\optimaldensity$), provided that
	\begin{equation*}
	\norm{1-\tilde{\density}/\optimaldensity}{L^p_{\optimaldistribution}(\domPS)}\leq \frac{1}{6}m^{-1-1/p} \quad\text{for some }p\geq 1,
	\end{equation*}
	where $\dvsp=\dim \vsp$,
	and provided that we replace $\kappa$ by $(5/36-5/6\log (5/6))/(1+r)$. We note that the total variation distance $\norm{\cdot}{\text{TV}}$ satisfies
	\begin{equation*}
	\norm{\tilde{\samplingmeasure}-\optimaldistribution}{\text{TV}}:=\frac{1}{2}\int_{\domPS}|\tilde{\density}(\psmi)-\optimaldensity(\psmi)|\measure(\mathrm{d}\psmi)=\frac{1}{2}\norm{1-\tilde{\density}/\optimaldensity}{L^1_{\optimaldistribution}(\domPS)}.
	\end{equation*}
\end{pro}

\begin{proof}
	The proof of \Cref{thm:dpls} in \cite{cohen2016optimal} is based on large deviation bounds for the matrix $\mathbf{G}$ of \Cref{eq:dpls:computation}. In particular, it is based on the observation that $\mathbf{G}$ is a Monte Carlo average,
	\begin{equation*}
	\mathbf{G}=\frac{1}{\NS}\sum_{i=1}^{\NS}\mathbf{X}_i,
	\end{equation*}
	of independent and identically distributed matrices
	\begin{equation*}
	\mathbf{X}_i:=\left(\optimalweight(\psmi_i)\onb_j(\psmi_i)\onb_k(\psmi_i)\right)_{j,k\in\{1,\dots,\dvsp\}} \quad \text{with }\psmi_i\sim \optimaldistribution
	\end{equation*}
	that satisfy
	\begin{equation*}
	\expect \mathbf{X}_i=\mathbf{I}
	\end{equation*}
	by $L^2_{\measure}(\domPS)$-orthonormality of the basis polynomials $\onb_j$, $j\in\{1,\dots,\dvsp\}$.
	A Chernoff inequality for matrices then provides the bound on $\prob(\|{\mathbf{G}-\mathbf{I}}\|\leq 1/2)$ in part (i) of \Cref{thm:dpls} from which everything else follows. The crucial insight is that this inequality permits small perturbations of the expected value. Indeed, if we replace $\optimaldistribution$ by $\tilde{\samplingmeasure}$ in the definition of $\mathbf{X}_i$, then  \cite[Theorem 1.1]{MR2946459} yields the same bound on $\prob(\|{\mathbf{G}-\mathbf{I}}\|\leq 1/2)$, with the new value of $\kappa$, provided that $\|{\expect\mathbf{X}_i-\mathbf{I}}\|\leq 1/6$. To show that this last estimate holds, we use
	\begin{equation*}
	\begin{split}\|{\expect\mathbf{X}_i-\mathbf{I}}\|&\leq \dvsp\norm{\expect\mathbf{X}_i-\mathbf{I}}{\text{max}}\\
	&=\dvsp\max_{j,k\in \{1,\dots,\dvsp\}}|\int_{\domPS}\optimalweight(\psmi)\onb_j(\psmi)\onb_k(\psmi)(\tilde{\density}(\psmi)-\optimaldensity(\psmi))\measure(d\psmi)|\\
	&\leq \dvsp\max_{j,k\in \{1,\dots,\dvsp\}}\norm{\optimalweight \onb_j\onb_k}{L^q_{\optimaldensity\measure}}\norm{1-\tilde{\density}/\optimaldensity}{L^{p}_{\optimaldensity\measure}},
	\end{split}
	\end{equation*}
	where we used Hölder's inequality with $1/q=1-1/p$ in the last step.

	The claim now follows if we can prove that
	\begin{equation*}
	\norm{\optimalweight \onb_j\onb_k}{L^q_{\optimaldensity\measure}}\leq m^{1-1/q}\quad\forall q\in[1,\infty].
	\end{equation*}
	For this, we first consider the case when $q=1$, $p=\infty$, for which
	\begin{equation*}
	\begin{split}
	\norm{\optimalweight \onb_j\onb_k}{L^q_{\optimaldensity\measure}}&=\int_{\domPS}|\onb_j(\psmi)\onb_k(\psmi)|\measure(d\psmi)\\
	&\leq \norm{\onb_j}{L^2_{\measure}}\norm{\onb_k}{L^2_{\measure}}\\
	&\leq 1
	\end{split}
	\end{equation*}
	by the Cauchy-Schwarz inequality and $L^2_{\measure}(\domPS)$-orthonormality of the functions $\onb_j$.

	Next, we consider the case when $q=\infty$, $p=1$, for which
	\begin{equation*}
	\begin{split}
	\norm{\optimalweight \onb_j\onb_k}{L^q_{\optimaldensity\measure}}&=\sup_{\psmi\in \domPS}|\optimalweight(\psmi)\onb_j(\psmi)\onb_k(\psmi)|\\
	&\leq \frac{m}{2}\frac{\onb_j^2(\psmi)+\onb_k^2(\psmi)}{\sum_{i=1}^{\dvsp}\onb_i^2(\psmi)}\\
	&\leq m,
	\end{split}
	\end{equation*}
	where we used the elementary inequality $ab\leq \frac{1}{2}a^2+\frac{1}{2}b^2$ for the second step.

	Finally the claim for $1<q<\infty$ follows from Littlewood's interpolation inequality for $L^q$ norms.
\end{proof}

So far, we have assumed that $\domPS$ and $\measure$ exhibit product structure, which allowed us to generate samples coordinate-wise, exploiting known bounds on univariate orthogonal polynomials. 
For more general cases, we now briefly discuss Metropolized independent sampling, which is a simple MCMC algorithm, for the generation of samples from the optimal distribution $\optimaldistribution$. For an extensive treatment of the theory of MCMC algorithms we refer to \cite{liu2008monte}.

The general strategy of MCMC algorithms for the generation of samples from $\optimaldistribution$ is to construct a Markov chain for which $\optimaldistribution$ is an invariant distribution. Ergodic theory then shows that under some assumptions the location of this Markov chain after $n\gg1$ steps is approximately distributed according to $\optimaldistribution$.
Metropolis-Hastings algorithms are MCMC algorithms that construct Markov chains based on user-specified \emph{proposal densities} $p(\psmi,\cdot)$, $\psmi\in\domPS$ (with respect to $\measure$) and a rejection step to ensure convergence to the desired limit distribution $\optimaldistribution$. More specifically, the transition kernel of a Metropolis-Hastings algorithm has the form
\begin{equation}
\label{eq:MC}
K(\psmi,\mathrm{d}\psmi'):=\begin{cases}
p(\psmi,\psmi')\min\{1,\frac{\optimaldensity(\psmi')p(\psmi',\psmi)}{\optimaldensity(\psmi)p(\psmi,\psmi')}\}\measure(\mathrm{d}\psmi')\quad&\text{if }\psmi'\not = \psmi\\
1-\int_{z\not = \psmi}p(\psmi,z)\min\{1,\frac{\optimaldensity(z)p(z,\psmi)}{\optimaldensity(\psmi)p(\psmi,z)}\}\measure(\mathrm{d}z)\quad &\text{if }\psmi'=\psmi.
\end{cases}
\end{equation}
This kernel can be interpreted (and implemented) as proposing a transition from the current state $\psmi$ to a new state $\psmi'$ drawn from the density $p(\psmi,\cdot)$, and rejecting this transition with a certain probability determined by the values of $\optimaldensity$ and $p$ at the current state $\psmi$ and the proposed state $\psmi'$. The rejection probability is designed to ensure the detailed balance condition $\optimaldistribution(\mathrm{d}\psmi)K(\psmi,\mathrm{d}\psmi')=\optimaldistribution(\mathrm{d}\psmi')K(\psmi',\mathrm{d}\psmi)$, which in turn guarantees that $\optimaldistribution$ is invariant under $K$.

Metropolized independent sampling is the name of the subset of Metropolis-Hastings algorithms for which the proposal density $p$ is independent of the current state $\psmi$. If we denote the corresponding state-independent proposal density by $p(\psmi')$ 
and define $\gap:=\inf_{\psmi\in\domPS}p(\psmi)/\optimaldensity(\psmi)$, then it can be shown \cite[Section 3.2.2]{liu1996metropolized} that starting from any distribution $\genericmeasure$ we have the bound
\begin{equation}
\label{eq:TV}
\norm{K^{n}\genericmeasure-\optimaldistribution}{TV}\leq 2 (1-g)^n
\end{equation} for the total variation distance between the $n$th step probability distribution $K^n\genericmeasure$ of the Markov chain and the target distribution $\optimaldistribution$.
This means that if the proposal  density satisfies $\gap:=\inf_{\psmi\in\domPS}p(\psmi)/\optimaldensity(\psmi)>0$, then $n:= \gap^{-1}\log(24\dvsp^2)$ Markov chain steps suffice to ensure that
\begin{equation*}
\norm{K^n\genericmeasure-\optimaldistribution}{TV}\leq 2 (1-\gap)^{\gap^{-1}\log(24m^2)}\leq \frac{1}{12m^2},
\end{equation*}
as required by \Cref{pro:stability}. To generate $\NS>0$ independent samples from $K^{n}\genericmeasure$, we have to run $\NS$ independent copies of the Markov chain, which differs from the more common practice to use $\NS$ successive, thus dependent, steps of a single Markov chain.

\subsection{Sampling from the arcsine distribution}
\label{ssec:arcsine}
We now determine lower and upper bounds for the optimal sampling distributions associated with downward closed polynomial spaces on $[0,1]^d$ in terms of the arcsine distribution. Namely, we show in \Cref{pro:sampling} below, under quite general assumptions on the measure $\measure$, the bound
\begin{equation*}
	C^{-d}\leq \frac{\mathup{d}\optimaldistribution_{\vsp}}{\mathup{d}\lambda}(\psmi)\leq C^d\arcsine_{d}(\psmi),
\end{equation*}
where $\arcsine_d$ is the arcsine density from \Cref{eq:arcdef}.

The lower bound can be used to make the bound in \Cref{thm:dpls} more precise. Indeed, it implies that the weight $\optimalweight_{\vsp}$ appearing in $e_{\vsp,\optimalweight_{\vsp},\infty}$ satisfies
\begin{equation}
\label{eq:weightbound}
\optimalweight_{\vsp}(\psmi)\leq C^{d}\frac{\text{d}\measure}{\text{d}\lambda}(\psmi).
\end{equation}

 The upper bound provides an alternative sampling strategy:  Instead of sampling from the optimal distribution, we may simply sample from the arcsine distribution with Lebesgue density $\arcsine_{d}$ without using Acceptance/Rejection or Markov chain methods. Indeed, since using the arcsine distribution for sample generation corresponds to the choice $\density=\arcsine_{d}\frac{\text{d}\lambda}{\text{d}\measure}$ for the density in \Cref{sec:dpls}, the associated weight function is given by
 \begin{equation}
 \label{eq:arcsineweight}
 \weight:=(\arcsine_{d})^{-1}\frac{\text{d}\measure}{\text{d}\lambda}.
 \end{equation}
 Hence, the upper bound  shows that the constant $K_{\vsp,\weight}$ from \Cref{thm:dpls} satisfies
 \begin{equation}
 \label{eq:arcsineK}
 \begin{split} K_{\vsp,\weight}=&\norm{\weight(\psmi)\sum_{j=1}^{\dvsp}\onb_j(\psmi)^2}{L^\infty(\domPS)}\\
 =&\norm{\weight(\psmi)m\frac{\text{d}\optimaldistribution}{\text{d}\measure}(\psmi)}{L^\infty(\domPS)}\\
 \leq &C^d\dvsp,
 \end{split}
 \end{equation}
 which is larger than the optimal value, $\dvsp$, only by the factor $C^d$, which is independent of $\vsp$.
 The advantages are that exact and independent samples from the univariate arcsine distribution can be generated  efficiently as $(\sin(X)+1)/2$ for a uniform random variable $X$ on $[-\pi/2,\pi/2]$, that we can use samples from the same distribution for all polynomial subspaces, and that the weight in \Cref{eq:arcsineweight}, which enters the error estimate in \Cref{thm:dpls} through $e_{\vsp,\weight,\infty}$, is known explicitly, is independent of $\vsp$, and vanishes at the boundary if $\frac{\text{d}\measure}{\text{d}\lambda}$ is bounded.

\begin{definition}[\textbf{Doubling measures}]
	A measure $\measure$ on $[0,1]$ is called doubling measure if it is absolutely continuous with respect to the Lebesgue measure and if there exists $L>0$ such that
	\begin{equation*}
	\measure(2I)\leq L\measure(I)
	\end{equation*}
	for any interval $I\subset [0,1]$. Here, $2I$ is defined as the interval with the same midpoint as $I$ and twice the length, intersected with $[0,1]$. We call $L>0$ the doubling constant of $\measure$.
\end{definition}
\begin{pro}[\textbf{Bounds on the optimal distribution}]
	\label{pro:sampling}
	If $[0,1]^d$, $d\geq 1$ is equipped with the product $\measure=\bigotimes_{j=1}^{d}\measure_j$  of doubling measures $\measure_j$ on $[0,1]$, then the optimal sampling distribution $
	\optimaldistribution_{\vsp}$  associated with a finite-dimensional downward closed space $\vsp$ of polynomials on $[0,1]^d$ satisfies
	\begin{equation}
	\label{eq:samplingbounds}
	C^{-d}\leq \frac{\mathup{d}\optimaldistribution_{\vsp}}{\mathup{d}\lambda}(\psmi)\leq C^d\arcsine_{d}(\psmi)
	\end{equation}
	for some constant $C>0$ depending only on the maximal doubling constant of the measures $\measure_j$, $j\in\{1,\dots,d\}$.
\end{pro}
\begin{proof}
	 \Cref{eq:samplingbounds} was shown to hold for the univariate optimal sampling distributions $\optimaldistribution_{k}(\psmi)$ associated with univariate spaces of polynomials of degree less than or equal to $k\in\N$ in \cite[Equation 7.14]{mastroianni2000weighted}.
	We prove the case $d>1$ by induction.

	To reduce the complexity of notation, we assume without loss of generality that $\measure_1=\measure_j$ for all $j\in\{1,\dots,d\}$. This, together with the assumption that $\vsp$ is downward closed, implies that
	\begin{equation*}
	\vsp=\vspan_{\mip\in \mis\subset \N^d}\{\onb_{\mip}(\psmi):=\onb_{\mip_1}(\psmi_1)\cdots \onb_{\mip_{d}}(\psmi_d)\}
	\end{equation*} for $(\onb_{j})_{j\in\N}$ the univariate orthonormal polynomials associated with $\measure_{1}$ and a multi-index set $\mis\subset\N^d$. For $j\in\N$, we define the multi-index sets $\mis_{j}:=\{\mip\in\mis: \mip_{1}=j\}$ and the spaces
	\begin{equation*}
	\tilde{\vsp}_{j}:=\vspan_{\mip\in \mis_{j}\subset \N^d}\{\tilde{\onb}_{\mip}(\tilde{\psmi}):=\onb_{\mip_2}(\psmi_2)\cdots \onb_{\mip_{d}}(\psmi_{d})\}
	\end{equation*} of polynomials on $[0,1]^{d-1}$ with associated optimal distributions $\tilde{\samplingmeasure}_{j}$ on $[0,1]^{d-1}$. This allows us to write
	\begin{equation*}
	\begin{split}
	\frac{\text{d}\optimaldistribution_{\vsp}}{\text{d}\lambda}(\psmi)&=\optimaldensity_{\vsp}(\psmi)\frac{\text{d}\measure}{\text{d}\lambda}(\psmi)\\
	&=\frac{1}{|\mis|}\sum_{\mip\in\mis}\onb_{\mip}^2(\psmi)\frac{\text{d}\measure}{\text{d}\lambda}(\psmi)\\
	&=\frac{1}{|\mis|}\sum_{j\in\N}\onb_{j}^2(\psmi_1)\frac{\text{d}\measure_{1}}{\text{d}\lambda}(\psmi_1)\sum_{\mip\in\mis_{j}}\tilde{\onb}_{\mip}^2(\tilde{\psmi})\prod_{j=2}^{d}\frac{\mathup{d}\measure_{1}}{\mathup{d}\lambda}(\psmi_j)\\
	&=\frac{1}{|\mis|}\sum_{j\in\N}\onb_{j}^2(\psmi_1)\frac{\text{d}\measure_{1}}{\text{d}\lambda}(\psmi_{1})|\mis_{j}|\frac{\text{d}\tilde{\samplingmeasure}_{j}}{\text{d}\lambda}(\tilde{\psmi}),
	\end{split}
	\end{equation*}
	which by the induction hypothesis for the case $d-1$ entails
	\begin{equation}
	\label{eq:temp}
	C^{-(d-1)}A(\psmi_1)\leq \frac{\text{d}\optimaldistribution_{\vsp}}{\text{d}\lambda}(\psmi)\leq A(\psmi_1)C^{d-1}\arcsine_{d-1}(\tilde{\psmi})
	\end{equation}
	with
	\begin{equation*}
	A(\psmi_1):=\frac{1}{|\mis|}\sum_{j\in\N}\onb_j^2(\psmi_1)\frac{\text{d}\measure_{1}}{\text{d}\lambda}(\psmi_{1})|\mis_{j}|.
	\end{equation*}
	We now use the fact that $A(\psmi_1)$ can be written as a weighted average of the univariate densities $\frac{\mathup{d}\optimaldistribution_{k}}{\mathup{d}\lambda}(\psmi_1)=\frac{1}{k}\sum_{j=1}^{k}\onb_j^2(\psmi_1)\frac{\text{d}\measure_{1}}{\text{d}\lambda}(\psmi_{1})$:
	\begin{equation*}
	A(\psmi_1)=\frac{1}{\sum_{k=1}^{\infty}{w_k}}\sum_{k=1}^{\infty}w_k\frac{\text{d}\optimaldistribution_{w_k}}{\text{d}\lambda}(\psmi_1)
	\end{equation*}
	with $w_k:=|\{j:|\mis_{j}|\geq k\}|$.
	Together with the case $d=1$ this shows
	\begin{equation*}
	C^{-1}\leq A(\psmi_1)\leq C \arcsine_{1}(\psmi_1),
	\end{equation*}
	which, when inserted into \Cref{eq:temp}, yields
	\begin{equation*}
	C^{-d}=C^{-(d-1)}C^{-1}\leq \frac{\text{d}\optimaldistribution_{\vsp}}{\text{d}\lambda}(\psmi)\leq C^{d-1}C\arcsine_{1}(\psmi_1)\arcsine_{d-1}(\tilde{\psmi})= C^{d}\arcsine_{d}(\psmi).
	\end{equation*}
\end{proof}

%% file: content/nonadaptive.tex
\section{Multilevel weighted least squares approximation}
\label{sec:nonadaptive}
In this section, we define a multilevel weighted polynomial least squares method and establish convergence rates for the approximation of a function $\rs_{\infty}\colon\domPS\subset\R^d\to\R$, $d\in\N\cup\{\infty\}$ in a normed vector space $(\F,\|\cdot\|_{\F})\hookrightarrow (L^2_{\measure}(\domPS),\norm{\cdot}{L^2_{\measure}(\domPS)})$ of continuous functions on $\domPS$, under the following assumptions.
\begin{itemize}
	\item{\textbf{A1:}} (Convergence of approximations) There exist functions $\rs_\n\in \F$, $\n\geq 1$ such that 
	\begin{equation*}
	\norm{\rs_{\infty}-\rs_\n}{\F}\leq C_0\n^{-\sc}
	\end{equation*} 
 
	\begin{equation*}
	\norm{\rs_{\infty}-\rs_{\n}}{L^2_{\measure}(\domPS)}\leq C_0 \n^{-\wc}
	\end{equation*}
	for some $C_0>0$, $\sc>0$ and $\wc\geq \sc$.
	
\item{\textbf{A2(p):}} (Polynomial approximability) There exist  downward closed spaces of polynomials $\vsp_{\dvsp}$, $\dvsp\geq 1$ on $\domPS$ such that
\begin{equation*}
\dim \vsp_{\dvsp}\leq C_1 \dvsp^{\dimp},
\end{equation*}
	\begin{equation*}
	e_{\dvsp,p}(\F):=\sup_{\rs\in \F} \frac{e_{\vsp_{\dvsp},p}(\rs)}{\norm{\rs}{\F}}\leq C_1\dvsp^{-\alpha}
	\end{equation*}
	for some $\dimp>0, \alpha>0$, $C_1>0$ and $p=2$ or $p=\infty$. (In the latter case, we use the shorthand $e_{\vsp_{\dvsp},\infty}(\rs):=e_{\vsp_{\dvsp},\optimalweight_{\dvsp},\infty}(\rs)$, where $\optimalweight_{\dvsp}$ is the optimal weight associated with $\vsp_{\dvsp}$.)

\item{\textbf{A3:}} (Sample work) The work required for a single evaluation of $\rs_\n$ satisfies $\work(\rs_n)\leq C_2\n^{\gamma}$ for some $\gamma>0$, $C_2>0$. 

\end{itemize}
\begin{rem}
		In Assumption A2(p), we have introduced the exponent $\dimp$, which in contrast to previous sections may be different from $1$, to be able to apply our results with common sequences of polynomial subspaces without the need for reparametrization. 
\end{rem}

\begin{exa}[\textbf{Polynomial approximability}]
	\label{exa:polapp}
	\begin{itemize}\item 
For univariate Sobolev spaces $\F=H^{\alpha}(\domPS)$, $\domPS=(0,1)$ with $\alpha>0$, Theorem 1 in \cite{quarteroni1984some} shows that
\begin{equation*}
e_{\dvsp,2}(H^{\alpha}(\domPS))\leq C \dvsp^{-\alpha}
\end{equation*}
for the space $V_{m}$ of univariate polynomials with degree less than $m$ and for $\measure$ the Lebesgue measure.
Analogous results also hold in higher dimensions. Here, optimal sequences of polynomial approximation spaces depend on the available smoothness. In particular, optimal polynomial approximation spaces for functions in Sobolev spaces $H^\alpha(\domPS)$ with $\domPS\subset\R^d$ and $\alpha>0$ are of total degree type, whereas functions in Sobolev spaces $H^{\alpha}_{\mix}(\domPS)$ of dominating mixed smoothness can be optimally approximated by hyperbolic cross polynomial spaces \cite{DuTeUl2015}.

	Similar results for the best approximation in the supremum norm hold for functions in Hölder spaces $\F=C^{s,t}(\domPS)$, $s\in\N$, $t\in[0,1]$ \cite[Theorem 2]{BagbyBosLevenberg2002} (and their dominating mixed smoothness analogues). 

\item 
Alternatively, we may simply define the space $\F$ via polynomial approximability of its elements. Assume that we have a sequence $(\vsp_\dvsp)_{\dvsp=1}^\infty$ of downward closed polynomial spaces on $\domPS\subset\R^\dps$ with $\dps\in\N\cup\{\infty\}$. If for some $\alpha>0$ we define
		\begin{equation*}
		\F:=\left\{ \rs\colon \domPS\to\R: \|\rs\|_{\F}:=\sup_{\dvsp\in\N} e_{\vsp_{\dvsp},p}(\rs)\dvsp^{\alpha}<\infty\right\}
		\end{equation*}
		with the auxiliary definition $\vsp_0:=\{0\}$, then it is easy to show that $\|\cdot\|_{\F}$ is a norm of $\F$ and that Assumption 2(p) holds with the given $\alpha$ and $C_1=1$.  The choice of the sequence of subspaces $\vsp_m$ can be based on truncating a orthogonal decomposition of $L^2_\measure(\Gamma)$ such as to include only basis functions whose contribution is above a given threshold in $\vsp_m$.  For more information on this construction, see \Cref{sec:adaptive} and \cite{Haji-AliNobileTamelliniEtAl2015a,devore1998nonlinear}.
\end{itemize}
\end{exa}

We now define the multilevel least squares method for a fixed number of levels $L\in\N$. We introduce subsequences
\begin{equation}
\label{eq:subseq}
\dvsp_k:=M\exp(k/({\dimp+\alpha})), \quad k\in\{0,\dots,L\}
\end{equation} and 
\begin{equation*}
\n_l:=\exp(l/(\gamma+\sc)), \quad l\in\{0,\dots,L\}
\end{equation*} with $M:=\exp(L\delta)$, $\delta:=\frac{\wc-\sc}{\alpha(\gamma+\sc)}\geq 0$ if $\gamma/\sc>\dimp/\alpha$ and $M:=1$ else. For our analysis we assume that $\dvsp$ and $\n$ can take non-integer values; in practice, rounding up to the nearest integer increases the required work only by a constant factor. By abuse of notation, we keep the simple notation $\vsp_k$, $e_{k,p}$,  and $\rs_{l}$ for the quantities $\vsp_{m_k}$, $e_{\dvsp_k,p}$, and $\rs_{n_l}$, respectively. 

Next, we draw independent, identically distributed, random samples
\begin{equation*}
\domPS_k=\{\psmi_{k,1},\dots,\psmi_{k,|\domPS_k|}\}\subset \domPS, \quad k\in \{0,\dots,L\}
\end{equation*}
with $\psmi_{k,j}\sim \optimaldistribution_{k}$, where $\optimaldistribution_{k}:=\optimaldistribution_{\vsp_k}$ is the optimal sampling distribution of $\vsp_{k}$ from \Cref{eq:optimaldistribution}. To ensure accuracy of our approximations, we couple the numbers of samples to the dimensions of the polynomial spaces via
\begin{equation}
\label{eq:size}
m_k^{\dimp}\leq \kappa \frac{|\domPS_k|}{\log |\domPS_k|}\leq 2 m_k^{\dimp}\;\;\forall k\in\{0,\dots,L\},\;\text{where } \kappa:=\frac{1-\log 2}{2+2L}.
\end{equation}
 By \Cref{eq:optimalK}, this guarantees that the assumption of \Cref{thm:dpls} is satisfied with $r=L$. Alternatively, we may replace $\kappa$ by $C^{-d}\kappa$ with $C$ from \Cref{pro:sampling} if $\domPS$ and $\measure\ll \lambda$ are products and if we use the arcsine distribution to generate samples, or we may choose $\kappa$ as in \Cref{pro:stability} if we use samples that are only approximately distributed according to the optimal distribution. 

Finally, we denote by $\P_{k}\colon F\to \vsp_k$ the random weighted least squares approximation using evaluations in $\domPS_k$, $k\in\{0,\dots,L\}$ and define the multilevel method
\begin{equation}
\label{eq:mldef}
\begin{split}
\smol_{L}(\rs_{\infty})&:=\P_{L}\rs_{0}+\sum_{l=1}^{L}\P_{L-l}(\rs_{l}-\rs_{l-1})\\&=\sum_{l=0}^{L}\P_{L-l}(\rs_{l}-\rs_{l-1})
\end{split}
\end{equation}
where we used the auxiliary definition $\rs_{-1}:=0$. 

For the sake of readability, we do not keep track of constants  and denote by $\lesssim$ any inequality that holds up to a factor depending only on $C_0,C_1,C_2,\alpha,\sc,\wc,\gamma$.
The exact values of these constants may be determined from the proof of the next Theorem.
\begin{thm}[\textbf{Convergence in probability}]
	\label{thm:main}	
Denote by 
\begin{equation}
\label{eq:workdef}
\work(\smol_{L}(\rs_{\infty})):=|\Gamma_L|\work(f_0)+\sum_{l=1}^{L}|\Gamma_{L-l}|\left(\work(f_l)+\work(f_{l-1})\right)
\end{equation}
the work that $\smol_{L}(\rs_{\infty})$ requires for evaluations of the functions $\rs_{l}$, $l\in\{0,\dots,L\}$.
Define
\begin{align*}
\rate&:=\begin{cases}
\dimp/\alpha& \text{if } \gamma/\sc\leq \dimp/\alpha\\
\theta \gamma/\sc+(1-\theta)\dimp/\alpha \text{ with }\theta:=\sc/\wc& \text{if }\gamma/\sc>\dimp/\alpha
\end{cases}\\
\text{and}\\
\lrate&:=\begin{cases}
2& \text{if } \gamma/\sc<\dimp/\alpha\\
3+\dimp/\alpha&\text{if }\gamma/\sc=\dimp/\alpha\\
1& \text{if }\gamma/\sc>\dimp/\alpha \text{ and }\wc=\sc\\
2&\text{if }\gamma/\sc>\dimp/\alpha \text{ and }\wc>\sc
\end{cases}.
\end{align*}
	
Let $0<\epsilon\lesssim 1$. If Assumptions A1, A2($\infty$), and A3 hold, then we may choose $L\in\N$ such that 
	\begin{equation*}
	\begin{split}
	\work(\smol_{L}(\rs_{\infty}))\lesssim \epsilon^{-\rate}|\log \epsilon|^{\lrate}\log|\log \epsilon|,
	\end{split}
	\end{equation*}

	and such that in an event $E$ with $\prob(E^{c})\lesssim \epsilon^{\log|\log \epsilon|}$ the multilevel approximation satisfies
	\begin{equation}
	\norm{\rs_{\infty}-\smol_{L}(\rs_{\infty})}{L^2_{\measure}(\domPS)}\leq \epsilon.
	\end{equation}
	
\end{thm}

\begin{proof}
	The strategy of this proof is to establish bounds on $\work(\smol_{L}(\rs_{\infty}))$ and $\norm{\rs_{\infty}-\smol_{L}(\rs_{\infty})}{L^2_{\measure}(\domPS)}$ for arbitrary $L\in\N$ first, and then to show that, for the right choice of $L$, the latter is smaller than $\epsilon$ and the former is bounded by $\epsilon^{-\rate}|\log\epsilon|^{\lrate}\log|\log \epsilon|$.
	
		\textbf{Work bounds.}  We may deduce immediately from \Cref{eq:size} the rough upper bound
		\begin{equation*}
		\sqrt{|\domPS_k|}\leq\frac{|\domPS_k|}{\log|\domPS_k|}\leq \frac{2}{\kappa}M^{\dimp}\exp(k\frac{\dimp}{\dimp+\alpha})\lesssim (L+1)M^{\dimp}\exp(k\frac{\dimp}{\dimp+\alpha})
		\end{equation*}
		on the number of samples at level $k\in\{0,\dots,L\}$. 
		Using \Cref{eq:size} again and inserting the previous estimate, we obtain the finer estimate
		\begin{equation*}
		\begin{split}
		|\domPS_k|&\leq (L+1)M^{\dimp}\exp(k\frac{\dimp}{\dimp+\alpha})\log {|\domPS_k|}\\
		&\lesssim (L+1)M^{\dimp}(\log (L+1)+\log M^{\dimp})\exp(k\frac{\dimp}{\dimp+\alpha})(k+1).
		\end{split}
		\end{equation*}
	Since 
	$$
	\work(\rs_{l})+\work(\rs_{l-1})\lesssim \exp(l\frac{\gamma}{\gamma+\sc})
	$$ by Assumption {A3}, we may conclude that
	
	\begin{equation}
	\label{eq:workbounds}
	\begin{split}
	\work(\smol_{L}(\rs_{\infty}))
	&\lesssim  (L+1)M^{\dimp}(\log (L+1)+\log M^{\dimp})\sum_{l=0}^{L}\exp((L-l)\frac{\dimp}{\dimp+\alpha})(L-l+1)\exp(l\frac{\gamma}{\gamma+\sc})\\
	&=(L+1)M^{\dimp}(\log (L+1)+\log M^{\dimp})\exp(L\frac{\dimp}{\dimp+\alpha})\sum_{l=0}^{L}\exp\Big(-l\big(\frac{\dimp}{\dimp+\alpha}-\frac{\gamma}{\gamma+\sc}\big)\Big)(L-l+1).
	\end{split}
	\end{equation}
	
	We now distinguish three cases.
	\begin{enumerate}[(a)]
		\item $\gamma/\sc<\dimp/\alpha$: In this case $\dimp/(\dimp+\alpha)>\gamma/(\gamma+\sc)$. Thus, the sum on the right-hand side of \Cref{eq:workbounds} satisfies
		\begin{equation*}
		\begin{split}
		\sum_{l=0}^{L}\exp\Big(-l\big(\frac{\dimp}{\dimp+\alpha}-\frac{\gamma}{\gamma+\sc}\big)\Big)(L-l+1)&\lesssim (L+1)\sum_{l=0}^{L}\exp\Big(-l\big(\frac{\dimp}{\dimp+\alpha}-\frac{\gamma}{\gamma+\sc}\big)\Big)\\
		&\lesssim L+1.
		\end{split}
		\end{equation*}
		Together with the fact that $M=1$ in the case under consideration, this shows that
		\begin{equation*}
		\work(\smol_{L}(\rs_{\infty}))\lesssim \exp(L\frac{\dimp}{\dimp+\alpha})(L+1)^2\log (L+1).
		\end{equation*}
		
		\item $\gamma/\sc=\dimp/\alpha$:  In this case $\dimp/(\dimp+\alpha)=\gamma/(\gamma+\sc)$. Thus, the sum on the right-hand side of \Cref{eq:workbounds} equals $\sum_{l=0}^{L}(L-l+1)\lesssim (L+1)^2$ and we obtain
		\begin{equation*}
		\work(\smol_{L}(\rs_{\infty}))\lesssim \exp(L\frac{\dimp}{\dimp+\alpha})(L+1)^3\log (L+1).
		\end{equation*}
		since $M=1$.
		
		\item $\gamma/\sc>\dimp/\alpha$:  In this case $\dimp/(\dimp+\alpha)<\gamma/(\gamma+\sc)$. Thus, the sum on the right-hand side of \Cref{eq:workbounds} satisfies 
		\begin{equation*}
		\begin{split}
		\sum_{l=0}^{L}&\exp\Big(-l\big(\frac{\dimp}{\dimp+\alpha}-\frac{\gamma}{\gamma+\sc}\big)\Big)(L-l+1)\\
		&=\exp\Big(L\big(\frac{\gamma}{\gamma+\sc}-\frac{\dimp}{\dimp+\alpha}\big)\Big)\sum_{l=0}^{L}\exp\Big(-l\big(\frac{\gamma}{\gamma+\sc}-\frac{\dimp}{\dimp+\alpha}\big)\Big)(l+1)\\
		&\lesssim\exp\Big(L\big(\frac{\gamma}{\gamma+\sc}-\frac{\dimp}{\dimp+\alpha}\big)\Big).
		\end{split}
		\end{equation*}
		If $\wc=\sc$, then $M=1$ and we obtain
		\begin{equation*}
		\begin{split}
		\work(\smol_{L}(\rs_{\infty}))&\lesssim (L+1)M^{\dimp}(\log(L+1)+\log M^{\dimp})\exp(L\frac{\gamma}{\gamma+\sc}))\\
		&\lesssim \exp\Big(L\big(\frac{\gamma}{\gamma+\sc}\big)\Big) (L+1)\log(L+1).
		\end{split}
		\end{equation*} 
	
	 If instead $\wc>\sc$, then $M=\exp(\delta L)$ and we obtain
		\begin{equation*}
		\begin{split}
		\work(\smol_{L}(\rs_{\infty}))&\lesssim (L+1)M^{\dimp}(\log(L+1)+\log M^{\dimp})\exp(L\frac{\gamma}{\gamma+\sc}))\\
		&\lesssim \exp\Big(L\big(\frac{\gamma}{\gamma+\sc}+{\dimp}\delta\big)\Big)(L+1)^2\log(L+1).
				\end{split}
		\end{equation*}
	\end{enumerate}
\textbf{Residual bounds.} First, we show that with high probability 
	\begin{align}
	\label{eq:conv1}
	\|\Id-\P_k\|_{\F\to L^2_{\measure}(\domPS)}&\lesssim M^{-\alpha}\exp(-k\alpha /(\dimp+\alpha))\quad \forall k\in\{0,\dots,L\}.
	\end{align} 
By part (ii) of \Cref{thm:dpls} together with Assumption A2($\infty$), it suffices to show that the event
\begin{equation*}
E:=\{\|{\mathbf{G}_k-\mathbf{I}_k}\|\leq 1/2 \;\forall k\in\N\}
\end{equation*}
has a high probability, where $\mathbf{G}_k$ is the Gramian matrix from \Cref{eq:dpls:computation}. But by the first part of the same theorem, the complementary probability that $\|{\mathbf{G}_k-\mathbf{I}_k}\|\leq 1/2$ for a fixed $k\in\N$ decays as the number of samples $|\domPS_{k}|$ increases. Since the sets $\domPS_{k}$ grow exponentially in $k$, by \Cref{eq:size},  we may conclude using a crude zeroth moment estimate and a geometric series bound:
		\begin{equation}
		\label{eq:probbound}
		\begin{split}
		\prob(E^c)&=\prob\left(\exists k\in\N:\|{\mathbf{G}_k-\mathbf{I}_k}\|>1/2 \right)\\
		&\leq \sum_{k=0}^{\infty}\prob(\|{\mathbf{G}_k-\mathbf{I}_k}\|>1/2)\\
		&\leq 2\sum_{k=0}^\infty |\domPS_{k}|^{-L}\\
		&\leq 2\kappa^{L}M^{-\dimp L}\sum_{k=0}^{\infty}\exp(-kL\frac{\dimp}{\dimp+\alpha})\\
		&=\frac{2\kappa^{L}M^{-\dimp L}}{1-\exp(-L\frac{\dimp}{\dimp+\alpha})}\\
		&\lesssim L^{-L}.
		\end{split}
		\end{equation}
		
		Assuming now that the samples $\domPS_{k}$, $k\in\N$ are such that \Cref{eq:conv1} holds for the associated operators $\P_k$, we obtain
		\begin{equation}
		\label{eq:convbounds}
		\begin{split}
		\norm{\rs_{\infty}-\smol_{L}(\rs_{\infty})}{L^2_{\measure}(\domPS)}
		&=\norm{\rs_{\infty}-\left(\sum_{l=0}^{L}(\rs_{l}-\rs_{l-1})-\sum_{l=0}^{L}(\Id-\Pi_{L-l})(\rs_{l}-\rs_{l-1})\right)}{L^2_{\measure}(\domPS)}\\
		&\leq \norm{\rs_{\infty}-\rs_{L}}{L^2_{\measure}(\domPS)}+\sum_{l=0}^{L}\norm{\Id-\Pi_{L-l}}{\F\to L^2_{\measure}(\domPS)}\norm{\rs_{l}-\rs_{l-1}}{\F}\\
		&\lesssim \exp(-L\frac{\wc}{\gamma+\sc})+M^{-\alpha}\sum_{l=0}^{L}\exp(-(L-l)\frac{\alpha}{\dimp+\alpha})\exp(-l\frac{\sc}{\gamma+\sc})\\
		&= \exp(-L\frac{\wc}{\gamma+\sc})+M^{-\alpha}\exp(-L\frac{\alpha}{\dimp+\alpha})\sum_{l=0}^{L}\exp\Big(l\big(\frac{\alpha}{\dimp+\alpha}-\frac{\sc}{\gamma+\sc}\big)\Big),\\
		\end{split}
		\end{equation}
		where we used Assumption A1.
			Again, we distinguish the cases (a)-(c).
			\begin{enumerate}[(a)]
				\item $\gamma/\sc<\dimp/\alpha$:  In this case $\alpha/(\dimp+\alpha)<\sc/(\gamma+\sc)$. Thus, the sum on the right-hand side of \Cref{eq:convbounds} is uniformly bounded in $L$ and we obtain 
				\begin{equation*}
			\begin{split}
			\norm{\rs_{\infty}-\smol_{L}(\rs_{\infty})}{L^2(\measure)}&\lesssim \exp(-L\frac{\wc}{\gamma+\sc})+\exp(-L\frac{\alpha}{\dimp+\alpha})\\
			&\lesssim \exp(-L\frac{\alpha}{\dimp+\alpha}),
			\end{split}
				\end{equation*}
				where we used the fact that $\wc\geq \sc$ for the last inequality.
				
				\item $\gamma/\sc=\dimp/\alpha$: In this case $\alpha/(\dimp+\alpha)=\sc/(\gamma+\sc)$. Thus, the sum on the right-hand side of \Cref{eq:workbounds} equals $L+1$ and we obtain
				\begin{equation*}
				\begin{split}
				\norm{\rs_{\infty}-\smol_{L}(\rs_{\infty})}{L^2(\measure)}&\lesssim \exp(-L\frac{\wc}{\gamma+\sc})+\exp(-L\frac{\alpha}{\dimp+\alpha})(L+1)\\
				&\lesssim \exp(-L\frac{\alpha}{\dimp+\alpha})(L+1),
				\end{split}
				\end{equation*}
				where we used the fact that $\wc\geq \sc$ for the last inequality.
				
				\item $\gamma/\sc>\dimp/\alpha$: In this case $\alpha/(\dimp+\alpha)>\sc/(\gamma+\sc)$. Thus, the sum on the right-hand side of \Cref{eq:workbounds} is a divergent geometric series and we obtain 
				\begin{equation*}
				\begin{split}
				\norm{\rs_{\infty}-\smol_{L}(\rs_{\infty})}{L^2(\measure)}&\lesssim \exp(-L\frac{\wc}{\gamma+\sc})+M^{-\alpha}\exp(-L\frac{\sc}{\gamma+\sc})\\
				&\lesssim \exp(-L\frac{\wc}{\gamma+\sc}),
				\end{split}
				\end{equation*}
				where we used the definition of $M=\exp(L\delta)$ and $\delta$ in the case $\gamma/\sc>\dimp/\wc$ in the last inequality.
			\end{enumerate}
			
			\textbf{Conclusion.} It remains to choose $L$ such that the residual bound equals $\epsilon$ and insert this choice of $L$ into the work bound. For simplicity, we assume $L$ can be any real number. In practice, rounding up to the next largest value decreases the residual and increases the work only by a constant factor. One final time, we distinguish the cases (a)-(c).
			
		\begin{enumerate}[(a)]
			\item $\gamma/\sc<\dimp/\alpha$: Defining $L$ as the solution of
			
			\begin{equation*}
			\exp(-L\frac{\alpha}{\dimp+\alpha})=\epsilon,
			\end{equation*}
			
			we obtain the second inequality in the following estimate:
			\begin{equation*}
			\begin{split}
		\work(\smol_{L}(\rs_{\infty}))&\lesssim \exp(L\frac{\dimp}{\dimp+\alpha})(L+1)^2\log (L+1)\\
		&\lesssim	\epsilon^{-\rate}|\log\epsilon|^2\log|\log \epsilon|.
			\end{split}
			\end{equation*}
			
			\item $\gamma/\sc=\dimp/\alpha$: 			Since we assumed that $\epsilon\lesssim 1$ there is a unique positive solution of
			\begin{equation*}
			\exp(-L\frac{\alpha}{\dimp+\alpha})(L+1)=\epsilon.
			\end{equation*}
			
			With this choice of $L$ we obtain the second inequality in the following estimate:
			\begin{equation*}
			\begin{split}
			\work(\smol_{L}(\rs_{\infty}))&\lesssim  \exp(L\frac{\dimp}{\dimp+\alpha})(L+1)^3\log(L+1)\\
			&\lesssim \epsilon^{-\rate}|\log\epsilon|^{3+\lambda}\log|\log \epsilon|.
			\end{split}
			\end{equation*}
			
			\item $\gamma/\sc>\dimp/\alpha$: 
		We assume $\wc>\sc$, the case $\wc=\sc$ can be treated analogously. Defining $L$ as the solution of
			\begin{equation*}
			\exp(-L\frac{\wc}{\gamma+\sc})=\epsilon,
			\end{equation*}	
			we obtain the second inequality in the following estimate:
			\begin{equation*}
			\begin{split}
			\work(\smol_{L}(\rs_{\infty}))&\lesssim \exp\Big(L\big(\frac{\gamma}{\gamma+\sc}+\dimp\delta\big)\Big)(L+1)^2\log(L+1)\\
			&\lesssim \epsilon^{-\lambda}|\log\epsilon|^2\log|\log\epsilon|.
			\end{split}
			\end{equation*}
		\end{enumerate}
		In all cases we chose $L$ such that $L\geq |\log\epsilon|$, thus $\prob(E^c)\lesssim L^{-L}\lesssim \epsilon^{\log|\log\epsilon|}$ by \Cref{eq:probbound}.
\end{proof}
\begin{rem}
	The proof does not exploit independence of samples across different $\domPS_k$, $k\in\{0,\dots,L\}$, but instead relies on a simple union bound (see \Cref{eq:probbound}). Thus, we could alternatively first create $\domPS_{L}$ and then define all $\domPS_l$ with $l<L$ as subsets of it.
\end{rem}
\begin{rem}
	To determine the polynomial coefficients of $\Pi_{L-l}(\rs_{l}-\rs_{l-1})$, $l\in\{0,\dots,L\}$, after the functions $\rs_{l}-\rs_{l-1}$ have been evaluated in all $\psmi\in \domPS_{L-l}$, we need to solve linear systems of the form
	\begin{equation}
	\mathbf{G}_k\mathbf{\p}_k=\mathbf{\coeff}_k,\quad k\in\{0,\dots,L\}
	\end{equation}
		as in \Cref{eq:dpls:computation}. In the event $E$ in which the residual estimate of the previous theorem holds, the condition numbers of all matrices $\mathbf{G}_k$ are bounded by 3. Therefore, using a suitable iterative solver  we can determine all coefficients $\mathbf{\p}_k$ to an accuracy of $\epsilon>0$ with $\mathcal{O}(\log\epsilon)$ iterations. Since matrix vector products with $\mathbf{G}_k$ require $\mathcal{O}(m_k^{2\dimp})$ operations by \Cref{rem:matvec}, the associated computational work is, up to logarithmic factors, given by
		\begin{equation*}
		\sum_{k=0}^{L}m_k^{2\dimp}=\sum_{k=0}^{L}M^{2\dimp}\exp(2k\dimp/(\dimp+\alpha))\lesssim M^{2\dimp}\exp(2L\dimp/(\dimp+\alpha)).
		\end{equation*}
	Inspection of the proof of the previous theorem shows that, even if we include this cost in the work specification, the conclusion holds true with slightly different logarithmic factors and the exponent
		\begin{align*}
		\tilde{\rate}&:=\begin{cases}
		2\dimp/\alpha& \text{if } \gamma/\sc\leq 2\dimp/\alpha\\
		\gamma/\sc& \text{if }\gamma/\sc>2\dimp/\alpha,
		\end{cases}
		\end{align*}
		instead of $\rate$ (assuming for simplicity that $\sc=\wc$), provided that we change the definition of the subsequence $m_k$ in \Cref{eq:subseq} to 
		$$
		m_k:=\exp(k/(2\dimp+\alpha)).
		$$
	\end{rem}

To obtain mean square convergence, we replace the least squares approximations $\P_{k}$ by the stabilized versions $\P_k^{\cond}$ from part (iii) of \Cref{thm:dpls}, and define
\begin{equation}
\label{eq:ml2}
\smol^{\cond}_{L}(\rs_{\infty}):=
\P^{\cond}_{L}\rs_{0}+\sum_{l=1}^{L}\P^{\cond}_{L-l}(\rs_{l}-\rs_{l-1}).
\end{equation}
\begin{thm}[\textbf{Mean square convergence}]
	\label{thm:main2}
	Let $0<\epsilon\lesssim 1$. If Assumptions A1, A2(2), and A3 hold, then we may choose $L\in\N$ such that  
		\begin{equation}
		\expect\norm{\rs_{\infty}-\smol^{\cond}_{L}(\rs_{\infty})}{L^2_{\measure}(\domPS)}^2\leq \epsilon^2
		\end{equation}
		and
	\begin{equation*}
	\begin{split}
	\work(\smol_{L}^{\cond}(\rs_{\infty}))\lesssim \epsilon^{-\rate}|\log \epsilon|^{\lrate}\log|\log\epsilon|,
	\end{split}
	\end{equation*}
	with $\rate$ and $\lrate$ as in \Cref{thm:main}.
\end{thm}
\begin{proof}
	The work bounds from the proof of \Cref{thm:main} hold unchanged.
	
We next establish residual bounds for arbitrary $L\in\N$ as before, using the error representation
	\begin{equation*}
	\rs_{\infty}-\smol^{\cond}_{L}(\rs_{\infty})=\rs_{\infty}-\rs_{L}+\sum_{l=0}^{L}(\Id-\P^{\cond}_{L-l})(\rs_{l}-\rs_{l-1}).
	\end{equation*}
	The triangle inequality of the norm $(\expect\norm{\cdot}{L^2_{\measure}(\domPS)}^2)^{1/2}$ implies that
	\begin{equation*}
	\begin{split}
	\left(\expect\norm{\rs_{\infty}-\smol^{\cond}_{L}(\rs_{\infty})}{L^2_{\measure}(\domPS)}^2\right)^{1/2}&\leq \left(\norm{\rs_{\infty}-\rs_{L}}{L^2_{\measure}(\domPS)}^2\right)^{1/2}+\sum_{l=0}^{L}\left(\expect\norm{(\Id-\P^{\cond}_{L-l})(\rs_{l}-\rs_{l-1})}{L^2_{\measure}(\domPS)}^2\right)^{1/2}\\
		&\lesssim \norm{\rs_{\infty}-\rs_{L}}{L^2(\measure)}+\sum_{l=0}^{L}\left(e_{\vsp_{L-l},2}^2(\rs_{l}-\rs_{l-1})+\norm{\rs_{l}-\rs_{l-1}}{L^2_{\measure}(\domPS)}^2|\domPS_{L-l}|^{-2\alpha/\dimp}\right)^{1/2}=:(\star)
	\end{split}
	\end{equation*}
	where we used part (iii) of \Cref{thm:dpls} together with the fact that $L\geq 2\alpha/\dimp$ for small enough $\epsilon$ for the second inequality. We observe that
	\begin{itemize}
		\item by Assumption {A1}, we have
		\begin{equation*}
		\norm{\rs_{\infty}-\rs_{L}}{L^2_{\measure}(\domPS)}\lesssim  \exp(-L\frac{\wc}{\gamma+\sc})
		\end{equation*}
		\item by Assumptions {A1} and {A2(2)}, we have 
		\begin{equation*}
		\begin{split}
		e_{\vsp_{L-l},2}^2(\rs_{l}-\rs_{l-1})\lesssim \left(M^{-\alpha}\exp(-(L-l)\frac{\alpha}{\dimp+\alpha})\exp(-l\frac{\sc}{\gamma+\sc})\right)^{2}
		\end{split}
		\end{equation*}

		\item by \Cref{eq:size} and Assumption A1, we have 
		\begin{equation*}
		\begin{split}
		\norm{\rs_{l}-\rs_{l-1}}{L^2_{\measure}(\domPS)}^2|\domPS_{L-l}|^{-2\alpha/\dimp}&\lesssim \left(M^{-\alpha}\exp(-l\frac{\wc}{\gamma+\sc})\exp(-(L-l)\frac{\alpha}{\dimp+\alpha})\right)^2.
		\end{split}
		\end{equation*}
	\end{itemize}
	Combining these observations we arrive at 
	\begin{equation*}
	\begin{split}
	(\star)&\lesssim \exp(-L\frac{\wc}{\gamma+\sc})+M^{-\alpha}\sum_{l=0}^L\exp\left(-(L-l)\frac{\alpha}{\dimp+\alpha}-l\frac{\sc}{\gamma+\sc}\right)\\
	&\lesssim \exp(-L\frac{\wc}{\gamma+\sc})+M^{-\alpha}\exp(-L\frac{\alpha}{\dimp+\alpha})\sum_{l=0}^{L}\exp\left(l\bigg(\frac{\alpha}{\dimp+\alpha}-\frac{\sc}{\gamma+\sc}\bigg)\right).
	\end{split}
	\end{equation*}
From here, the proof may be concluded exactly as that of \Cref{thm:main}.
\end{proof}

%% file: content/adaptive.tex
\section{An adaptive algorithm}
\label{sec:adaptive}
We introduce in this section an adaptive algorithm for the case when an optimal sequence of polynomial subspaces, the rate of convergence $\rs_l\to \rs_{\infty}$, or the cost for evaluations of $\rs_l$ are unknown. 

To describe our algorithm, we restrict ourselves to the case when $\domPS=[0,1]^{\dps}$, $\dps\in\N$ and when $\measure=\lambda$ is the Lebesgue measure.
 By the results in \Cref{ssec:arcsine}, we may then use samples and weights from the arcsine distribution instead of the optimal distributions. This allows us to keep previous samples when we extend the polynomial subspaces, whereas using the optimal distribution, which depends on the polynomial subspace, would require throwing away all samples each time the polynomial subspace is extended.
 
 We next describe the building blocks that are used by our adaptive algorithm to select polynomial approximation subspaces.
\begin{definition}[\textbf{Multivariate Legendre polynomials}]\leavevmode
\begin{enumerate}[(i)]
	\item We denote by $(\leg_i)_{i\in\N}$ the univariate $L^2_{\lambda}([0,1])$-orthonormal Legendre polynomials and define their tensor products
	\begin{equation*}
	\begin{split}
	\leg_{\mip}:=\bigotimes_{j=1}^\dps\leg_{\mip_j}\colon [0,1]^d\to\R,\\
	\leg_{\mip}(\psmi):=\prod_{j=1}^{d}\leg_{\mip_j}(\psmi_j)
	\end{split}
	\end{equation*}
	 for $\mip\in\N^\dps$.
\item For each multi-index $\mi\in\N^{d}$, we define the polynomial subspace
\begin{equation*}
\pss_{\mi}:=\vspan\{\leg_{\mip}:2^{\mi}-1\leq\mip< 2^{\mi+1}-1\}\subset L^2([0,1]^d,\lambda).
\end{equation*}

\end{enumerate}
\end{definition}
\begin{rem}[\textbf{Orthonormal decomposition}]
	Since polynomials are dense in $L^2_{\lambda}([0,1]^d)$, the subspaces $(\pss_{\mi})_{\mi\in\N^d}$ form an orthonormal decomposition of $L^2_{\lambda}([0,1]^d)$. 	We use exponentially large subspaces instead of the simpler, one-dimensional subspaces $\pss_{\mi}=\R\cdot\leg_{\mi}$ to avoid computational overhead resulting from slow construction of large polynomial subspaces.
\end{rem}
We use the notation $\rs_{-1}:=0$ to avoid separate treatment of the term corresponding to $l=0$ in the following.
To describe a multilevel approximation, 
 we need to construct a sequence $(\vsp_k)_{k=0}^{L}$ of polynomial subspaces, such that the difference $\rs_{l}-\rs_{l-1}$ is projected onto $\vsp_{L-l}$ using weighted least squares approximation. The final approximation is then defined as 
  \begin{equation}
  \label{eq:adaptive}
  \sum_{l=0}^{L}\P_{{L-l}}(\rs_l-\rs_{l-1}).
  \end{equation}
  where $\P_{k}$ projects onto $\vsp_{k}$ for $0\leq k\leq L$.
As in \Cref{sec:nonadaptive}, if the samples used by $\P_{{k}}$ are distributed according to the optimal distribution of $\vsp_{k}$, then we require that the number of samples $\NS_{k}$ satisfy
 \begin{equation}
 \label{eq:adaptivestability}
 \kappa \frac{\NS_k}{\log \NS_k}\geq \dim \vsp_{k}
 \end{equation}
 for some $\kappa>0$. However, we then have to throw away all samples each time the space $\vsp_{k}$ is extended, as this changes the optimal distribution. As an alternative, we may use samples from the arcsine distribution, which is independent of the polynomial subspaces $\vsp_{k}$ and thus allows us to reuse samples and corresponding function evaluations. By \Cref{ssec:arcsine}, this increases the number of required samples only by a constant factor (that depends exponentially on the dimension $d$, however).\\

  To construct the sequence of polynomial subspaces in an adaptive fashion, our algorithm constructs a (finite) downward closed multi-index set $\mis\subset\N^{d+1}$. Given such a set, we let 
  \begin{equation*}
\vsp_{k}:=\bigoplus_{\mi\in\N^{d}: (\mi,L-k)\in\mis}\pss_{\mi}\quad 0\leq k\leq L,
\end{equation*}
where
\begin{equation*}
L:=\max\{l\in\N:\exists\mi\in\N^{d} \text{ s.t. }(\mi,l)\in\mis\}<\infty,
\end{equation*} 
which means that we project the difference $\rs_l-\rs_{l-1}$ onto the subspace $\vsp_{L-l}$ that is determined by the slice $\mis_{l}:=\{\mi\in\N^d:(\mi,l)\in \mis\}$ of the multi-index set $\mis$.
To construct $\mis$, starting with $\mathcal{I}=\{\mathbf{0}\}$, our algorithm adds one multi-index at a time according to the following procedure, which resembles algorithms for adaptive sparse grid integration \cite{MR2163199,gerstner2003dimension}.
We call
\begin{equation*}
\mia:=\{(\mi,l)\in \N^{d+1}\setminus \mis:\mis\cup\{\mi,l\}\text{ is downward closed}\}
\end{equation*}
the set of \emph{admissible multi-indices}.
\begin{enumerate}[(i)]
	\item
For each admissible multi-index $(\mathbf{k},l)$, we estimate the norm of the projection of $\rs_l-\rs_{l-1}$ onto $\pss_{\mi}$. This estimate represents the gain that is made by adding $(\mi,l)$ to $\mis$.
 Furthermore, we estimate the work that adding this multi-index would incur.
\item We expand $\mathcal{I}$ by the multi-index that maximizes the ratio between the gain and work estimates.
\end{enumerate}
\begin{figure}[ht]
	\centering
	\input{./figures/DCset.tex}
	\caption{Example with $d=1$ of a multi-index set $\mis$ and the associated set of admissible multi-indices $\mia$, as well as neighbors $\neighbors(2,1)$ of  $(2,1)\in\mia$. In this example $L=1$,  $\vsp_{1}=\vspan\{1,\psmi,\dots,\psmi^{6}\}=\vspan\{\leg_0(\psmi),\dots,\leg_{6}(\psmi)\}$, and  $\vsp_{0}=\vspan\{1,\psmi,\psmi^2\}=\vspan\{\leg_0(\psmi),\leg_1(\psmi),\leg_2(\psmi)\}$.}
	\label{fig:dc}
\end{figure}

We next explain how we arrive at the gain and work estimates that are required in step (i).
To estimate the norm of the orthogonal projection of $\rs_l-\rs_{l-1}$ onto $\pss_{\mi}$, we compute the arithmetic average of corresponding estimates for the neighbors $\neighbors(\mi,l)=\{(\mi^{(1)},l^{(1)}),\dots\}$ of $(\mi,l)$ in $\mathcal{I}$. Here, by neighbor we mean elements of $\mis$ that differ from $(\mi,l)$ in a single entry by $1$, see \Cref{fig:dc}.
For each such neighbor, we estimate the norm of the orthogonal projection $\operatorname{Proj}_{\mi^{(j)}}(\rs_{l^{(j)}}-\rs_{l^{(j)}-1})$ of $\rs_{l^{(j)}}-\rs_{l^{(j)}-1}$ onto $\pss_{\mi^{(j)}}$ simply by computing the Euclidean norm of those basis coefficients of $\P_{L-l^{(j)}}(\rs_{l^{(j)}}-\rs_{l^{(j)}-1})$ that belong to $\pss_{\mi^{(j)}}$. (Recall that $\P_{L-l^{(j)}}$ is a discrete projection onto the space $\vsp_{L-l^{(j)}}$ of which $\pss_{\mi^{(j)}}$ is a subspace since $\mi^{(j)}\in\mis_{l^{(j)}}$.)
The final estimate can be expressed as
\begin{equation*}
\frac{1}{|\neighbors(\mi,l)|}\sum_{j=1}^{|\neighbors(\mi,l)|}\norm{\operatorname{Proj}_{\mi^{(j)}}\P_{L-l^{(j)}}(\rs_{l^{(j)}}-\rs_{l^{(j)}-1})}{L^2_{\lambda}([0,1]^d)}.
\end{equation*}

To estimate the work that adding $(\mi,l)$ to $\mis$ incurs, we observe that \Cref{eq:adaptivestability} tells us exactly how many new samples are needed. More specifically, if we denote by $\NS(\mis_{l})$ the minimal solution of \Cref{eq:adaptivestability} for the polynomial subspace determined by $\mis_l$, then the required number of new samples of $\rs_{l}-\rs_{l-1}$ is $\NS(\mis_{l}\cup\{\mi\})-\NS(\mis_{l})$. It therefore remains to determine the  work per sample,  $\work(\rs_{l}-\rs_{l-1})$. If this work is unknown, then we store for each level $l$ an estimate, which we update with the observed computational work divided by the number of generated samples each time $\mathcal{I}_{l}$ changes. The final estimate of the work associated with $(\mi,l)$ is
\begin{equation*}
  \work(\rs_{l}-\rs_{l-1}) \cdot\left(\NS(\mis_{l}\cup\{\mi\})-\NS(\mis_{l})\right).
\end{equation*}

\Cref{alg:adaptive} gives a summary of our algorithm in pseudocode.

\begin{algorithm}[ht]
\caption{Adaptive multilevel algorithm.}\label{alg:adaptive}
\begin{algorithmic}[1]
		\Function{MLA}{$(\rs_l)_{l\in\N}$,STEPS}
		\State  $\mis\gets \{\mathbf{0}\}$
		\State $X_l\gets \varnothing\;\forall l\in\N$
		\State $\Delta_l\gets 0\;\forall l\in\N$
		\For{$0\leq i<\text{STEPS}$}
			\State $(\mi,l)\gets \argmax_{(\mi,l)\in\mia}\frac{\text{GAIN}((\mi,l),(\Delta_l)_{l\in\N},\mis)}{\text{WORK}((\mi,l),\mis)}$
			\State $N_{+}\gets \NS(\mis_{l}\cup\{\mathbf{k}\})-\NS(\mis_{l})$
			\State $\mis\gets \mathcal{I}\cup\{(\mi,l)\}$
			\For{$0\leq j<N_{+}$}
				\State Generate $\psmi\sim \arcsine_{d}$
				\State $y\gets (\rs_{l}-\rs_{l-1})(\psmi)$
				\State $X_l\gets X_l\cup\{(\psmi,y)\}$
			\EndFor
			\State $\Delta_l\gets \P_{L-l}(\rs_{l}-\rs_{l-1})$
		\EndFor
		\State \Return $\sum_{0\leq l\leq L}\Delta_l$
	\EndFunction
	\item[]
	\Function{GAIN}{$(\mi,l)$,$(\Delta_l)_{l\in\N}$,$\mathcal{I}$}
		\State $s=0$
		\For{$(\mi^{(j)},l^{(j)})\in \neighbors(\mi,l)$}
			\State $s\gets s+\|\operatorname{Proj}_{\mi^{(j)}}\Delta_{l^{(j)}} \|_{L^2_{\lambda}}$
		\EndFor

		\State \Return $s/|\neighbors(\mathbf{k},l)|$
	\EndFunction
	\item[]
	\Function{WORK}{$(\mi,l)$,$\mathcal{I}$}
		\State \Return $\work(\rs_{l}-\rs_{l-1})\cdot\left(\NS(\mis_{l}\cup\{\mi\})-\NS(\mis_{l})\right)$
	\EndFunction

\end{algorithmic}
\end{algorithm}

%% file: figures/DCset.tex
\begin{tikzpicture}

\begin{axis}[
xmin=-0.5, xmax=3.5,
ymin=-0.5, ymax=3.5,
axis on top,
xmajorgrids,
ymajorgrids,
xlabel=$\mi$,
ylabel=$l$
]
\addplot [only marks,mark size=4, draw=black, fill=black]
table {%
x                      y
+0.000000000000000e+00 +0.000000000000000e+00
+1.000000000000000e+00 +0.000000000000000e+00
+0.000000000000000e+00 +1.000000000000000e+00
+2.000000000000000e+00 +0.000000000000000e+00
+1.000000000000000e+00 +1.000000000000000e+00
};
\addplot [only marks, mark size=4,draw=blue, fill=blue]
table {%
x                      y
+3.000000000000000e+00 +0.000000000000000e+00
+2.000000000000000e+00 +1.000000000000000e+00
+0.000000000000000e+00 +2.000000000000000e+00
};
\addplot [only marks,mark size=4,line width=1.5, draw=red, fill=black]
table {%
x                      y
+1.000000000000000e+00 +1.000000000000000e+00
+2.000000000000000e+00 +0.000000000000000e+00
};
\addlegendentry{$\mis$};
\addlegendentry{$\mia$};
\addlegendentry{$\neighbors(2,1)$};
\end{axis}

\end{tikzpicture}

%% file: content/UQ.tex
\section{Application to parametric PDE}
\label{sec:uq}

We assume in this section that $\pde(\cdot,\psmi)$ is the solution of some partial differential equation (PDE) with parameters $\psmi\in\domPS\subset\R^\dps$ and that we are interested in the \emph{response surface}
\begin{equation*}
\psmi\mapsto \rs_{\infty}(\psmi):=\QoI(\pde(\cdot,\psmi))\in\R,
\end{equation*}
where $\QoI(\pde(\cdot,\psmi))$ is a real-valued quantity of interest, such as a point evaluation, a spatial average, or a maximum.  
In most situations, we cannot evaluate $\rs_{\infty}(\psmi)$ exactly, as this would require an analytic solution of the PDE. Instead, we have to work with discretized solutions $\pde_{n}(\cdot,\psmi)$ for each $\psmi$, which yield approximate response surfaces
\begin{equation*}
\begin{split}
\rs_{n}\colon &\domPS\to\R \\
&\psmi\mapsto Q(\pde_{n}(\cdot,\psmi)).
\end{split}
\end{equation*}
For example, if we employ finite element discretizations with maximal element diameter $h:=n^{-1}$, then the work required for evaluations of $\rs_{n}$ grows like $h^{-\gamma}=n^{\gamma}$ for some $\gamma>0$. To apply the multilevel method of \Cref{sec:nonadaptive}, we need to verify the remaining Assumptions A1 and A2 from there. 

As a motivating example, we consider a linear elliptic second order PDE, which has been extensively studied in recent years \cite{harbrecht2013multilevel,ChkifaCohenSchwab2015,CohenDevoreSchwab2011,BabuskaTemponeZouraris2004},
\begin{equation}
\label{eq:UQex}
\begin{aligned}
-\nabla \cdot (a(x,\psmi) \nabla \pde(x,\psmi))&=\rhs(x)&\text{ in }U\subset\R^{\dpde}\\
\pde(x,\psmi)&=0&\text{ on } \partial U,
\end{aligned}
\end{equation}
with $a:U\times \domPS\to\R$ and $\domPS:=[0,1]^\dps$.

\begin{pro}
	\label{pro:finite}
	For any $n\in\N$, let $\pde_{n}$ be finite element approximations of order $r\geq 1$ and maximal element diameter $h:=(n+1)^{-1}$, and let $\rs_{n}(\psmi):=Q(\pde_n(\cdot,\psmi))$.
	Assume that $g$ and $U$ are sufficiently smooth, that 
	\begin{equation}
		\inf_{x\in U,\psmi\in\domPS}a(x,\psmi)>0,
	\end{equation}
	and that $Q$ is a continuous linear functional on $L^2(U)$.

	\begin{enumerate}[(i)]
		\item If $a\in C^{r}(U\times\domPS)$ for some $r\geq 1$, then  
		\begin{equation*}
		\norm{\rs_{\infty}-\rs_{n}}{L^2(\domPS)}\lesssim h^{r+1}
		\end{equation*}
		and
		\begin{equation*}
		\norm{\rs_{\infty}-\rs_{n}}{C^{r-1}(\domPS)}\lesssim h^{2}.
		\end{equation*}
		\item If for some $r,s\geq 1$ we have
		\begin{equation}
		\label{eq:tensor}
		\begin{split}
		a\in C^{r}(U)\otimes C^{s}(\domPS):=\{a\colon U\times \domPS\to\R :  \norm{\partial_{x}^{\mathbf{r}}\partial_{\psmi}^{\mathbf{s}}a}{C^0(U\times\domPS)}<\infty\;\forall\;|\mathbf{r}|_{1}\leq r,|\mathbf{s}|_{1}\leq s\}
		,
		\end{split}
		\end{equation} then 
		\begin{equation*}
		\norm{\rs_{\infty}-\rs_{n}}{C^{s}(\domPS)}\lesssim h^{r+1}.
		\end{equation*}
		
	\end{enumerate}
\end{pro}
\begin{proof}
	In both cases, the standard theory of second order elliptic differential equations shows that $\psmi\mapsto \pde(\cdot,\psmi)$ is well defined as a map from $\domPS$ into $H^{r+1}(U)$, with 
	\begin{equation*}
	\norm{\pde}{L^\infty(\domPS;H^{r+1}(U))}<\infty.
	\end{equation*}
	Next, we observe that the derivatives $\partial_{\psmi_{j}}\pde(\cdot,\psmi)$, $j\in \{1,\dots,d\}$ satisfy PDEs with the same operator as in \Cref{eq:UQex} but with new right-hand sides
	\begin{equation*}
	\tilde{\rhs}(x):=\nabla\cdot(\partial_{\psmi_j}a(x,\psmi)\nabla \pde(x,\psmi)).
	\end{equation*}
	The regularity of this right-hand side now depends on the assumptions on the coefficient $a$.
	In case (i) we have $\partial_{\psmi_j}a(\cdot,\psmi)\in C^{r-1}(U)$ and thus $\tilde{\rhs}\in H^{r-2}(U)$. Therefore, $\partial_{\psmi_{j}}\pde(\cdot,\psmi)\in H^{r}(U)$ for each $\psmi\in\domPS$ and, moreover, we have the uniform estimate
		\begin{equation*}
		\norm{\partial_{\psmi_j}\pde}{L^\infty(\domPS;H^{r}(U))}<\infty.
		\end{equation*}
		In case (ii) we have $\partial_{\psmi_j}a(\cdot,\psmi)\in C^{r}(U)$ and thus $\tilde{\rhs}\in H^{r-1}(U)$. Therefore, $\partial_{\psmi_{j}}\pde(\cdot,\psmi)\in H^{r+1}(U)$ for each $\psmi\in\domPS$ and, moreover, we have the uniform estimate
		\begin{equation*}
		\norm{\partial_{\psmi_j}\pde}{L^\infty(\domPS;H^{r+1}(U))}<\infty.
		\end{equation*}
		Repeatedly applying these arguments yields
			\begin{equation*}
			\norm{\pde}{C^{r-1}(\domPS;H^2(U))}<\infty,
			\end{equation*}
			and
			\begin{equation*}
			\norm{\pde}{C^{s}(\domPS;H^{r+1}(U))}<\infty,
			\end{equation*}
		in cases (i) and (ii), respectively. We may now conclude by using standard finite-element theory. In case (i), we have
		\begin{equation*}
		\begin{split}
		\norm{\rs_{\infty}-\rs_{n}}{L^2(\domPS)}&\leq \norm{Q}{}\norm{\pde-\pde_{n}}{L^2(\domPS;L^2(U))}\\
		&\lesssim h^{r+1}\norm{\pde}{L^2(\domPS;H^{r+1}(U))}
		\end{split}
		\end{equation*}
		and
		\begin{equation*}
		\begin{split}
		\norm{\rs_{\infty}-\rs_{n}}{C^{r-1}(\domPS)}&\lesssim \norm{\pde-\pde_{n}}{C^{r-1}(\domPS;L^2(U))}\\
		&\lesssim h^{2}\norm{\pde}{C^{r-1}(\domPS;H^{2}(U))},
		\end{split}
		\end{equation*} 
	whereas in case (ii), we have
	\begin{equation*}
	\begin{split}
	\norm{\rs_{\infty}-\rs_{n}}{C^{s}(\domPS)}&\lesssim \norm{\pde-\pde_{n}}{C^{s}(\domPS;L^2(U))}\\
	&\lesssim h^{r+1}\norm{\pde}{C^{s}(\domPS;H^{r+1}(U))},
	\end{split}
	\end{equation*}
\end{proof}
\begin{rem}
In case (i) of the previous proposition, differentiating with respect to $\psmi$ reduces the number of available derivatives in $x$, which are required for convergence of the finite element method. Thus, the convergence in $L^2(\domPS)$ is faster than that in $C^{r-1}(\domPS)$. Case (ii), on the other hand, describes the so-called \emph{mixed smoothness} of the coefficient in $x$ and $\psmi$, meaning that differentiating in $\psmi$ does not affect the differentiability with respect to $x$.  
\end{rem}
If the coefficients depend analytically on $\psmi$, then the same holds for $\rs_{\infty}$, which can be exploited to obtain algebraic polynomial approximability rates of $\rs_{\infty}$ even in the case of infinite-dimensional parameters \cite{ChkifaCohenSchwab2015,Haji-AliNobileTamelliniEtAl2015}, as shown below.



\begin{pro}
\label{pro:UQ}
Let $\domPS:=[-1,1]^{\infty}$.
Assume that $Q$ is a linear and continuous functional on $L^2(U)$, that $0<\inf_{x,\psmi}  a(x,\psmi)\leq \sup_{x,\psmi}  a(x,\psmi)<\infty$, and that
	\begin{align*}
	a(x,\psmi)&=\bar{a}(x)+\sum_{j=0}^{\infty}\ps_j\psi_j(x),\\
		a(x,\psmi)&=\bar{a}(x)+\left(\sum_{j=0}^{\infty}\ps_j\psi_j(x)\right)^2,\quad \\
	\intertext{or}
		a(x,\psmi)&=\exp\left(\sum_{j=0}^{\infty}\ps_j\psi_j(x)\right). 
	\end{align*}
	If there exists $r_{\max}>1$ such that
	\begin{equation*}
	 \|\psi_j\|_{C^{r}(U)}\lesssim (j+1)^{-(r_{\max}+1 - r)}\quad\forall j\in\N, \; 0\leq r< r_{\max},
	 \end{equation*} 
	 	 then, for any $r\in \N$ with $1\leq r< r_{\max}$, finite element approximations with maximal element diameter $h:=(n+1)^{-1}$ achieve
\begin{equation*}
\norm{\rs_{\infty}-\rs_{\n}}{L^{\infty}(\domPS)}\leq C h^{r+1}
\end{equation*}
with a constant $C$ independent of $n$.
Furthermore, for any such $r$, there is a sequence $(\vsp_\dvsp)_{\dvsp\in\N}$ of downward closed polynomial spaces with $\dim \vsp_\dvsp=m$ such that finite element approximations with order $r$ and maximal diameter $h:=(n+1)^{-1}$ achieve 
\begin{equation*}
e_{\vsp_\dvsp,1,\infty}(\rs_{\infty}-\rs_\n)\leq C(\dvsp+1)^{-\uqsummability}h^{r+1} \quad \forall\, 0<\uqsummability<r_{\max}- r
\end{equation*}
with a constant $C$ independent of $n$ and $m$.
\end{pro}
\begin{proof}
	
	It was shown in \cite[Theorem 4.1 \& Section 5]{ChkifaCohenSchwab2015} that for each $0\leq r<r_{\max}$ there exists a set $\domPS_{r}\subset \C^{\infty}$, $\domPS\subset \domPS_{r}$ such that
	$\norm{a}{L^{\infty}(\domPS_{r};C^{r}(U))}<\infty$ and such that $\psmi\mapsto \pde(\cdot,\psmi)$ may be extended to a complex differentiable map from $\domPS_{r}$ into $H^{1+r}(U)$ with
		\begin{equation}
		\label{eq:complexbounds}
		\begin{split}
	\norm{\pde}{L^{\infty}(\domPS_{r};H^{1+r}(U))}<\infty
		\end{split}
		\end{equation}
	
For a detailed description of the sets $\domPS_{r}$ we refer to \cite{ChkifaCohenSchwab2015}. For our purposes it suffices to know that the better the summability of $(\norm{\psi_j}{C^{r}(U)})_{j\in\N}$, the larger $\domPS_{r}$ can be chosen; and the larger $\domPS_{r}$ the better the polynomial approximability properties of complex differentiable maps defined on $\domPS_{r}$.
In particular, the results of \cite[Section 2]{ChkifaCohenSchwab2015}, show that when  restricted to the smaller set $\domPS$ such maps may be approximated at algebraic convergence rates within downward closed polynomial subspaces. More specifically, \cite[Equation (2.27)]{ChkifaCohenSchwab2015} shows that if a function $e$ is complex differentiable on $\domPS_{r}$, then for any $\dvsp\in\N$ there exists a downward closed polynomial subspace $\vsp_{\dvsp}$ such that
	\begin{equation*}
	\begin{split}
	\inf_{\tilde{\p}\in \vsp_{\dvsp}\otimes L^2(U)}\norm{e-\tilde{\p}}{L^\infty(\domPS;L^2(U))}\lesssim (\dvsp+1)^{-\alpha} \norm{e}{L^\infty(\domPS_{r};L^2(U))}
	\end{split}
	\end{equation*}
	for all $\alpha<r_{\max}- r$. 
 Applying this estimate with $e:=\pde-\pde_{n}$ shows
	\begin{equation*}
	\begin{split}
	\inf_{\p\in\vsp_{\dvsp}}\norm{(\rs_{\infty}-\rs_{\n})-\p}{L^\infty(\domPS)}&\leq \norm{Q}{} \inf_{\tilde{\p}\in \vsp_{\dvsp}\otimes L^2(U)}\norm{(\pde-\pde_n)-\tilde{\p}}{L^\infty(\domPS;L^2(U))}\\
\lesssim & (\dvsp+1)^{-\alpha} \norm{\pde-\pde_{n}}{L^\infty(\domPS_{r};L^2(U))}.
	\end{split}
	\end{equation*}

By standard finite element analysis we finally obtain
\begin{equation*}
\norm{\pde-\pde_{n}}{L^{\infty}(\domPS_{r};L^2(U))}\leq C h^{r+1}\norm{\pde}{L^{\infty}(\domPS_{r};H^{r+1}(U))}.
\end{equation*}
with  $C=C\big(	\norm{a}{L^{\infty}(\domPS_{r};C^{r}(U))}\big)<\infty$. Combining the previous two estimates with \Cref{eq:complexbounds} concludes the proof.
\end{proof}

\begin{rem}
	Similar results can also be shown for PDEs of parabolic type and for some nonlinear PDEs \cite{ChkifaCohenSchwab2015}.
\end{rem}

%% file: content/numerics.tex
\section{Numerical Experiments}
\label{sec:numerics}
To support our theoretical analysis, we performed numerical experiments on linear elliptic parametric PDEs of the form

\begin{equation}
\label{eq:pdenum}
\begin{aligned}
-\nabla \cdot (a(x,\psmi) \nabla \pde(x,\psmi))&=1&&\text{ in }U:=[-1,1]^D\\
\pde(x,\psmi)&=0&&\text{ on } \partial U,
\end{aligned}
\end{equation}
as in \Cref{sec:uq}.
We let
\begin{equation*}
a(x,\psmi)={1 + \|x\|_2^{r} + \|\psmi\|_2^{s}},\quad \psmi\in\domPS:=[-1,1]^d
\end{equation*}
for $r := 1$, $s := 3$, $D:=2$ and $d\in\{2,3,4,6\}$.
Our goal was to approximate the response surface
\begin{equation*}
\psmi\mapsto \rs(\psmi):=\QoI(\pde(\cdot,\psmi)):=0.5 \int_{U}\pde(\cdot,\psmi)\;dx
\end{equation*}
in $L^2(\Gamma)$.
%
%
The numerical scheme we used to solve \Cref{eq:pdenum} employs
centered finite difference approximations to the derivatives with a
constant mesh size, $h$. Such a numerical scheme converges
asymptotically at a rate of $\mathcal O(h^{2})$ in the $L^2$ norm
and requires a computational work of $\mathcal O(h^{-2})$, since the
PDE is two-dimensional. This corresponds to the values
$\beta_s = \beta_w = 2$ and $\gamma=2$ for the parameters in
Assumptions A2 and A3.
To estimate the projection error of our
estimate we evaluate the $L^2$ error norm using Monte Carlo sampling with $M=1000$ samples,
\begin{equation}\label{eq:l2-mc-error}
  \norm{f - S_L(f)}{L^2(\Gamma)}^2 \approx \frac{1}{M} \sum_{j=1}^M
  (f_{L+1}(\psmi_j) - S_L(f)(\psmi_j))^2.
\end{equation}
In our tests we employ both the nonadaptive and the adaptive
algorithms from Sections 4 and 5. As a basis for the nonadaptive
algorithm, we use total degree polynomial spaces
\(
  \vsp_{\dvsp} := \vspan\left\{ \leg_{\eta} : |\eta|_{1} \leq \dvsp \right\},
\)
where $\leg_\eta$ is a tensor product of Legendre
polynomials as in \Cref{sec:adaptive}.  We also compare the multilevel algorithm to the
straightforward, single-level approach, which for a given polynomial
approximation space $\vsp_{\dvsp}$ uses samples from a fixed PDE discretization
level that matches the accuracy of the polynomial best
approximation in $\vsp_{\dvsp}$. To find these matching PDE discretization levels, we consider the complexity curve of the single-level
method as the lower envelope of complexity curves with different
PDE discretization levels. Even though such a method is not practical, the
choice of discretization level for a given tolerance is always
optimal. The random points were sampled from the optimal
distribution as explained in \Cref{sec:optimal-sampling}.

Before presenting the numerical results, let us derive some a-priori
estimates of the complexity of the single-level and multilevel
projection methods.  From \Cref{pro:finite}, if
$a \in C^{r}(U)\otimes C^{s}(\domPS)$, then using finite elements
of order $r$ and mesh size $h$ would yield convergence in the space $F:=C^{s}(\domPS)$ with the values $\sc=\wc=r+1$ of the parameters in \Cref{sec:nonadaptive}, and  optimal solvers
would require the work $\mathcal{O}(h^{-\gamma})$,  $\gamma:=D$. Furthermore, since functions in
$C^{s}(\domPS)$ are approximable by polynomials of total degree less than or equal to $k$ at the rate $\mathcal{O}(k^{-s})$ in the supremum norm \cite{BagbyBosLevenberg2002}, we expect at least $\alpha=s$.
Even though our choice $a(x,\psmi) = 1 + \|  x\|_2^{r} + \| \psmi\|_2^{s}$
satisfies only $a \in C^{r-1, 1}(U)\otimes C^{s-1,1}(\domPS)$, we do not expect different rates than those derived above for $a\in C^{r}(U)\otimes C^{s}(\domPS)$. Finally, the dimension of total degree polynomial spaces $\vsp_{\dvsp}$ equals $\binom{\dvsp+d}{d}$ and asymptotically we have $\binom{\dvsp+d}{d}\lesssim \dvsp^d$.

Thus, we expect the complexity of the single-level method to be
$\mathcal{O}\left(\epsilon^{-\frac{D}{r+1} -
    \frac{d}{s}}\log(\epsilon^{-1})\right)$, while the complexity of
the multilevel method is of
$\mathcal{O}\left({\epsilon^{-\max\left({\frac{D}{r + 1},
          \frac{d}{s}}\right)}} \log(\epsilon^{-1})^{t}\right)$,
where
\[t =
  \begin{cases}
    1 & \frac{D}{r + 1} > \frac{d}{s},\\
    3 + \frac{D}{r+1} & \frac{D}{r + 1} = \frac{d}{s},\\
    2 & \frac{D}{r + 1} < \frac{d}{s}.
  \end{cases}
\]
Hence, for $r=1$ and $s=3$, the complexity of the single-level
method is
$\mathcal{O}\left(\epsilon^{-1 -
    \frac{d}{3}}\log(\epsilon^{-1})\right)$ and the complexity of the
multilevel method is
$\mathcal{O}\left({\epsilon^{-\max({1, \frac{d}{3}})}}
  \log(\epsilon^{-1})^{t}\right)$ where
\[t =
  \begin{cases}
    1, & d < 3,\\
    4, & d = 3,\\
    2, & d > 3.
  \end{cases}
\]

\Cref{fig:kink-work} shows the work estimate as defined in
\cref{eq:workdef} versus the $L^2$ error approximation in
\Cref{eq:l2-mc-error}. The results for the multilevel algorithm displayed there were obtained with the work parameter $\dimp:=d/2$, which we found describes the pre-asymptotic behavior of $\dim \vsp_{\dvsp}=\binom{\dvsp+d}{d}$ better than $\dimp:=d$. The theoretical rates
satisfactorily match the obtained numerical rates, which show an
improvement of the multilevel methods over the single-level method. Note
that the work estimate does \emph{not} include the cost of generating
points or the cost of assembling the projection matrix and computing
the projection. On the other hand, these costs are included in
\Cref{fig:kink-time}, which shows the total running time in seconds of the
three different methods. While these plots still show the same
complexity rates as \Cref{fig:kink-work} for all three methods for
sufficiently small errors, these plots also show the overhead of the
multilevel methods, especially as $d$ increases. The overhead of the
adaptive algorithm for the multilevel method is especially significant
and more work needs to be done to reduce it.

\setlength\figureheight{7.4cm}
\setlength\figurewidth{8.2cm}
\providecommand{\figlabel}{fig:}

\begin{figure}
	\centering
	\begin{subfigure}{0.49\textwidth}
      \renewcommand{\figlabel}{fig:work-est-vs-error-d2}
      \input{./figures/poisson-kink-2/work-est-vs-error.tex}
      \caption{$d=2$}
	\end{subfigure}
	\begin{subfigure}{0.5\textwidth}
      \renewcommand{\figlabel}{fig:work-est-vs-error-d3}
      \input{./figures/poisson-kink-3/work-est-vs-error.tex}
      \caption{$d=3$}
	\end{subfigure}
	\begin{subfigure}{0.49\textwidth}
      \renewcommand{\figlabel}{fig:work-est-vs-error-d4}
      \input{./figures/poisson-kink-4/work-est-vs-error.tex}
      \caption{$d=4$}
	\end{subfigure}
	\begin{subfigure}{0.5\textwidth}
      \renewcommand{\figlabel}{fig:work-est-vs-error-d6}
      \input{./figures/poisson-kink-6/work-est-vs-error.tex}
      \caption{$d=6$}
	\end{subfigure}
	\caption{$L^2([-1,1]^d)$-error, approximated using
      \Cref{eq:l2-mc-error} vs work estimate \Cref{eq:workdef} of
      single-level (SL), multilevel (ML) and adaptive ML (ML adaptive)
      methods for a linear elliptic PDE with non-smooth parameter
      dependence. The grey dotted lines are the complexity curves of
      different runs of the single-level, each with a different PDE
      discretization level. The single-level (SL) complexity curve is
      then the lower envelope of all single-level complexity
      curves. This figure shows the agreement of the numerical results
      with the theoretical rates.}
	\label{fig:kink-work}
  \end{figure}
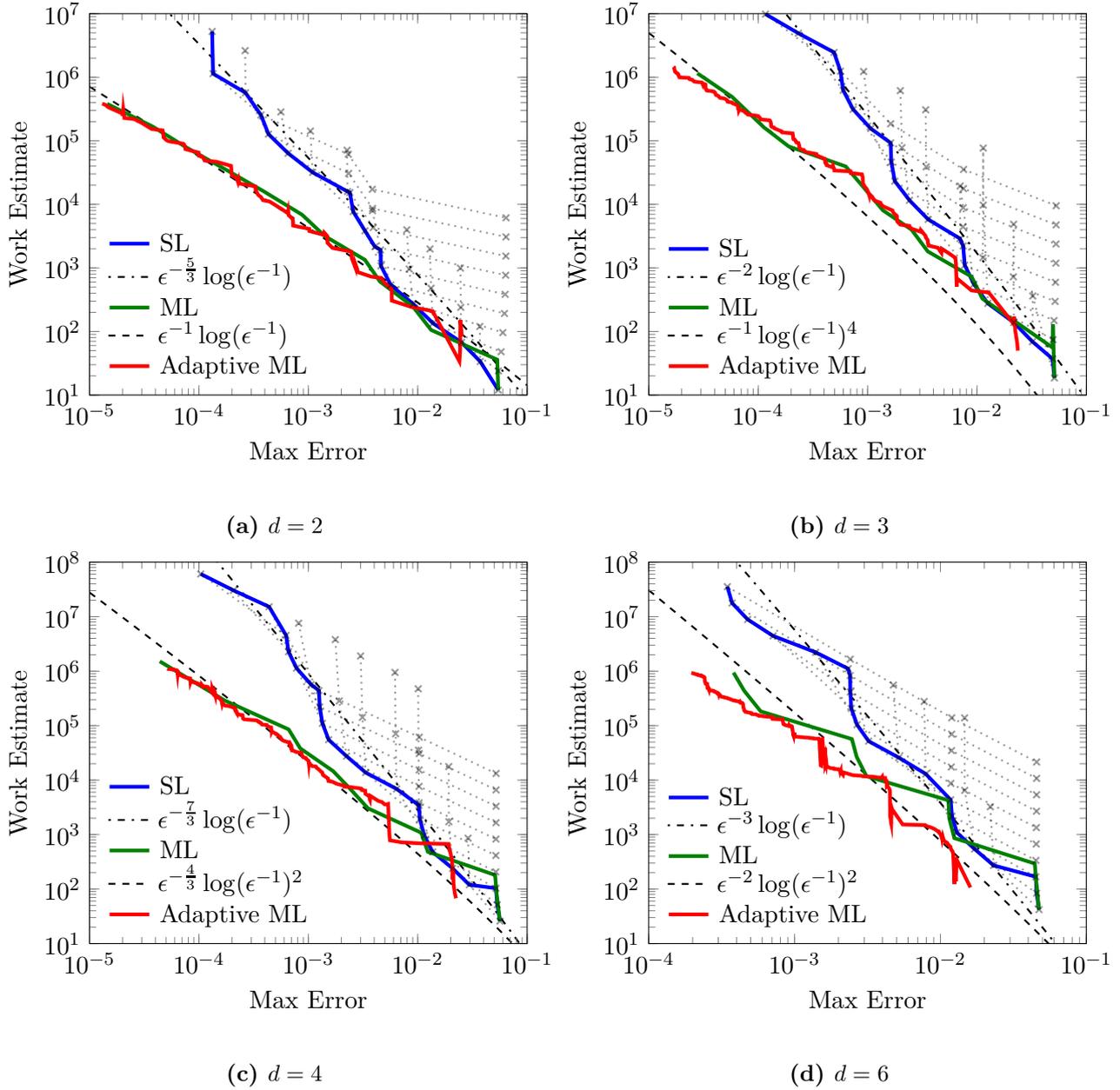

  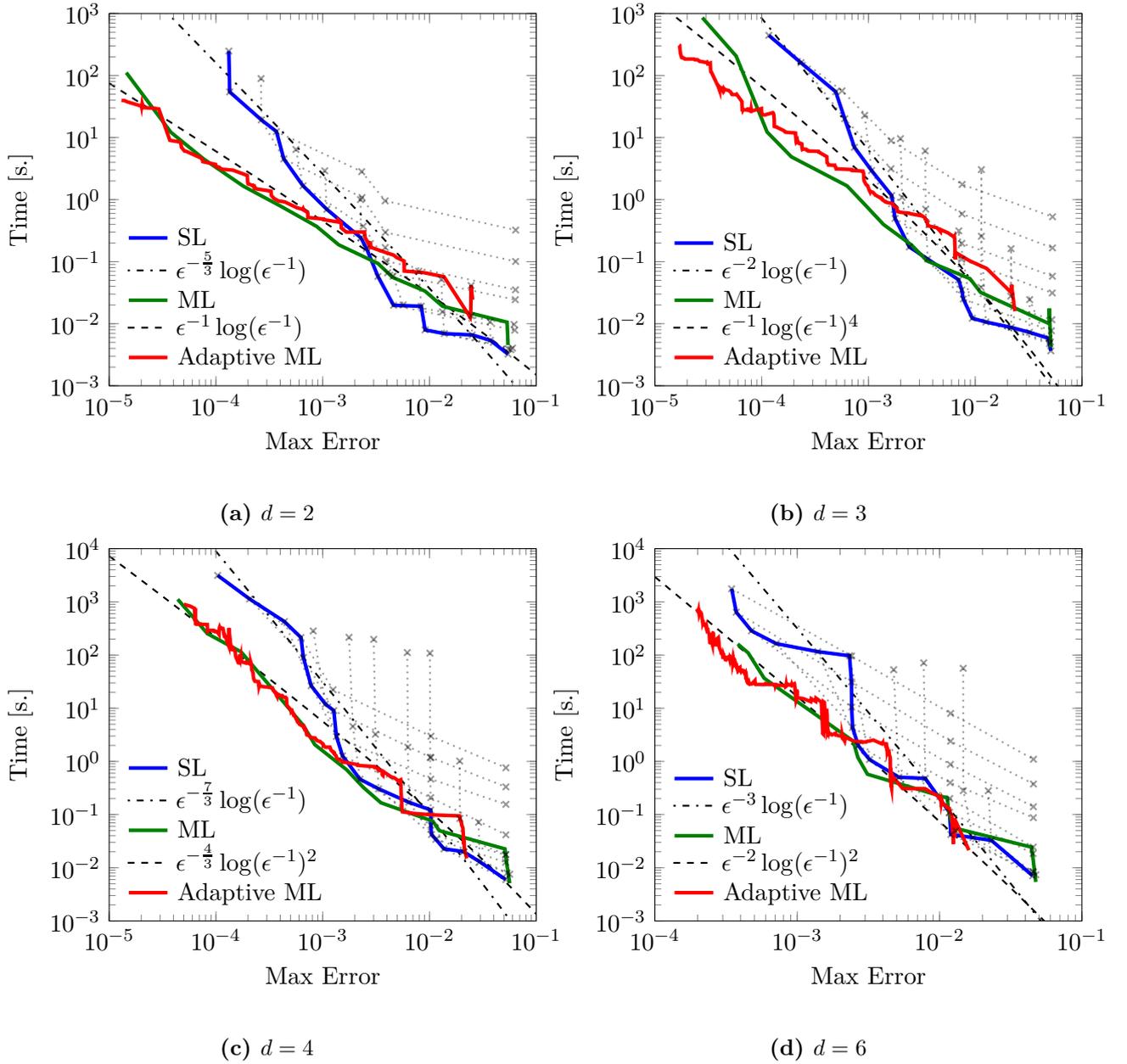
\begin{figure}
	\centering
    \begin{subfigure}{0.49\textwidth}
      \renewcommand{\figlabel}{fig:total-time-vs-error-d2}
      \input{./figures/poisson-kink-2/total-time-vs-error.tex}
      \caption{$d=2$}
	\end{subfigure}
	\begin{subfigure}{0.5\textwidth}
      \renewcommand{\figlabel}{fig:total-time-vs-error-d3}
      \input{./figures/poisson-kink-3/total-time-vs-error.tex}
      \caption{$d=3$}
	\end{subfigure}
	\begin{subfigure}{0.49\textwidth}
      \renewcommand{\figlabel}{fig:total-time-vs-error-d4}
      \input{./figures/poisson-kink-4/total-time-vs-error.tex}
      \caption{$d=4$}
	\end{subfigure}
	\begin{subfigure}{0.5\textwidth}
      \renewcommand{\figlabel}{fig:total-time-vs-error-d6}
      \input{./figures/poisson-kink-6/total-time-vs-error.tex}
      \caption{$d=6$}
	\end{subfigure}
	\caption{Similar to \Cref{fig:kink-work}, but showing the total
      running time of the methods instead of their work estimate. The
      discrepancy of the two figures is due to the overhead of
      sampling the points, assembling the projection matrix and
      computing the projection. Moreover, this plot shows the overhead
      of the adaptive algorithm compared to the non-adaptive one.}
	\label{fig:kink-time}
  \end{figure}

%% file: figures/poisson-kink-2/work-est-vs-error.tex
\begin{tikzpicture}

\begin{axis}[
xlabel={Max Error},
ylabel={Work Estimate},
xmin=1e-05, xmax=0.1,
ymin=10, ymax=10000000,
xmode=log,
ymode=log,
axis on top,
name=\figlabel,
width=\figurewidth,
height=\figureheight,
xtick={1e-06,1e-05,0.0001,0.001,0.01,0.1,1,10},
xticklabels={,$10^{-5}$,$10^{-4}$,$10^{-3}$,$10^{-2}$,$10^{-1}$,,},
ytick={1,10,100,1000,10000,100000,1000000,10000000,100000000},
yticklabels={,$10^{1}$,$10^{2}$,$10^{3}$,$10^{4}$,$10^{5}$,$10^{6}$,$10^{7}$,},
tick pos=both
]
\addplot [thick, black, opacity=0.4, dotted, mark=x, mark size=2, mark options={solid,fill opacity=0}, forget plot]
table {%
0.0369623451546 120
0.0372614771795 33.6882590646
0.0544651671895 12
};
\addplot [thick, black, opacity=0.4, dotted, mark=x, mark size=2, mark options={solid,fill opacity=0}, forget plot]
table {%
0.0245481961676 994.2411520904
0.0247423094662 240
0.0251277204434 67.3765181294
0.0546298039869 24
};
\addplot [thick, black, opacity=0.4, dotted, mark=x, mark size=2, mark options={solid,fill opacity=0}, forget plot]
table {%
0.01294146219175 1988.482304176
0.0132315526855 480
0.01378949183 134.7530362588
0.0578140123296 48
};
\addplot [thick, black, opacity=0.4, dotted, mark=x, mark size=2, mark options={solid,fill opacity=0}, forget plot]
table {%
0.00794432726001667 3976.964608364
0.00835749102082 960
0.00910236800804 269.506072518
0.0599827769776 96
};
\addplot [thick, black, opacity=0.4, dotted, mark=x, mark size=2, mark options={solid,fill opacity=0}, forget plot]
table {%
0.00385477527667 7953.929216712
0.00456852438756 1920
0.0056613128853 539.012145036
0.0620678979094 192
};
\addplot [thick, black, opacity=0.4, dotted, mark=x, mark size=2, mark options={solid,fill opacity=0}, forget plot]
table {%
0.00226099166803 71156.37380646
0.00228054100492 15907.85843341
0.00329499256326 3840
0.00461446459391 1078.02429007
0.062948013559 384
};
\addplot [thick, black, opacity=0.4, dotted, mark=x, mark size=2, mark options={solid,fill opacity=0}, forget plot]
table {%
0.0010445467109675 142312.74761344
0.00109392185408 31815.71686694
0.0025720983848 7680
0.00405857663354 2156.04858014
0.0636452477121 768
};
\addplot [thick, black, opacity=0.4, dotted, mark=x, mark size=2, mark options={solid,fill opacity=0}, forget plot]
table {%
0.0005609622120305 284625.4952258
0.000653437109663 63631.4337338
0.00239813854189 15360
0.00392095687905 4312.09716028
0.0639281712417 1536
};
\addplot [thick, black, opacity=0.4, dotted, mark=x, mark size=2, mark options={solid,fill opacity=0}, forget plot]
table {%
0.000264467681055333 2623699.5408114
0.000266457400041 569250.9904515
0.000432004928221 127262.8674675
0.00233503884306 30720
0.00386412368807 8624.19432056
0.0641029550006 3072
};
\addplot [thick, black, opacity=0.4, dotted, mark=x, mark size=2, mark options={solid,fill opacity=0}, forget plot]
table {%
0.000132059976775 5247399.0816256
0.000134597669583 1138501.980904
0.000368612154141 254525.734935
0.00231851851758 61440
0.00384578519125 17248.38864112
0.0641815074831 6144
};
\addplot [thick, black, dash pattern=on 1pt off 3pt on 3pt off 3pt]
table {%
1e-05 193021887.814683
1.09749876549306e-05 163961840.496109
1.20450354025878e-05 139267630.459887
1.32194114846603e-05 118284622.624733
1.45082877849594e-05 100456167.171241
1.59228279334109e-05 85308957.1167019
1.74752840000768e-05 72440570.6818636
1.91791026167249e-05 61508872.9128733
2.10490414451202e-05 52222999.4640243
2.31012970008316e-05 44335686.6895027
2.53536449397011e-05 37636747.3038486
2.78255940220713e-05 31947520.7617652
3.05385550883341e-05 27116152.9535329
3.35160265093884e-05 23013581.4729542
3.67837977182863e-05 19530121.1527858
4.03701725859655e-05 16572560.2569017
4.43062145758388e-05 14061691.0767736
4.86260158006535e-05 11930210.0495089
5.33669923120631e-05 10120932.1913441
5.85702081805666e-05 8585272.87593681
6.42807311728432e-05 7281956.99540622
7.05480231071865e-05 6175921.50636993
7.74263682681127e-05 5237382.43861646
8.49753435908644e-05 4441041.76284724
9.3260334688322e-05 3765413.18875333
0.000102353102189903 3192249.09145162
0.000112332403297803 2706053.42462305
0.000123284673944207 2293667.74205591
0.000135304777457981 1943919.37485102
0.000148496826225446 1647322.4496375
0.000162975083462064 1395823.82663227
0.000178864952905744 1182587.22172474
0.000196304065004027 1001809.78502394
0.000215443469003188 848566.265886036
0.000236448941264541 718676.623829525
0.000259502421139973 608593.565075543
0.00028480358684358 515307.012003619
0.000312571584968824 436262.961442553
0.000343046928631492 369294.569210329
0.000376493580679247 312563.622706493
0.000413201240011534 264510.839178926
0.000453487850812858 223813.661794928
0.000497702356433211 189350.425020921
0.000546227721768434 160169.930310568
0.000599484250318941 135465.617187574
0.000657933224657568 114553.63728944
0.000722080901838546 96854.2430460833
0.000792482898353917 81875.9911535349
0.000869749002617784 69202.3362075097
0.000954548456661834 58480.2537744656
0.00104761575278967 49410.5864909792
0.00114975699539774 41739.8529350288
0.00126185688306602 35253.2982287533
0.00138488637139387 29768.998652038
0.00151991108295293 25132.8608545662
0.00166810053720006 21214.3803030242
0.00183073828029537 17903.0440294651
0.00200923300256505 15105.2801001109
0.00220513073990305 12741.8709632788
0.00242012826479438 10745.7603535303
0.00265608778294669 9060.19406048882
0.00291505306282518 7637.14389888693
0.00319926713779738 6435.97188251611
0.00351119173421513 5422.29811385303
0.00385352859371053 4567.04142745357
0.0042292428743895 3845.60651676445
0.00464158883361278 3237.19525659494
0.00509413801481638 2724.22331398912
0.00559081018251223 2291.82600944051
0.00613590727341317 1927.43982543428
0.00673415065775082 1620.44802571172
0.00739072203352578 1361.88060208895
0.00811130830789687 1144.16025340062
0.00890215085445039 960.887363333185
0.00977009957299226 806.658014658657
0.0107226722201032 676.909985635831
0.01176811952435 567.792444700022
0.0129154966501488 476.05571289057
0.0141747416292681 398.958017493127
0.0155567614393047 334.186630147957
0.0170735264747069 279.791180969527
0.0187381742286039 234.127277878778
0.0205651230834865 195.808846569854
0.0225701971963392 163.667849127434
0.0247707635599171 136.720244906646
0.0271858824273294 114.137231508356
0.0298364724028334 95.2209512999664
0.0327454916287773 79.3839739936263
0.0359381366380463 66.1319717340845
0.0394420605943766 55.0490928801436
0.0432876128108306 45.7856166573333
0.047508101621028 38.0475352103697
0.0521400828799969 31.5877640689064
0.0572236765935022 26.1987281656407
0.0628029144183425 21.7061095893008
0.068926121043497 17.9635762992856
0.075646332755463 14.8483389916287
0.0830217568131974 12.2574069654997
0.0911162756115489 10.1044338549612
0.1 8.31706102086792
};
\label{\figlabel-line1}
\addplot [thick, black, dashed]
table {%
1e-05 711740.79626965
1.09749876549306e-05 643271.206938027
1.20450354025878e-05 581349.818667477
1.32194114846603e-05 525353.549015187
1.45082877849594e-05 474718.342623386
1.59228279334109e-05 428933.601013713
1.74752840000768e-05 387537.136219963
1.91791026167249e-05 350110.599146875
2.10490414451202e-05 316275.338134456
2.31012970008316e-05 285688.647370952
2.53536449397011e-05 258040.368572739
2.78255940220713e-05 233049.812772363
3.05385550883341e-05 210462.972159392
3.35160265093884e-05 190049.994732534
3.67837977182863e-05 171602.897072463
4.03701725859655e-05 154933.492857557
4.43062145758388e-05 139871.51684148
4.86260158006535e-05 126262.92591228
5.33669923120631e-05 113968.360575814
5.85702081805666e-05 102861.751768325
6.42807311728432e-05 92829.0593189476
7.05480231071865e-05 83767.1296664276
7.74263682681127e-05 75582.6615977362
8.49753435908644e-05 68191.2698308178
9.3260334688322e-05 61516.6372195114
0.000102353102189903 55489.7472250006
0.000112332403297803 50048.1890833271
0.000123284673944207 45135.5288101089
0.000135304777457981 40700.7398285288
0.000148496826225446 36697.6875911242
0.000162975083462064 33084.6630955744
0.000178864952905744 29823.960674662
0.000196304065004027 26881.4958755365
0.000215443469003188 24226.4596375404
0.000236448941264541 21831.0053349864
0.000259502421139973 19669.96557487
0.00028480358684358 17720.5959327022
0.000312571584968824 15962.3430752914
0.000343046928631492 14376.6349599897
0.000376493580679247 12946.6910179643
0.000413201240011534 11657.3504266043
0.000453487850812858 10494.9167551272
0.000497702356433211 9447.01742957108
0.000546227721768434 8502.47661020885
0.000599484250318941 7651.20020745187
0.000657933224657568 6884.07188279991
0.000722080901838546 6192.85899053668
0.000792482898353917 5570.12751472013
0.000869749002617784 5009.16514554663
0.000954548456661834 4503.91172025189
0.00104761575278967 4048.89632714395
0.00114975699539774 3639.18043786411
0.00126185688306602 3270.30649319173
0.00138488637139387 2938.25142224092
0.00151991108295293 2639.38462427475
0.00166810053720006 2370.42998707287
0.00183073828029537 2128.43155626994
0.00200923300256505 1910.72250673345
0.00220513073990305 1714.8971002325
0.00242012826479438 1538.78534368935
0.00265608778294669 1380.43008950061
0.00291505306282518 1238.06634403335
0.00319926713779738 1110.10257268512
0.00351119173421513 995.103810067919
0.00385352859371053 891.776402133333
0.0042292428743895 798.954223579675
0.00464158883361278 715.586228837316
0.00509413801481638 640.72520846231
0.00559081018251223 573.517635016331
0.00613590727341317 513.194493594204
0.00673415065775082 459.063002189116
0.00739072203352578 410.499136160071
0.00811130830789687 366.940879276465
0.00890215085445039 327.882131242879
0.00977009957299226 292.867208327449
0.0107226722201032 261.485879796583
0.01176811952435 233.368888358261
0.0129154966501488 208.183907790674
0.0141747416292681 185.631895432429
0.0155567614393047 165.443801280123
0.0170735264747069 147.377599119601
0.0187381742286039 131.215608445705
0.0205651230834865 116.762078935315
0.0225701971963392 103.841011960201
0.0247707635599171 92.2941960872454
0.0271858824273294 81.9794357386748
0.0298364724028334 72.7689541966385
0.0327454916287773 64.547953955108
0.0359381366380463 57.2133190660633
0.0394420605943766 50.6724456129939
0.0432876128108306 44.8421877879096
0.047508101621028 39.6479082620407
0.0521400828799969 35.0226226375145
0.0572236765935022 30.9062287587497
0.0628029144183425 27.2448125581954
0.068926121043497 23.9900229205306
0.075646332755463 21.0985087808125
0.0830217568131974 18.5314123327906
0.0911162756115489 16.2539128204989
0.1 14.234815925393
};
\label{\figlabel-line3}
\addplot [ultra thick, blue]
table {%
0.000132059976775 5247399.0816256
0.000134597669583 1138501.980904
0.000266457400041 569250.9904515
0.000368612154141 254525.734935
0.000432004928221 127262.8674675
0.000653437109663 63631.4337338
0.00109392185408 31815.71686694
0.00228054100492 15907.85843341
0.00239813854189 15360
0.0025720983848 7680
0.00329499256326 3840
0.00405857663354 2156.04858014
0.00456852438756 1920
0.00461446459391 1078.02429007
0.0056613128853 539.012145036
0.00910236800804 269.506072518
0.01378949183 134.7530362588
0.0251277204434 67.3765181294
0.0372614771795 33.6882590646
0.0544651671895 12
};
\label{\figlabel-line0}
\addplot [ultra thick, green!50.0!black]
table {%
1.44639587272e-05 381616.272973868
3.78640122421e-05 167068.489783086
7.69552596262e-05 78409.8317259242
0.000181755811891 34080.502697318
0.000393103387445 15928.4330079298
0.000877573845432 6852.398163241
0.00140973705312 3177.6387935986
0.00329158299043 1340.2591087768
0.00451594669534 610.1295543882
0.00905869618308 245.064777194
0.0132448727133 105.6882590646
0.0533682590909 36
0.0544651671895 12
};
\label{\figlabel-line2}
\addplot [ultra thick, red]
table {%
1.3034378468e-05 383517.986738826
1.39050009416e-05 369609.163757975
1.48798437703e-05 335767.317465405
2.0074898237175e-05 267161.59028684
2.01188108224714e-05 245317.665178866
2.03324265740333e-05 309531.74886987
2.08081775932e-05 227937.783860046
2.54827506018e-05 219745.783860046
2.65519824786e-05 213297.747916658
2.72823539789e-05 206343.33642626
2.91569805826e-05 196547.462867188
3.701686162619e-05 152942.219742479
4.65300789151e-05 145851.211738779
4.69296625133e-05 135816.902038799
4.706852640005e-05 129790.125744773
4.7501381257825e-05 112688.763400133
5.088955632824e-05 101270.727143231
5.31329975684e-05 98046.7091715261
7.21409464253e-05 89854.7091715261
7.35018960717e-05 86377.5034263281
7.48626282901e-05 80226.7666864981
7.53449840032e-05 73763.3843938862
9.9333451235125e-05 65183.3788930986
9.99690003499e-05 58052.1552990853
0.000107236989272 56440.1463132343
0.000108754000751 54701.5434406319
0.0001126193183345 48864.6605372219
0.00014446327997 46816.6605372219
0.000196583722616 40318.7553358594
0.000197616639876333 36763.2998544486
0.0001989819586762 31104.3584289597
0.000199789998581 28595.7810039707
0.000212397094372 26823.0290030457
0.000219984164598 26017.024510123
0.000221974003399 25147.72307382
0.000227750879287 22229.2816221078
0.000230948709816 24401.7667869538
0.000235090033656 21205.2816221078
0.000327992842683 19157.2816221078
0.000329831117772333 16682.7716694665
0.00034922632046225 13951.6621560394
0.000368473400331 13065.2861555774
0.0003845962244095 11606.0654297174
0.000441431160307 11094.0654297174
0.0005753077757775 8350.0045625176
0.000646608734871 7279.6722060396
0.000653394667458 7906.8165622876
0.000660508158657 6736.5509148267
0.00071630841432 5928.8901096585
0.000724210191333 5400.3729236585
0.000726201391158 4623.4324289046
0.00101875768125 4111.4324289046
0.00102049864232 3797.8602507789
0.00138715636771 3304.7056050583
0.00139885001707 3576.2662506638
0.00148806368052 3176.7056050583
0.00149595611671667 2477.7220075841
0.00166427833724 2058.5596069494
0.00246312064836 1802.5596069494
0.00248308743525 1691.6317290194
0.00249925771548 1414.0893595726
0.00269048596292 1010.7259476494
0.00269649508992 1146.5062704521
0.00280521551337 946.7259476494
0.00282000929781 861.1535034983
0.003774509187635 727.3619589384
0.00463375049753 695.3619589384
0.00491080915654 594.61139732298
0.00577584791158 466.61139732298
0.005798797633845 305.47670006348
0.00886445468886 241.47670006348
0.013669101956 209.47670006348
0.024027090706 33.84962500724
0.0243881186303 49.84962500724
0.0245109480602 150.69047798802
0.02535514762155 83.79470570787
};
\label{\figlabel-line4}
\coordinate (legend) at (axis description cs:0.03,0.03);
\end{axis}

\matrix [matrix of nodes,
inner sep=1pt, row sep=1pt,cells={anchor=west},anchor={south west},at={(0.03,0.03)}, anchor=south west, draw=none, fill=none] at (legend) {
\ref{\figlabel-line0} {SL}\\
\ref{\figlabel-line1} {$\epsilon^{-\frac{ 5 }{ 3 }}\log(\epsilon^{-1})$}\\
\ref{\figlabel-line2} {ML}\\
\ref{\figlabel-line3} {$\epsilon^{-1}\log(\epsilon^{-1})$}\\
\ref{\figlabel-line4} {Adaptive ML}\\
};
\end{tikzpicture}

%% file: figures/poisson-kink-3/work-est-vs-error.tex
\begin{tikzpicture}

\begin{axis}[
xlabel={Max Error},
ylabel={Work Estimate},
xmin=1e-05, xmax=0.1,
ymin=10, ymax=10000000,
xmode=log,
ymode=log,
axis on top,
name=\figlabel,
width=\figurewidth,
height=\figureheight,
xtick={1e-06,1e-05,0.0001,0.001,0.01,0.1,1,10},
xticklabels={,$10^{-5}$,$10^{-4}$,$10^{-3}$,$10^{-2}$,$10^{-1}$,,},
ytick={1,10,100,1000,10000,100000,1000000,10000000,100000000},
yticklabels={,$10^{1}$,$10^{2}$,$10^{3}$,$10^{4}$,$10^{5}$,$10^{6}$,$10^{7}$,},
tick pos=both
]
\addplot [thick, black, opacity=0.4, dotted, mark=x, mark size=2, mark options={solid,fill opacity=0}, forget plot]
table {%
0.032164710232 69.1886323727
0.0323385638832 361.8947501008
0.051035633904 18.5754247591
};
\addplot [thick, black, opacity=0.4, dotted, mark=x, mark size=2, mark options={solid,fill opacity=0}, forget plot]
table {%
0.0215099418036 723.7895002021
0.0216397120011 4867.526024689
0.0217242797332 138.3772647455
0.0490792003595 37.1508495182
};
\addplot [thick, black, opacity=0.4, dotted, mark=x, mark size=2, mark options={solid,fill opacity=0}, forget plot]
table {%
0.0113282255642 1447.579000405
0.0114090379439 76914.110944276
0.0114164151911 9735.052049378
0.012577796492 276.754529491
0.0496575860038 74.3016990364
};
\addplot [thick, black, opacity=0.4, dotted, mark=x, mark size=2, mark options={solid,fill opacity=0}, forget plot]
table {%
0.00700020634185 2895.15800081
0.00701893900149 19470.104098755
0.00938574447436 553.509058982
0.0506262399478 148.603398073
};
\addplot [thick, black, opacity=0.4, dotted, mark=x, mark size=2, mark options={solid,fill opacity=0}, forget plot]
table {%
0.00339649414451 307656.44377673
0.00343469517396 38940.208197512
0.00358106706727 5790.316001612
0.00773531370794 1107.018117964
0.0517209057266 297.206796146
};
\addplot [thick, black, opacity=0.4, dotted, mark=x, mark size=2, mark options={solid,fill opacity=0}, forget plot]
table {%
0.0019931085648 615312.88755353
0.00206216284501 77880.41639506
0.0024145667152 11580.63200323
0.00747873832964 2214.036235931
0.0522163261965 594.413592291
};
\addplot [thick, black, opacity=0.4, dotted, mark=x, mark size=2, mark options={solid,fill opacity=0}, forget plot]
table {%
0.0009204159105695 1230625.77511105
0.00106861061652 155760.83279042
0.00177325026714 23161.26400645
0.00745826609476 4428.07247185
0.0526202481683 1188.82718458
};
\addplot [thick, black, opacity=0.4, dotted, mark=x, mark size=2, mark options={solid,fill opacity=0}, forget plot]
table {%
0.000493992319159 2461251.5502122
0.00073802650909 311521.6655798
0.00164834196358 46322.5280129
0.00749324983904 8856.14494371
0.05278678499 2377.65436916
};
\addplot [thick, black, opacity=0.4, dotted, mark=x, mark size=2, mark options={solid,fill opacity=0}, forget plot]
table {%
0.000233040511938 4922503.1004248
0.000598815200992 623043.3311597
0.00162332440566 92645.0560259
0.00752662028315 17712.28988747
0.0528903900769 4755.30873833
};
\addplot [thick, black, opacity=0.4, dotted, mark=x, mark size=2, mark options={solid,fill opacity=0}, forget plot]
table {%
0.0001165553314975 9845006.2008547
0.00056522288294 1246086.66232
0.0016256597444 185290.1120518
0.0075444453459 35424.5797748
0.0529371274145 9510.61747666
};
\addplot [thick, black, dash pattern=on 1pt off 3pt on 3pt off 3pt]
table {%
1e-05 4247662241.04439
1.09749876549306e-05 3497986977.08867
1.20450354025878e-05 2880431665.17057
1.32194114846603e-05 2371743206.54259
1.45082877849594e-05 1952755938.22796
1.59228279334109e-05 1607673978.99932
1.74752840000768e-05 1323479440.72066
1.91791026167249e-05 1089444452.74566
2.10490414451202e-05 896728807.049833
2.31012970008316e-05 738048216.724462
2.53536449397011e-05 607400808.327169
2.78255940220713e-05 499841636.722683
3.05385550883341e-05 411296799.815261
3.35160265093884e-05 338410206.292905
3.67837977182863e-05 278417266.910374
4.03701725859655e-05 229040784.107569
4.43062145758388e-05 188405143.173012
4.86260158006535e-05 154965591.484351
5.33669923120631e-05 127449955.976938
5.85702081805666e-05 104810613.857127
6.42807311728432e-05 86184914.9752186
7.05480231071865e-05 70862570.4671783
7.74263682681127e-05 58258783.0326291
8.49753435908644e-05 47892109.2486144
9.3260334688322e-05 39366221.6346528
0.000102353102189903 32354884.3931478
0.000112332403297803 26589577.3019887
0.000123284673944207 21849301.6303237
0.000135304777457981 17952183.8936959
0.000148496826225446 14748560.8208886
0.000162975083462064 12115284.5958094
0.000178864952905744 9951033.34513978
0.000196304065004027 8172449.68313259
0.000215443469003188 6710961.31467271
0.000236448941264541 5510163.40417582
0.000259502421139973 4523663.60391576
0.00028480358684358 3713308.09492341
0.000312571584968824 3047721.38134859
0.000343046928631492 2501104.43490921
0.000376493580679247 2052245.55481198
0.000413201240011534 1683706.35722152
0.000453487850812858 1381151.9394621
0.000497702356433211 1132799.72700641
0.000546227721768434 928966.011483401
0.000599484250318941 761692.894845177
0.000657933224657568 624441.408115537
0.000722080901838546 511839.087892472
0.000792482898353917 419472.36484652
0.000869749002617784 343715.824014196
0.000954548456661834 281591.801144955
0.00104761575278967 230654.935799257
0.00114975699539774 188897.254048057
0.00126185688306602 154670.137523375
0.00138488637139387 126620.180901583
0.00151991108295293 103636.471121581
0.00166810053720006 84807.2588949574
0.00183073828029537 69384.3529489508
0.00200923300256505 56753.863623767
0.00220513073990305 46412.166180794
0.00242012826479438 37946.154738533
0.00265608778294669 31017.0227732859
0.00291505306282518 25346.941888629
0.00319926713779738 20708.1222479161
0.00351119173421513 16913.8299395077
0.00385352859371053 13811.0121139179
0.0042292428743895 11274.2428841199
0.00464158883361278 9200.75409259546
0.00509413801481638 7506.35707842349
0.00559081018251223 6122.09613558111
0.00613590727341317 4991.50276527335
0.00673415065775082 4068.34318140864
0.00739072203352578 3314.77072690895
0.00811130830789687 2699.81063793624
0.00890215085445039 2198.11756094486
0.00977009957299226 1788.95688358251
0.0107226722201032 1455.36969602644
0.01176811952435 1183.48839254203
0.0129154966501488 961.97583202377
0.0141747416292681 781.565829749759
0.0155567614393047 634.686738629896
0.0170735264747069 515.153151493922
0.0187381742286039 417.913443582548
0.0205651230834865 338.843080883435
0.0225701971963392 274.575431198029
0.0247707635599171 222.3633014388
0.0271858824273294 179.965644661379
0.0298364724028334 145.554881410123
0.0327454916287773 117.641101272089
0.0359381366380463 95.0100842755018
0.0394420605943766 76.6726343707392
0.0432876128108306 61.8231704078335
0.047508101621028 49.805891606787
0.0521400828799969 40.0871391486319
0.0572236765935022 32.2328252190036
0.0628029144183425 25.8900054822274
0.068926121043497 20.7718386545672
0.075646332755463 16.6453142300011
0.0830217568131974 13.3212419476412
0.0911162756115489 10.6460887542957
0.1 8.49532448208878
};
\label{\figlabel-line1}
\addplot [thick, black, dashed]
table {%
1e-05 4976638.83475196
1.09749876549306e-05 4389724.41793334
1.20450354025878e-05 3870999.29773914
1.32194114846603e-05 3412649.84895434
1.45082877849594e-05 3007746.43371076
1.59228279334109e-05 2650144.56302362
1.74752840000768e-05 2334396.99003651
1.91791026167249e-05 2055675.53797336
2.10490414451202e-05 1809701.59565616
2.31012970008316e-05 1592684.32933987
2.53536449397011e-05 1401265.76303051
2.78255940220713e-05 1232471.97172266
3.05385550883341e-05 1083669.71430914
3.35160265093884e-05 952527.906343535
3.67837977182863e-05 836983.398328014
4.03701725859655e-05 735210.583604189
4.43062145758388e-05 645594.412004739
4.86260158006535e-05 566706.431857217
5.33669923120631e-05 497283.524325417
5.85702081805666e-05 436209.030970735
6.42807311728432e-05 382496.0082999
7.05480231071865e-05 335272.372369246
7.74263682681127e-05 293767.722625427
8.49753435908644e-05 257301.657423333
9.3260334688322e-05 225273.414381845
0.000102353102189903 197152.68719205
0.000112332403297803 172471.486926127
0.000123284673944207 150816.930527394
0.000135304777457981 131824.852188607
0.000148496826225446 115174.144920958
0.000162975083462064 100581.749936155
0.000178864952905744 87798.220647335
0.000196304065004027 76603.7962652827
0.000215443469003188 66804.9272351445
0.000236448941264541 58231.2012239875
0.000259502421139973 50732.6241191753
0.00028480358684358 44177.2156099519
0.000312571584968824 38448.8834697236
0.000343046928631492 33445.5446966009
0.000376493580679247 29077.4652603306
0.000413201240011534 25265.7933942708
0.000453487850812858 21941.2642056343
0.000497702356433211 19043.0558951114
0.000546227721768434 16517.7801131789
0.000599484250318941 14318.5909660502
0.000657933224657568 12404.3989470836
0.000722080901838546 10739.1776342485
0.000792482898353917 9291.35238293421
0.000869749002617784 8033.26147557669
0.000954548456661834 6940.68128267995
0.00104761575278967 5992.40795933928
0.00114975699539774 5169.88906113031
0.00126185688306602 4456.89922548041
0.00138488637139387 3839.25474030946
0.00151991108295293 3304.56242052167
0.00166810053720006 2841.99874348366
0.00183073828029537 2442.11566461068
0.00200923300256505 2096.66995042488
0.00220513073990305 1798.47323501001
0.00242012826479438 1541.26033205368
0.00265608778294669 1319.57362342661
0.00291505306282518 1128.66160075072
0.00319926713779738 964.389862429069
0.00351119173421513 823.16306850415
0.00385352859371053 701.856532452069
0.0042292428743895 597.756285251255
0.00464158883361278 508.506585129562
0.00509413801481638 432.063968372926
0.00559081018251223 366.657044317378
0.00613590727341317 310.751332780663
0.00673415065775082 263.018526170884
0.00739072203352578 222.309632624112
0.00811130830789687 187.631521910749
0.00890215085445039 158.126453522158
0.00977009957299226 133.054217199753
0.0107226722201032 111.776560991269
0.01176811952435 93.7436214154563
0.0129154966501488 78.4821051085996
0.0141747416292681 65.5850019643097
0.0155567614393047 54.7026367488334
0.0170735264747069 45.534889908531
0.0187381742286039 37.8244391651294
0.0205651230834865 31.3508918548465
0.0225701971963392 25.9256941073046
0.0247707635599171 21.3877171425766
0.0271858824273294 17.5994334223059
0.0298364724028334 14.4436063290594
0.0327454916287773 11.8204266483465
0.0359381366380463 9.64503755021114
0.0394420605943766 7.84539715335559
0.0432876128108306 6.36043422907291
0.047508101621028 5.13845827487981
0.0521400828799969 4.13579015562423
0.0572236765935022 3.31558385850608
0.0628029144183425 2.64681371322598
0.068926121043497 2.10340475573597
0.075646332755463 1.66348682224711
0.0830217568131974 1.30875550069652
0.0911162756115489 1.02392528506243
0.1 0.796262213560313
};
\label{\figlabel-line3}
\addplot [ultra thick, blue]
table {%
0.0001165553314975 9845006.2008547
0.000232761587034 4922503.1004248
0.000493564532628 2461251.5502122
0.00056522288294 1246086.66232
0.000598815200992 623043.3311597
0.00073802650909 311521.6655798
0.00106861061652 155760.83279042
0.00162332440566 92645.0560259
0.00164834196358 46322.5280129
0.00177325026714 23161.26400645
0.0024145667152 11580.63200323
0.00358106706727 5790.316001612
0.00699311153206 2895.15800081
0.00747873832964 2214.036235931
0.00773531370794 1107.018117964
0.00938574447436 553.509058982
0.012577796492 276.754529491
0.0217242797332 138.3772647455
0.032164710232 69.1886323727
0.0490792003595 37.1508495182
0.051035633904 18.5754247591
};
\label{\figlabel-line0}
\addplot [ultra thick, green!50.0!black]
table {%
2.77827474765e-05 1152962.3334012
5.75099002429e-05 493514.142529296
0.000112277003268 163790.047094071
0.000189863742175 81647.3554463046
0.000630736910547 39496.886904104
0.000919070948055 18531.561945883
0.00138045231564 8048.8994667681
0.00252201901249 3996.3954117929
0.00355017263387 1817.2503308473
0.00893802979483 727.6777903736
0.0112029278899 329.2445790003
0.0491479964041 55.7262742773
0.0493803165062 130.0279733137
0.051035633904 18.5754247591
};
\label{\figlabel-line2}
\addplot [ultra thick, red]
table {%
1.682905020824e-05 1465343.06887342
1.70492165828667e-05 1432620.32537575
1.70829609509e-05 1383076.82631634
1.727061807898e-05 1273207.1080593
1.736302509678e-05 1227807.87249397
1.7799389140975e-05 1193915.68002601
1.786538869205e-05 1168566.43935378
1.82194032139e-05 1148703.29900138
1.856508115925e-05 1125601.46645678
1.87717791392e-05 1051628.68486908
1.89474020089e-05 1018559.05929778
1.90675213093e-05 994661.455494854
2.3235792567e-05 978123.471502334
2.35003881724667e-05 934388.977361361
2.40564354213e-05 924006.324762458
2.48313255168e-05 915117.538082155
2.53234413125333e-05 891141.131022587
2.611400090278e-05 851067.527495727
3.02262617698333e-05 830184.235922585
3.07481399056e-05 820190.498320702
3.11307156634e-05 777015.878295156
3.124116675068e-05 795892.895903336
3.29682610306e-05 768823.878295156
3.30779291260438e-05 698274.344493606
3.33397167697556e-05 644783.705256492
3.38128293446e-05 618603.421923612
3.922086420675e-05 597321.30846969
4.040411529026e-05 570824.765072729
4.050032614895e-05 558264.134559534
4.11935211211143e-05 524204.698363391
4.21888423596e-05 519685.851445811
4.277386761995e-05 513156.533749525
4.3347078465e-05 508088.999329637
4.36207134424e-05 499195.685649433
4.39598908302e-05 486914.172601508
4.42634295686e-05 479882.103246832
4.47561574026e-05 473155.038197176
4.66864988026e-05 468710.644857017
4.8744441424e-05 464349.068050405
5.106364617055e-05 447177.667176095
5.19660353981e-05 434713.311142095
5.22337797112e-05 422764.509240635
5.25823878421e-05 411440.792733585
5.7551844229275e-05 400940.519333712
5.81883983372e-05 390411.028137782
6.52066751877778e-05 358549.069851966
6.55487280999769e-05 321430.248989756
6.61412936914e-05 317427.802184936
6.64338127461e-05 300217.624933135
6.69693097881e-05 309656.133737205
7.435644045065e-05 293348.44634553
7.44428594431e-05 290593.318500134
7.56537529264333e-05 282816.816548881
7.67297344697333e-05 275156.043416231
7.79030973681e-05 272896.619957441
7.8707419815e-05 269146.509283092
7.93225039713e-05 265504.187159583
7.96826933982e-05 261988.152482272
8.0190482274e-05 258624.619957441
9.93595104183e-05 246271.181276414
9.96115338976e-05 241587.29711738
0.0001013383994725 236741.990948784
0.000102775627007 234146.327799052
0.000104142115523 225513.976617019
0.0001055760732915 216715.787545298
0.0001271346248165 210047.524337203
0.000128170533344 207549.089936677
0.0001296658032005 199440.55132751
0.0001309227239555 186107.256498174
0.0001312628377554 171012.214179366
0.000132778786775 165747.468581401
0.000139017122603 162201.964579551
0.000140030996222 157482.710177516
0.000161864210675727 141828.699061826
0.00016361695023 140639.04542396
0.0001647859181715 136219.827603627
0.000166708480837 135073.20916862
0.000206367076476 130977.20916862
0.0002075592408526 122942.825624996
0.00020959627951575 107380.627505207
0.00021043711795 105622.610166555
0.000215839795109 104511.511831505
0.000218486581616 101739.351367437
0.000221478502515 97446.5011488516
0.000222819139759 94330.4121403406
0.000233317096064 92263.1641412656
0.000234430351084 89275.9636659096
0.0002373581483655 84286.8758920729
0.000238005333901286 79124.7694878455
0.0002403864397425 76146.7727827097
0.000243655907728 73514.3999837267
0.000261422112111 71741.6479828017
0.0002638372614615 69843.0855125755
0.000265363881059 68932.5049817041
0.000268249969741 68053.4963123738
0.0002713256933005 66688.8443381222
0.00034035304319725 62002.5557222732
0.000342829138345 58950.9109803232
0.000407766433344 51888.3308053483
0.0004099354345794 45132.4951725422
0.000421078262066 43721.2490566909
0.000422786674327 41222.0393344207
0.000479867118785 38163.1781691127
0.000483363282856 40083.1781691127
0.000497531146255286 35222.5131318682
0.000498172384842 34042.6995313612
0.00069860140118 33030.3936807582
0.000701836986693 31134.8810250614
0.0008877216977065 29671.3766903938
0.00090078833417725 27299.679203122
0.000906178962919857 22677.4389342099
0.000910468585032 21151.6165632459
0.000946374523019 18777.3132322953
0.0009803715917955 17601.2986852389
0.00101185457164667 15618.621837364
0.00103484144939 13823.2210769081
0.00103877776224 14481.3142766551
0.00110193295374 13129.8801513771
0.0011141376555 12539.9733511241
0.001160295898305 11423.6635603664
0.001260940236745 10880
0.00131180501709 10579.4937046998
0.00131695337261 10282.0802952397
0.00135348108296 9697.4556752799
0.00135851087529 9463.073758133
0.00145913469431 8948.7740166547
0.0016874048423 8033.5513429856
0.00169723382203 7751.1234106371
0.0017326181717 7253.5966595442
0.001849454865255 6591.4869830665
0.00185957730776667 5741.0989335429
0.00186588730191 5387.2327926977
0.00198582708167 4799.7801929393
0.00241314955181 4578.1861928242
0.00245187576688 4283.2327926977
0.00251679456459 4155.2327926977
0.00306595322722 3697.6214558629
0.00307740476944 3899.2327926977
0.00338655805440667 3186.9628631771
0.00340626702522 3010.0297927554
0.003545520928785 2734.709492761
0.00356141453672 2474.6195752042
0.003601981847475 2246.6121332869
0.004391029690795 2072.9364745392
0.00473438971969 1724.4076368528
0.00476834606068 2001.7576517077
0.00555048966464 1692.4076368528
0.00634597103084 1439.3311742023
0.0063816242189 1316.5958854162
0.0064201347306 1196.4665352114
0.00652480405925 506
0.00655479476002 684.8487414329
0.00891908090713 442
0.0128186847164 410
0.0217122569854667 130.56842503044
0.021862309007675 171.6992500145
0.0223981635367 102.86917501594
0.0233249075327 66
0.0236130554393 50
};
\label{\figlabel-line4}
\coordinate (legend) at (axis description cs:0.03,0.03);
\end{axis}

\matrix [matrix of nodes,
inner sep=1pt, row sep=1pt,cells={anchor=west},anchor={south west},at={(0.03,0.03)}, anchor=south west, draw=none, fill=none] at (legend) {
\ref{\figlabel-line0} {SL}\\
\ref{\figlabel-line1} {$\epsilon^{-2}\log(\epsilon^{-1})$}\\
\ref{\figlabel-line2} {ML}\\
\ref{\figlabel-line3} {$\epsilon^{-1}\log(\epsilon^{-1})^{4}$}\\
\ref{\figlabel-line4} {Adaptive ML}\\
};
\end{tikzpicture}

%% file: figures/poisson-kink-4/work-est-vs-error.tex
\begin{tikzpicture}

\begin{axis}[
xlabel={Max Error},
ylabel={Work Estimate},
xmin=1e-05, xmax=0.1,
ymin=10, ymax=100000000,
xmode=log,
ymode=log,
axis on top,
name=\figlabel,
width=\figurewidth,
height=\figureheight,
xtick={1e-06,1e-05,0.0001,0.001,0.01,0.1,1,10},
xticklabels={,$10^{-5}$,$10^{-4}$,$10^{-3}$,$10^{-2}$,$10^{-1}$,,},
ytick={1,10,100,1000,10000,100000,1000000,10000000,100000000,1000000000},
yticklabels={,$10^{1}$,$10^{2}$,$10^{3}$,$10^{4}$,$10^{5}$,$10^{6}$,$10^{7}$,$10^{8}$,},
tick pos=both
]
\addplot [thick, black, opacity=0.4, dotted, mark=x, mark size=2, mark options={solid,fill opacity=0}, forget plot]
table {%
0.028850722164 860.964596731
0.0295967561139 120
0.0564662247388 25.8496250072
};
\addplot [thick, black, opacity=0.4, dotted, mark=x, mark size=2, mark options={solid,fill opacity=0}, forget plot]
table {%
0.0192019313784 17729.308694613
0.0192637300731 1721.92919346
0.0208200459224 240
0.0533431109427 51.6992500144
};
\addplot [thick, black, opacity=0.4, dotted, mark=x, mark size=2, mark options={solid,fill opacity=0}, forget plot]
table {%
0.0100960086376 474522.350003194
0.0101660326209 35458.617389116
0.0102601901521 3443.85838692
0.0136774813885 480
0.0518649941625 103.398500029
};
\addplot [thick, black, opacity=0.4, dotted, mark=x, mark size=2, mark options={solid,fill opacity=0}, forget plot]
table {%
0.0061981501287 949044.700006466
0.00628025704927 70917.234778231
0.00641119183914 6887.716773841
0.0114773994563 960
0.0516988407492 206.797000058
};
\addplot [thick, black, opacity=0.4, dotted, mark=x, mark size=2, mark options={solid,fill opacity=0}, forget plot]
table {%
0.00300619987375 1898089.40001133
0.00311908110281 141834.46955628
0.00333851386423 13775.433547661
0.0104501445112 1920
0.0517785643764 413.594000115
};
\addplot [thick, black, opacity=0.4, dotted, mark=x, mark size=2, mark options={solid,fill opacity=0}, forget plot]
table {%
0.00176426575202 3796178.80002177
0.00191623362252 283668.93911357
0.0022329793161 27550.86709542
0.0102950449777 3840
0.0518619939716 827.188000231
};
\addplot [thick, black, opacity=0.4, dotted, mark=x, mark size=2, mark options={solid,fill opacity=0}, forget plot]
table {%
0.000814952629474 7592357.60005345
0.00105902604125 567337.87822605
0.00153439349992 55101.73419075
0.0102773465982 7680
0.0519454543224 1654.37600046
};
\addplot [thick, black, opacity=0.4, dotted, mark=x, mark size=2, mark options={solid,fill opacity=0}, forget plot]
table {%
0.000437595506412 15184715.2001245
0.000779824570117 1134675.7564519
0.00134457874406 110203.46838147
0.010294830759 15360
0.0519833560028 3308.75200092
};
\addplot [thick, black, opacity=0.4, dotted, mark=x, mark size=2, mark options={solid,fill opacity=0}, forget plot]
table {%
0.000206737057054 30369430.40015
0.000659063252098 2269351.5129038
0.00127031301023 220406.9367628
0.0103123898405 30720
0.0520078815272 6617.50400185
};
\addplot [thick, black, opacity=0.4, dotted, mark=x, mark size=2, mark options={solid,fill opacity=0}, forget plot]
table {%
0.00010382158587 60738860.80041
0.000625460792001 4538703.0258086
0.00124953797103 440813.8735266
0.01032191961 61440
0.052019174254 13235.0080037
};
\addplot [thick, black, dash pattern=on 1pt off 3pt on 3pt off 3pt]
table {%
1e-05 69913274343.3201
1.09749876549306e-05 55816145415.4777
1.20450354025878e-05 44558567074.2286
1.32194114846603e-05 35569133994.3212
1.45082877849594e-05 28391319686.2399
1.59228279334109e-05 22660400526.1907
1.74752840000768e-05 18085011324.6011
1.91791026167249e-05 14432403562.1417
2.10490414451202e-05 11516663574.4876
2.31012970008316e-05 9189296838.89229
2.53536449397011e-05 7331703568.28418
2.78255940220713e-05 5849166020.88434
3.05385550883341e-05 4666044060.09946
3.35160265093884e-05 3721936370.19658
3.67837977182863e-05 2968613403.24415
4.03701725859655e-05 2367567045.78651
4.43062145758388e-05 1888053104.32843
4.86260158006535e-05 1505527580.25948
5.33669923120631e-05 1200397587.63912
5.85702081805666e-05 957023661.208743
6.42807311728432e-05 762922906.729107
7.05480231071865e-05 608132600.713329
7.74263682681127e-05 484701963.162567
8.49753435908644e-05 386286313.877716
9.3260334688322e-05 307823007.22835
0.000102353102189903 245272683.289935
0.000112332403297803 195412683.985313
0.000123284673944207 155672128.352935
0.000135304777457981 124000254.870397
0.000148496826225446 98761327.6522763
0.000162975083462064 78650752.6591814
0.000178864952905744 62628128.0161062
0.000196304065004027 49863813.6662283
0.000215443469003188 39696293.4583269
0.000236448941264541 31598152.2089218
0.000259502421139973 25148929.1261799
0.00028480358684358 20013459.4727821
0.000312571584968824 15924596.2516628
0.000343046928631492 12669427.2225603
0.000376493580679247 10078281.0449379
0.000413201240011534 8015958.85911766
0.000453487850812858 6374741.40455195
0.000497702356433211 5068812.61750755
0.000546227721768434 4029813.17153683
0.000599484250318941 3203295.31449643
0.000657933224657568 2545896.56402768
0.000722080901838546 2023086.70423384
0.000792482898353917 1607371.96045121
0.000869749002617784 1276863.71838882
0.000954548456661834 1014137.89778841
0.00104761575278967 805326.046739025
0.00114975699539774 639391.155447091
0.00126185688306602 507550.708032723
0.00138488637139387 402817.085109736
0.00151991108295293 319631.487516556
0.00166810053720006 253572.383108798
0.00183073828029537 201123.331838877
0.00200923300256505 159488.117232809
0.00220513073990305 126443.562682566
0.00242012826479438 100222.364668474
0.00265608778294669 79419.8326208912
0.00291505306282518 62919.6668180868
0.00319926713779738 49834.8954772259
0.00351119173421513 39460.8810669021
0.00385352859371053 31237.9345657276
0.0042292428743895 24721.57738531
0.00464158883361278 19558.8898666222
0.00509413801481638 15469.7033003067
0.00559081018251223 12231.6457836652
0.00613590727341317 9668.25404209196
0.00673415065775082 7639.5240803742
0.00739072203352578 6034.40153622975
0.00811130830789687 4764.81453796003
0.00890215085445039 3760.93302321077
0.00977009957299226 2967.40308292937
0.0107226722201032 2340.35632168236
0.01176811952435 1845.0351562285
0.0129154966501488 1453.90754681353
0.0141747416292681 1145.17057345474
0.0155567614393047 901.562889384547
0.0170735264747069 709.422486544082
0.0187381742286039 557.939254240134
0.0205651230834865 438.562187165622
0.0225701971963392 344.529348628042
0.0247707635599171 270.495253469906
0.0271858824273294 212.235548620448
0.0298364724028334 166.413012736237
0.0327454916287773 130.392189017692
0.0359381366380463 102.092581329464
0.0394420605943766 79.8724218818844
0.0432876128108306 62.4366692658978
0.047508101621028 48.7642063220695
0.0521400828799969 38.0502479487005
0.0572236765935022 29.6607950047265
0.0628029144183425 23.0966260569455
0.068926121043497 17.9648389283238
0.075646332755463 13.9563666984933
0.0830217568131974 10.8282201398106
0.0911162756115489 8.38946814233092
0.1 6.49017347026521
};
\label{\figlabel-line1}
\addplot [thick, black, dashed]
table {%
1e-05 27733888.5974785
1.09749876549306e-05 24104121.025312
1.20450354025878e-05 20946631.0670259
1.32194114846603e-05 18200296.7363664
1.45082877849594e-05 15811868.0028174
1.59228279334109e-05 13734957.4629518
1.74752840000768e-05 11929159.8979768
1.91791026167249e-05 10359284.3211949
2.10490414451202e-05 8994684.19772138
2.31012970008316e-05 7808673.33475226
2.53536449397011e-05 6778016.52705036
2.78255940220713e-05 5882485.42805681
3.05385550883341e-05 5104471.32743016
3.35160265093884e-05 4428647.57299133
3.67837977182863e-05 3841675.29834045
4.03701725859655e-05 3331946.92371613
4.43062145758388e-05 2889362.60176433
4.86260158006535e-05 2505135.39468642
5.33669923120631e-05 2171621.50603154
5.85702081805666e-05 1882172.35905353
6.42807311728432e-05 1631005.72268701
7.05480231071865e-05 1413093.44334541
7.74263682681127e-05 1224063.65248883
8.49753435908644e-05 1060115.5920038
9.3260334688322e-05 917945.436906004
0.000102353102189903 794681.702111317
0.000112332403297803 687829.000856381
0.000123284673944207 595219.080137693
0.000135304777457981 514968.19620496
0.000148496826225446 445440.013246402
0.000162975083462064 385213.313173877
0.000178864952905744 333053.895803585
0.000196304065004027 287890.128436868
0.000215443469003188 248791.673361785
0.000236448941264541 214950.982418192
0.000259502421139973 185667.200630289
0.00028480358684358 160332.167000707
0.000312571584968824 138418.240742926
0.000343046928631492 119467.716258255
0.000376493580679247 103083.620697876
0.000413201240011534 88921.7145640275
0.000453487850812858 76683.5389985036
0.000497702356433211 66110.3736188143
0.000546227721768434 56977.9863742213
0.000599484250318941 49092.0722385343
0.000657933224657568 42284.2909245435
0.000722080901838546 36408.8254497889
0.000792482898353917 31339.3935260885
0.000869749002617784 26966.6525787669
0.000954548456661834 23195.9468939983
0.00104761575278967 19945.35209082
0.00114975699539774 17143.9779460729
0.00126185688306602 14730.495677319
0.00138488637139387 12651.8602079094
0.00151991108295293 10862.2017845958
0.00166810053720006 9321.86466528345
0.00183073828029537 7996.57350713208
0.00200923300256505 6856.71061933889
0.00220513073990305 5876.68944951708
0.00242012826479438 5034.41159023181
0.00265608778294669 4310.79626011357
0.00291505306282518 3689.3726643949
0.00319926713779738 3155.92690089311
0.00351119173421513 2698.19617394297
0.00385352859371053 2305.60403194947
0.0042292428743895 1969.03117270415
0.00464158883361278 1680.61708060323
0.00509413801481638 1433.58838553911
0.00559081018251223 1222.11037679932
0.00613590727341317 1041.15857749738
0.00673415065775082 886.407695189207
0.00739072203352578 754.135620501787
0.00811130830789687 641.140454868272
0.00890215085445039 544.668816962199
0.00977009957299226 462.353910490567
0.0107226722201032 392.162038284212
0.01176811952435 332.346423153177
0.0129154966501488 281.407348266373
0.0141747416292681 238.057761922944
0.0155567614393047 201.193606162536
0.0170735264747069 169.868228021201
0.0187381742286039 143.270318385466
0.0205651230834865 120.704898072627
0.0225701971963392 101.576935485736
0.0247707635599171 85.3772362736177
0.0271858824273294 71.6702940129451
0.0298364724028334 60.0838330143948
0.0327454916287773 50.2998108007578
0.0359381366380463 42.0466793606605
0.0394420605943766 35.0927315975891
0.0432876128108306 29.2403830346896
0.047508101621028 24.3212592913176
0.0521400828799969 20.1919775426591
0.0572236765935022 16.7305254778497
0.0628029144183425 13.8331545048468
0.068926121043497 11.4117153894324
0.075646332755463 9.39137440178568
0.0830217568131974 7.70865658575938
0.0911162756115489 6.30977014419157
0.1 5.14917230506868
};
\label{\figlabel-line3}
\addplot [ultra thick, blue]
table {%
0.00010382158587 60738860.80041
0.000206737057054 30369430.40015
0.000437595506412 15184715.2001245
0.000625460792001 4538703.0258086
0.000659063252098 2269351.5129038
0.000779824570117 1134675.7564519
0.00105902604125 567337.87822605
0.00124953797103 440813.8735266
0.00127031301023 220406.9367628
0.00134457874406 110203.46838147
0.00153439349992 55101.73419075
0.0022329793161 27550.86709542
0.00333851386423 13775.433547661
0.00641119183914 6887.716773841
0.0102601901521 3443.85838692
0.0104501445112 1920
0.0114773994563 960
0.0136774813885 480
0.0208200459224 240
0.0295967561139 120
0.0518649941625 103.398500029
0.0533431109427 51.6992500144
0.0564662247388 25.8496250072
};
\label{\figlabel-line0}
\addplot [ultra thick, green!50.0!black]
table {%
4.38378183888e-05 1513471.99855922
8.33039967375e-05 697420.705528749
0.000171950944015 289395.059014125
0.000660061750766 85382.2357567899
0.000842738934826 38258.7907047359
0.00167966598481 14697.068178722
0.00241062258745 6918.051790997
0.00351338479303 3028.543597135
0.0107445668526 1083.789500202
0.0123733619109 481.8947501014
0.0509587050092 180.9473750506
0.0527863355072 77.5488750216
0.0564662247388 25.8496250072
};
\label{\figlabel-line2}
\addplot [ultra thick, red]
table {%
5.166463392456e-05 1114859.89723887
5.20135367082e-05 1101400.93059029
5.7706200994e-05 1069106.37891003
5.7980304607975e-05 1082206.63951236
6.3070808781975e-05 1032581.12858385
6.3629830322e-05 954140.356467983
6.38128119674727e-05 909303.406310925
6.42126193414e-05 1025490.12058015
6.45657540918667e-05 887558.318688084
6.514230915655e-05 873514.373721379
6.54609337485e-05 870924.021242659
6.58753575629e-05 864765.433054006
6.64041904747e-05 853011.284184412
6.67000993935e-05 850451.0249778
6.712063362655e-05 844373.647241976
6.79076474708e-05 841829.407245873
6.82759345052e-05 838803.798033413
6.894180324015e-05 833259.69666716
6.92918998478e-05 830252.316521701
6.94118091888e-05 827742.363687376
7.0035513574575e-05 817292.724516834
7.05590631134e-05 814329.980915945
7.08431043756e-05 811377.934329183
7.12845008058333e-05 804479.088305369
7.18697233814e-05 802006.935265293
7.19891405109e-05 797625.578296213
7.26023650848e-05 795173.88280244
7.28651290182e-05 792268.021225903
7.346570295132e-05 781197.237837933
7.38800139614333e-05 774091.891417241
7.48168366828e-05 771684.861585926
7.535829793975e-05 766486.601197581
7.65595305171e-05 756452.291497581
7.934722410826e-05 746773.6072116
8.01180722695e-05 745388.663544554
8.05925289719e-05 744008.217479668
8.11911794664e-05 742632.381304444
8.189152365385e-05 739895.016978763
8.30494103435e-05 737177.592394932
8.3664163077475e-05 630740.334523072
8.38409439523692e-05 638846.903643423
9.45226777655e-05 610960.699601374
9.4886613247625e-05 607289.917942282
9.55524797338e-05 604604.579957736
9.62390114112e-05 601942.294973665
9.69429892956e-05 599304.59884673
9.770873976985e-05 580624.685711584
9.89797120225e-05 579329.509472201
0.000100280593018 578041.703321698
0.000101609151162 576761.573718402
0.000102602044138 575489.45372034
0.0001042739187785 561943.69165467
0.000105180947027 557225.644685424
0.000106800784959833 540313.021779576
0.0001068196360515 550660.382070104
0.000107892517542 535636.479718559
0.000108445123697 534410.63197164
0.000108923735235 527521.23770948
0.000116201299295 521383.420304042
0.00011718427798 518589.105016354
0.000118660442074 515809.422051492
0.000120080097512 514618.175862375
0.0001210951463305 505532.485284261
0.000121814529012 498982.354983161
0.000122153364157 492636.413397501
0.00013177976245613 405274.735498703
0.000132260211363 482393.007585001
0.000137386962308 394664.076646867
0.000138286429285 396712.076646867
0.000138981426191 401729.231496853
0.0001411174197545 359004.922219714
0.000141416285912 375916.639771845
0.000142950775949 351854.131354556
0.000143179899814333 341697.486060182
0.000147086520859333 331592.783879111
0.000150913561057 328064.390482481
0.000153682507651 318032.889217894
0.000161356526235929 286572.103907406
0.000161854015827 268144.623044259
0.000162846205595833 252879.581038376
0.000163483286634 260184.90345642
0.00016588218708 249362.746006217
0.000168702923257 245918.048875142
0.000169870306769 242553.307212941
0.000171040267601 239278.242062409
0.000171988023008 236105.271269543
0.000173328237249 233050.816362387
0.000190654625461 227859.898644035
0.000192903627601 225121.580431191
0.000210698079811429 212939.91071948
0.000211145831562 211483.930424368
0.0002157782068618 186946.855367015
0.000216288453707 185523.90545872
0.00022262982363975 168426.526927144
0.000223019972331333 173161.564260966
0.00022647550308425 162061.970295563
0.000231244310738 159553.392870572
0.000249840669419 156653.101385755
0.000252717155200533 137713.833081289
0.000253756249862571 134061.171872273
0.000256495873754 132302.754356197
0.000259896649859 130580.405790659
0.000338490328946 124800.223070653
0.000346592552494333 120187.549917755
0.000352909739106 118570.514128117
0.0003542643043605 115605.088445451
0.000360989619306 113998.45517431
0.000368701837839 112397.251511907
0.00037408188956575 106186.76149244
0.000377758986374 104549.228917191
0.000416320841375 102931.790476372
0.00041874780832375 99616.5395538978
0.000422186624728 98089.3121003358
0.0004412406621815 94584.7144064608
0.000448583492883 93698.3384059988
0.000457289197606 92674.3384059988
0.000461717299636 76688.692212151
0.000463479573614278 77942.980924645
0.000505691831200555 71423.2870792485
0.0005080297727942 69517.5066322892
0.000517407575865 68758.9073651744
0.000524953307809 68002.5050620594
0.00053291639749475 65481.4043057773
0.000535546653815 64213.4707501631
0.00054133953001 63465.19763545
0.000541624773118 62837.7094268408
0.000553936667545 61469.0126475794
0.00056297374972 58628.9919718822
0.0005723394315855 55818.9014536128
0.00058063449875 54484.9266434557
0.0005885053909565 53159.8309883399
0.00059960466074 52439.7296378769
0.000608391303682 51722.9827546765
0.0006282243834955 51009.7168318161
0.000637596038142 50300.0682785355
0.000657579434886 49946.6461495224
0.000660368057151 49594.1846341799
0.000662094272679 48296.6028941958
0.000670635922083 47251.4099809142
0.000680752413720667 43570.3863556123
0.00068465449003525 45045.0751521025
0.00068640633154 45734.1457121329
0.000697662462052 43228.8223361766
0.000699097717867 42888.5040965032
0.000709099572131812 34970.4099651785
0.00074009013215 34299.0754690425
0.000757035519925 33633.50422303
0.000766205202986 32974.0801912789
0.000850270263834 31950.0801912789
0.000867051475432 30722.9287526813
0.000870609530594 29827.2914638353
0.000876805291243 28086.9084230313
0.000885652875215 27245.7230074743
0.000893747462195 26426.9567198393
0.000945793988961 25573.4551408826
0.0009499959849865 24780.2124426706
0.00100439878711571 22609.1290414599
0.00101110823161333 20073.8392986433
0.00101322395762333 20802.1381455853
0.00104314506775 19379.3984456187
0.00105873637318 18694.8188924117
0.00108560765242 18067.6745361627
0.00116815474592 17811.6745361627
0.00118025535796 17487.8804763395
0.00119616523312333 16599.2100711341
0.00121545264835 16281.1800716291
0.00123631120291 15965.2432933377
0.00125651350121 15409.8812772974
0.00127178305999 15098.4392255818
0.00128681520997 14789.4200955636
0.00129831519478 14482.9581588323
0.001337910213155 13878.3247219011
0.00134100728547 13642.4019177576
0.00136116993381 12372.8219656737
0.00150759960148 10797.3224461608
0.0015171688806 10387.9393023415
0.001526964476075 9609.511089842
0.00165447259295 9351.1050899571
0.0018427350117 8474.955666487
0.002123076005635 7911.0718897697
0.00212878627448 7597.499711643
0.00239883444475 7469.499711643
0.00305150512756 6989.5903894313
0.00307912091924 6769.7881999214
0.00309098586283 6554.4946292271
0.00330844289854 6412.1369835649
0.0033214185031575 5607.9349515169
0.00362578735408 5014.7153514811
0.003640959907864 4436.5703596445
0.00366091114612 4324.6156985389
0.00369424566286 5125.5123515388
0.00437704018641 4134.2263919537
0.004721115804445 3950.2817293755
0.00497970201738 3589.9547295995
0.00507997784354 3557.9547295995
0.00536616648322 3653.9547295995
0.00551412912657 977.7922893914
0.005540666360772 863.080362545
0.0055714440701825 784.6873180138
0.00709047350622 720.6873180138
0.0108101508011 688.6873180138
0.0191364893705 672.6873180138
0.0192290086527 616.7099874613
0.0194264722851 561.7594400835
0.0195581412143 507.9360474105
0.0199436140142 455.3619589384
0.0201293567057 404.18906596122
0.0205159086314 354.61139732298
0.02070218203545 271.29612798257
0.0208391967506 150.69047798792
0.0210360630119 225.77745004908
0.0212382346776 122.99122797342
0.0218290756358 83.79470570777
0.0223567165325 67.79470570777
};
\label{\figlabel-line4}
\coordinate (legend) at (axis description cs:0.03,0.03);
\end{axis}

\matrix [matrix of nodes,
inner sep=1pt, row sep=1pt,cells={anchor=west},anchor={south west},at={(0.03,0.03)}, anchor=south west, draw=none, fill=none] at (legend) {
\ref{\figlabel-line0} {SL}\\
\ref{\figlabel-line1} {$\epsilon^{-\frac{ 7 }{ 3 }}\log(\epsilon^{-1})$}\\
\ref{\figlabel-line2} {ML}\\
\ref{\figlabel-line3} {$\epsilon^{-\frac{ 4 }{ 3 }}\log(\epsilon^{-1})^{2}$}\\
\ref{\figlabel-line4} {Adaptive ML}\\
};
\end{tikzpicture}

%% file: figures/poisson-kink-6/work-est-vs-error.tex
\begin{tikzpicture}

\begin{axis}[
xlabel={Max Error},
ylabel={Work Estimate},
xmin=0.0001, xmax=0.1,
ymin=10, ymax=100000000,
xmode=log,
ymode=log,
axis on top,
name=\figlabel,
width=\figurewidth,
height=\figureheight,
xtick={1e-05,0.0001,0.001,0.01,0.1,1,10},
xticklabels={,$10^{-4}$,$10^{-3}$,$10^{-2}$,$10^{-1}$,,},
ytick={1,10,100,1000,10000,100000,1000000,10000000,100000000,1000000000},
yticklabels={,$10^{1}$,$10^{2}$,$10^{3}$,$10^{4}$,$10^{5}$,$10^{6}$,$10^{7}$,$10^{8}$,},
tick pos=both
]
\addplot [thick, black, opacity=0.4, dotted, mark=x, mark size=2, mark options={solid,fill opacity=0}, forget plot]
table {%
0.0219923706361 3242.861659258
0.0231485452804 272.0469357268
0.0473996127045 42
};
\addplot [thick, black, opacity=0.4, dotted, mark=x, mark size=2, mark options={solid,fill opacity=0}, forget plot]
table {%
0.0146468171491 138770.661198911
0.0146825404421 6485.723318515
0.0171486181992 544.093871455
0.0456595617549 84
};
\addplot [thick, black, opacity=0.4, dotted, mark=x, mark size=2, mark options={solid,fill opacity=0}, forget plot]
table {%
0.00773597291035 277541.322398841
0.0079395211684 12971.446637068
0.0129739708107 1088.187742908
0.045076536669 168
};
\addplot [thick, black, opacity=0.4, dotted, mark=x, mark size=2, mark options={solid,fill opacity=0}, forget plot]
table {%
0.00475918310523 555082.64479768
0.0051763040426 25942.893274045
0.0120103618342 2176.375485814
0.0451498063531 336
};
\addplot [thick, black, opacity=0.4, dotted, mark=x, mark size=2, mark options={solid,fill opacity=0}, forget plot]
table {%
0.00232937038259 1110165.28959527
0.00321933670459 51885.78654808
0.0117499134021 4352.75097164
0.045355625829 672
};
\addplot [thick, black, opacity=0.4, dotted, mark=x, mark size=2, mark options={solid,fill opacity=0}, forget plot]
table {%
0.00139560009836 2220330.57919079
0.00268287982911 103771.57309624
0.011787563253 8705.50194327
0.0454708567658 1344
};
\addplot [thick, black, opacity=0.4, dotted, mark=x, mark size=2, mark options={solid,fill opacity=0}, forget plot]
table {%
0.000712175432305 4440661.1583815
0.00244887368137 207543.14619228
0.0118688266345 17411.00388653
0.0455720926749 2688
};
\addplot [thick, black, opacity=0.4, dotted, mark=x, mark size=2, mark options={solid,fill opacity=0}, forget plot]
table {%
0.000477321676295 8881322.3167619
0.0024110537785 415086.2923847
0.0119135687612 34822.00777307
0.0456154902222 5376
};
\addplot [thick, black, opacity=0.4, dotted, mark=x, mark size=2, mark options={solid,fill opacity=0}, forget plot]
table {%
0.0003730803792 17762644.6335615
0.00240481031792 830172.5847693
0.0119444028253 69644.0155461
0.0456429401838 10752
};
\addplot [thick, black, opacity=0.4, dotted, mark=x, mark size=2, mark options={solid,fill opacity=0}, forget plot]
table {%
0.000346112906001 35525289.267034
0.00240622839308 1660345.1695387
0.011959031944 139288.0310921
0.0456554324994 21504
};
\addplot [thick, black, dash pattern=on 1pt off 3pt on 3pt off 3pt]
table {%
0.0001 7626820736.63698
0.000107226722201032 6139483156.75302
0.000114975699539774 4941909385.97196
0.000123284673944207 3977700180.80826
0.000132194114846603 3201424126.30454
0.000141747416292681 2576486515.87892
0.000151991108295294 2073411997.37051
0.000162975083462064 1668460710.31394
0.000174752840000768 1342513202.72012
0.000187381742286039 1080171963.58569
0.000200923300256505 869037523.82696
0.000215443469003188 699125234.191725
0.000231012970008316 562395403.672034
0.000247707635599171 452374782.254358
0.000265608778294669 363851644.446026
0.00028480358684358 292630174.045025
0.000305385550883342 235332626.647209
0.000327454916287773 189239983.886992
0.000351119173421514 152163616.780631
0.000376493580679247 122341928.95756
0.000403701725859656 98357121.8874773
0.000432876128108306 79068168.1514695
0.000464158883361278 63556839.4719748
0.000497702356433211 51084249.1668764
0.000533669923120631 41055862.5911518
0.000572236765935022 32993327.0811228
0.000613590727341318 26511793.5459594
0.000657933224657568 21301660.1706526
0.000705480231071865 17113876.8036932
0.000756463327554629 13748116.2529869
0.000811130830789688 11043253.7643107
0.000869749002617784 8869704.74163855
0.000932603346883221 7123258.39202427
0.001 5720115.55247775
0.00107226722201032 4592895.79665494
0.00114975699539774 3687424.69568678
0.00123284673944207 2960148.97176429
0.00132194114846603 2376056.96874165
0.00141747416292681 1907005.76765842
0.00151991108295294 1530375.52186871
0.00162975083462064 1227987.08279106
0.00174752840000769 985231.463286542
0.00187381742286038 790369.729452942
0.00200923300256505 633969.99689016
0.00215443469003189 508454.715775798
0.00231012970008316 407736.667662224
0.00247707635599171 326926.313225839
0.00265608778294669 262096.52354163
0.0028480358684358 210093.458288735
0.00305385550883342 168384.551825158
0.00327454916287773 134936.336336812
0.00351119173421513 108116.254028343
0.00376493580679247 86613.7550142014
0.00403701725859656 69376.8984742027
0.00432876128108306 55561.4154577895
0.00464158883361278 44489.7876303824
0.00497702356433211 35618.3755658955
0.00533669923120632 28511.0156882998
0.00572236765935022 22817.814990683
0.00613590727341317 18258.1219703306
0.00657933224657568 14606.8526884475
0.00705480231071864 11683.5120486752
0.00756463327554629 9343.37997775807
0.00811130830789687 7470.43636997489
0.00869749002617784 5971.68240031112
0.00932603346883221 4772.58312265626
0.01 3813.4103683185
0.0107226722201032 3046.30843655685
0.0114975699539774 2432.94000540159
0.0123284673944207 1942.59776272031
0.0132194114846603 1550.68981117876
0.0141747416292681 1237.52501943791
0.0151991108295293 987.339046366915
0.0162975083462065 787.51345526818
0.0174752840000768 627.949723852962
0.0187381742286039 500.567495320196
0.0200923300256505 398.90246995336
0.0215443469003188 317.784197359874
0.0231012970008316 253.077931652416
0.0247707635599171 201.477844197319
0.0265608778294669 160.341402637233
0.028480358684358 127.556742532447
0.0305385550883342 101.436477003107
0.0327454916287773 80.6326887866313
0.0351119173421513 64.0688912760553
0.0376493580679247 50.8855810708435
0.0403701725859656 40.3966750609281
0.0432876128108306 32.0546627641093
0.0464158883361278 25.42273578879
0.0497702356433211 20.1525019649146
0.0533669923120631 15.9661687854479
0.0572236765935022 12.6423029002433
0.0613590727341318 10.0044503947017
0.0657933224657568 7.91204520624239
0.0705480231071865 6.25314729365712
0.0756463327554629 4.93864370252928
0.0811130830789687 3.89761897563907
0.0869749002617784 3.07366005898366
0.093260334688322 2.42190785328826
0.1 1.90670518415925
};
\label{\figlabel-line1}
\addplot [thick, black, dashed]
table {%
0.0001 30689951.3453701
0.000107226722201032 26289653.8698154
0.000114975699539774 22517642.5574392
0.000123284673944207 19284552.9802928
0.000132194114846603 16513686.4542586
0.000141747416292681 14139220.7952778
0.000151991108295294 12104673.0657134
0.000162975083462064 10361578.9221138
0.000174752840000768 8868358.13239276
0.000187381742286039 7589340.0946419
0.000200923300256505 6493926.8579681
0.000215443469003188 5555874.30102541
0.000231012970008316 4752674.83782579
0.000247707635599171 4065027.3545634
0.000265608778294669 3476382.08858179
0.00028480358684358 2972549.88688575
0.000305385550883342 2541366.7660209
0.000327454916287773 2172405.97151322
0.000351119173421514 1856730.83246511
0.000376493580679247 1586682.65037303
0.000403701725859656 1355698.67230221
0.000432876128108306 1158155.89575705
0.000464158883361278 989237.051866593
0.000497702356433211 844815.628578174
0.000533669923120631 721357.238220957
0.000572236765935022 615835.01421352
0.000613590727341318 525657.048578772
0.000657933224657568 448604.162804144
0.000705480231071865 382776.545904478
0.000756463327554629 326548.00086407
0.000811130830789688 278526.718729298
0.000869749002617784 237521.652602015
0.000932603346883221 202513.69517766
0.001 172630.976317707
0.00107226722201032 147127.694053789
0.00114975699539774 125365.975634879
0.00123284673944207 106800.336678311
0.00132194114846603 90964.3678281764
0.00141747416292681 77459.3309858485
0.00151991108295294 65944.3923818062
0.00162975083462064 56128.2585579335
0.00174752840000769 47762.0146295164
0.00187381742286038 40632.9927732166
0.00200923300256505 34559.5234090066
0.00215443469003189 29386.442583936
0.00231012970008316 24981.2471163218
0.00247707635599171 21230.8045437123
0.00265608778294669 18038.5382002323
0.0028480358684358 15322.019140818
0.00305385550883342 13010.9063994634
0.00327454916287773 11045.1854460907
0.00351119173421513 9373.66188966111
0.00376493580679247 7952.67363332084
0.00403701725859656 6744.98996638877
0.00432876128108306 5718.87060286172
0.00464158883361278 4847.26155414631
0.00497702356433211 4107.10804702258
0.00533669923120632 3478.76754543285
0.00572236765935022 2945.50837438487
0.00613590727341317 2493.08153424375
0.00657933224657568 2109.35508387907
0.00705480231071864 1784.00200434955
0.00756463327554629 1508.23376777636
0.00811130830789687 1274.57296027508
0.00869749002617784 1076.65927029872
0.00932603346883221 909.083977652515
0.01 767.248783634254
0.0107226722201032 647.245426257037
0.0114975699539774 545.753040918764
0.0123284673944207 459.950668702493
0.0132194114846603 387.442692444886
0.0141747416292681 326.195303981827
0.0151991108295293 274.482382442482
0.0162975083462065 230.83939985842
0.0174752840000768 194.024172459478
0.0187381742286039 162.983448790347
0.0200923300256505 136.824473438511
0.0215443469003188 114.7907913435
0.0231012970008316 96.2416654659723
0.0247707635599171 80.6345726923848
0.0265608778294669 67.5103215133954
0.028480358684358 56.4804021900054
0.0305385550883342 47.2162374780934
0.0327454916287773 39.4400509459603
0.0351119173421513 32.9171117112927
0.0376493580679247 27.449150087509
0.0403701725859656 22.8687690576488
0.0432876128108306 19.0347024503534
0.0464158883361278 15.8277928298655
0.0497702356433211 13.1475809884778
0.0533669923120631 10.9094150224774
0.0572236765935022 9.04200068908139
0.0613590727341318 7.485326430456
0.0657933224657568 6.18890640929799
0.0705480231071865 5.11029338282236
0.0756463327554629 4.21382046548131
0.0811130830789687 3.46953698071099
0.0869749002617784 2.8523088367559
0.093260334688322 2.34105831625375
0.1 1.91812195908563
};
\label{\figlabel-line3}
\addplot [ultra thick, blue]
table {%
0.000346112906001 35525289.267034
0.0003730803792 17762644.6335615
0.000477321676295 8881322.3167619
0.000712175432305 4440661.1583815
0.00139560009836 2220330.57919079
0.00232937038259 1110165.28959527
0.00240481031792 830172.5847693
0.0024110537785 415086.2923847
0.00244887368137 207543.14619228
0.00268287982911 103771.57309624
0.00321933670459 51885.78654808
0.0051763040426 25942.893274045
0.0079395211684 12971.446637068
0.0117499134021 4352.75097164
0.0120103618342 2176.375485814
0.0129739708107 1088.187742908
0.0171486181992 544.093871455
0.0231485452804 272.0469357268
0.045076536669 168
0.0456595617549 84
0.0473996127045 42
};
\label{\figlabel-line0}
\addplot [ultra thick, green!50.0!black]
table {%
0.000381566361572 939702.594321478
0.000451865796974 435158.631860635
0.000585790565792 182886.650630729
0.0024728679842 56750.660015565
0.00266604551364 26753.899178135
0.00311350959574 11755.518759435
0.0113418394376 4256.3285500898
0.0115006079295 1992.1408071818
0.0124560343379 860.0469357268
0.0443133452182 294
0.045512938072 126
0.0473996127045 42
};
\label{\figlabel-line2}
\addplot [ultra thick, red]
table {%
0.0001960709611192 930157.512370898
0.000196854020298 928236.865324251
0.00019804267732275 922620.368767824
0.000198603139171 920702.180480205
0.000199217701188 918784.984935518
0.000201205139752 915053.826808621
0.000202277559396667 910269.000433092
0.00020299445505 906039.886567485
0.000204172098731 904129.27730391
0.000205601422325667 896094.17797283
0.00020695197751575 889980.515700413
0.000207189973371 889028.759823475
0.000208255156965 887126.054005154
0.000209293310891 885224.428006759
0.000210497834136 884274.021921504
0.0002110656993055 880475.129086978
0.0002121289537415 878577.332970156
0.00021355474089 876680.645855164
0.0002144376853995 872980.845004136
0.000214863960154667 872033.761276931
0.0002162747768775 867302.595077653
0.000217025117762 851355.107510917
0.000218028211424 849474.08278071
0.000218802922344 847594.272835331
0.000219540432973 845715.685718652
0.000220985433509333 842900.115710606
0.000222027028759 841024.621569549
0.000223273544902 839150.379272043
0.000223641710546 835405.684499056
0.0002246833374775 831666.101169211
0.000226459466349 829798.248634764
0.0002278963871455 826998.919453819
0.000228448407873 825134.346672231
0.0002295748511238 816760.348182383
0.000230658385783 822339.981227738
0.000231432388934 814903.202696817
0.000232840166556667 811193.071489944
0.000233845033977 809340.106802442
0.000235061101871 807488.557017403
0.000236289876817667 794632.630403126
0.000237021541556 790952.98854724
0.000237641899166 786423.716901189
0.000238861588659333 782756.385864875
0.000240236713054667 779095.367624839
0.000241400224986 768205.674062337
0.0002423824746896 707513.435582674
0.000243209009407945 641181.599327459
0.000244510391002583 611151.51897795
0.000245448292627889 593165.801493489
0.0002514116945725 583905.313806056
0.000252312856879 533679.984230185
0.000253078487053067 488482.560043672
0.000254301324113 484380.14830949
0.000255685449247 479569.486143572
0.00025664855162 475532.056201384
0.000258033353835 471529.394853993
0.000259258378160625 455725.7475764
0.000259564185494 463223.155497493
0.000260665249915 451838.442933728
0.000262555296169 444194.980119529
0.000271821357645 441619.357340977
0.000273146157854 439812.594383321
0.0002745402101 438007.650533865
0.000275706888827 436204.54389164
0.000279031961344 434403.29291865
0.000281613775824333 427270.748087546
0.0002825831778344 420100.153104564
0.0002873768395395 416526.748296986
0.000290928390999 407553.647311435
0.000291518843047429 402560.791363835
0.000292507597082333 400787.691477627
0.000295409021421 399016.776622499
0.000296831862402 393717.4070818
0.000299747804516 388204.9650037
0.00030051986085475 379555.982341807
0.000302247549054 377801.132523964
0.000303130251286 367481.770119335
0.000304198980762 376048.696202855
0.00031196363135225 350756.379875285
0.000312784552774 347311.68274421
0.000316057226580333 337498.905138601
0.000322939430171 335431.657139526
0.000325451023899 334561.396238847
0.000328297896535 333691.785429312
0.000329646255921 332822.829320433
0.000331620727146667 328193.233422545
0.000333927321716 327325.600808605
0.000335208690935 326458.637125063
0.000337222825823 325592.347253361
0.0003404374447565 324021.813116546
0.0003417188045385 322292.647382552
0.000343748703051 321429.103035365
0.000346811421293 320566.258017664
0.000347592665874 319704.117666857
0.000349832932715333 316419.018255745
0.000350657361504 314858.795298912
0.000351790949572 314000.261281758
0.000355189775868 312285.41465097
0.0003567716952615 310730.206877358
0.000359228010425 309874.662906967
0.000362370974331 306460.180020777
0.0003638470968518 304060.834525841
0.0003654546503845 302361.511384948
0.00036716586277 301513.071562011
0.000368858490075 300665.455917088
0.0003703217690545 299124.004712054
0.0003719916421635 297432.961799388
0.00037341143468 296588.718466254
0.000375017896589 295052.865980499
0.000376239087398 294210.356145669
0.000380563621461 292678.502386834
0.000383259838497 291837.759219689
0.0003853494346595 290309.994138881
0.000386195906104 287396.79875114
0.000396178762489 285624.046750215
0.000397553316306 283947.076594123
0.00040059381862 282273.798829105
0.000403765067384 281438.568017083
0.000406319199868 280604.288796572
0.000415648338658 277865.970583732
0.0004223096414522 268153.910151643
0.000423775785566 271999.648504207
0.000425732040548 266504.486151582
0.0004282857708295 264859.35551192
0.000430738152901 264038.43195043
0.000432783000265 263218.62043654
0.00043434641816 262399.934527219
0.000437876509574 261582.38811427
0.000439393788988333 259272.695280622
0.000442367915847333 252482.821702754
0.000443522961622 243331.794557612
0.000445484308653529 221905.480924714
0.00044702992839225 212329.713797406
0.0004495685620275 201027.443952424
0.000451297117691429 190694.068017707
0.000453790036123 188692.737344096
0.000455738948988 184745.7538764
0.000457718227683 182802.101555052
0.000459946523051 180879.867288781
0.000460720303909 178980.370147959
0.0004780664734075 175817.2910331
0.000504233355798 170925.743338311
0.0005060332322185 167376.051244563
0.000509468836734 165653.702679025
0.000509686034308 163971.331847919
0.000512997822149 162333.799272666
0.000544860367978 161194.938107356
0.000567859081628 152013.771804808
0.000569800301941375 153540.999258372
0.000596597141971 150980.14780527
0.000599072725923 149523.550111393
0.0006276311424835 147268.015004509
0.000699834225015286 144081.44821375
0.0007016768976054 134945.043747829
0.000704293847567667 133033.838669812
0.000828017262612 130117.480664288
0.000829293815150857 123364.888917857
0.0008330600142377 117551.047400093
0.000837256138046875 108185.507146556
0.000840622521942857 106231.732980698
0.000856196724845 103676.961582524
0.0008590828447325 102739.32771975
0.0009101524214365 98828.4860008578
0.000914848168306 102128.944701901
0.0009181191978315 100293.553872361
0.00095230470152 98311.6740010878
0.000954761222611 97432.4652430548
0.000959790925001 96571.2909602818
0.000981163360295928 67219.120443044
0.000985524338298875 64857.2567908983
0.000991251030629584 73397.4922903951
0.000994346231230176 64348.4273887813
0.000998799281019742 74190.7349886071
0.00100165426563 63843.7486460003
0.00100749973353 63343.4159776013
0.00101138130559 62847.6441132453
0.00110563288523667 61501.7512642734
0.0012457397417 60034.3184125125
0.00125521263354 59565.5014811283
0.00134898317914 58743.8743344939
0.00135611294534 57840.6943410302
0.00147141347477 57492.2551402462
0.00147854612149 57052.6507612258
0.00151228225734 18594.1566325836
0.00152762647136 18339.7419315218
0.00153169773782 55658.7789111644
0.00154298529365 56285.9232674124
0.0015526802791 56716.5104088018
0.00160323936772857 18978.5573659691
0.00160776268993437 39404.8526513443
0.00161595835644875 36585.7837322733
0.00162365803446333 33429.8672149833
0.00162906210730889 19234.5573659691
0.0016397705949975 20173.7873616629
0.00164430546038833 19576.8471425713
0.0016494392719275 21457.4472225528
0.00166219183754 18083.7419315218
0.001665512593 17831.4025601158
0.0016776084294 17581.2362259251
0.00169713595455 17333.3502937551
0.00172505790934 17087.8632924538
0.00173927036233 16844.9067522784
0.00175821376495 16604.6274689883
0.00185990376843 16443.6510453259
0.001874236136795 15971.8054370428
0.00200837974719 15819.2096825785
0.00202386821951 15588.0426497164
0.00204695500937 15360.3619751951
0.00205922574622 14144.0113213184
0.00207299838124 14774.2928178115
0.00207363653959 14558.9992471174
0.00208760631721667 13945.7006467692
0.00215259003478625 12492.5304852855
0.00217377843189 13816.4976468269
0.00252399607583 12273.0716932381
0.00285229592348 12057.0787373184
0.00288848750212 11885.9338490166
0.00292826364476 11775.1368489589
0.003426858180955 11569.5654700316
0.00393458200465 10991.4023032344
0.00398398449074 11493.2675927998
0.00398607940283 11369.3246267147
0.00403096434561 11055.4023032344
0.0041822005356 10959.4023032344
0.004190465814158 10132.1926451054
0.00421124559292667 9726.3179807244
0.0042509789662 9562.4725284044
0.00427443175164 9280.32726702779
0.004282512831564 8352.5523969709
0.00429752568439 8513.7631232185
0.00432481144988 7871.7791847402
0.00434296669641667 7596.9248809222
0.00440247627817 7324.3253315446
0.00442786081377 7166.0892902593
0.00444601752833 7008.3885532654
0.00447335183174 6851.2358301724
0.00450357449599 6426.6737408889
0.00451522277149667 2126.2771150159
0.00453942048376 2030.8253991673
0.00455911329184333 1899.4555514239
0.004578781292512 1678.808609856
0.00459771037051 1593.2361657053
0.004622683267108 2805.6806532669
0.00463707588289111 3073.475672742
0.0046610916286025 2933.6806532669
0.00551433624408 1529.2361657053
0.00786218126462 1497.2361657053
0.00800750761571 1433.1359823588
0.0085054816497 1369.5323070931
0.00888182786324 1306.4474642462
0.00921071739139 1243.9058806947
0.00944671915919 1181.9343976517
0.00974764103562 1141.6902917364
0.00978738616827 1080.3185414132
0.0100046770699 1019.579406371
0.0102440683766 862.0013612829
0.01027439170435 826.411949867
0.0103644001146 921.3606469327
0.0105621558774 767.8098334443
0.0110994895487 710.0180752295
0.0113329095418 653.0979065979
0.0117176084402 597.1205760453
0.0118040737031 542.1700286676
0.012124501541 488.3466359947
0.0122354578118 435.7725475225
0.0125027666171 122.99122797332
0.0125966541808 384.59965454522
0.012662267381 209.47670006368
0.0127585080144 138.99122797332
0.0128663968753 254.99537799697
0.0129423751977 302.72123592168
0.0130108248144 335.02198590718
0.0160678013788 106.99122797332
};
\label{\figlabel-line4}
\coordinate (legend) at (axis description cs:0.03,0.03);
\end{axis}

\matrix [matrix of nodes,
inner sep=1pt, row sep=1pt,cells={anchor=west},anchor={south west},at={(0.03,0.03)}, anchor=south west, draw=none, fill=none] at (legend) {
\ref{\figlabel-line0} {SL}\\
\ref{\figlabel-line1} {$\epsilon^{-3}\log(\epsilon^{-1})$}\\
\ref{\figlabel-line2} {ML}\\
\ref{\figlabel-line3} {$\epsilon^{-2}\log(\epsilon^{-1})^{2}$}\\
\ref{\figlabel-line4} {Adaptive ML}\\
};
\end{tikzpicture}

%% file: figures/poisson-kink-2/total-time-vs-error.tex
\begin{tikzpicture}

\begin{axis}[
xlabel={Max Error},
ylabel={Time [s.]},
xmin=1e-05, xmax=0.1,
ymin=0.001, ymax=1000,
xmode=log,
ymode=log,
axis on top,
name=\figlabel,
width=\figurewidth,
height=\figureheight,
xtick={1e-06,1e-05,0.0001,0.001,0.01,0.1,1,10},
xticklabels={,$10^{-5}$,$10^{-4}$,$10^{-3}$,$10^{-2}$,$10^{-1}$,,},
ytick={0.0001,0.001,0.01,0.1,1,10,100,1000,10000},
yticklabels={,$10^{-3}$,$10^{-2}$,$10^{-1}$,$10^{0}$,$10^{1}$,$10^{2}$,$10^{3}$,},
tick pos=both
]
\addplot [thick, black, opacity=0.4, dotted, mark=x, mark size=2, mark options={solid,fill opacity=0}, forget plot]
table {%
0.0369623451546 0.0107380000000001
0.0372614771795 0.00540699999999994
0.0544651671895 0.00323100000000021
};
\addplot [thick, black, opacity=0.4, dotted, mark=x, mark size=2, mark options={solid,fill opacity=0}, forget plot]
table {%
0.0245481961676 0.0412470000000003
0.0247423094662 0.0134560000000004
0.0251277204434 0.00659800000000033
0.0546298039869 0.00394800000000006
};
\addplot [thick, black, opacity=0.4, dotted, mark=x, mark size=2, mark options={solid,fill opacity=0}, forget plot]
table {%
0.01294146219175 0.0516829999999999
0.0132315526855 0.0155589999999999
0.01378949183 0.00693899999999981
0.0578140123296 0.00382199999999977
};
\addplot [thick, black, opacity=0.4, dotted, mark=x, mark size=2, mark options={solid,fill opacity=0}, forget plot]
table {%
0.00794432726001667 0.0701039999999997
0.00835749102082 0.0191599999999996
0.00910236800804 0.00796000000000019
0.0599827769776 0.00405100000000003
};
\addplot [thick, black, opacity=0.4, dotted, mark=x, mark size=2, mark options={solid,fill opacity=0}, forget plot]
table {%
0.00385477527667 0.169318000000001
0.00456852438756 0.0585730000000004
0.0056613128853 0.0201300000000004
0.0620678979094 0.00973300000000021
};
\addplot [thick, black, opacity=0.4, dotted, mark=x, mark size=2, mark options={solid,fill opacity=0}, forget plot]
table {%
0.00226099166803 1.001119
0.00228054100492 0.242481
0.00329499256326 0.0579010000000002
0.00461446459391 0.0198119999999997
0.062948013559 0.00787199999999988
};
\addplot [thick, black, opacity=0.4, dotted, mark=x, mark size=2, mark options={solid,fill opacity=0}, forget plot]
table {%
0.0010445467109675 2.915033
0.00109392185408 0.697092
0.0025720983848 0.163286
0.00405857663354 0.0660989999999999
0.0636452477121 0.0243819999999999
};
\addplot [thick, black, opacity=0.4, dotted, mark=x, mark size=2, mark options={solid,fill opacity=0}, forget plot]
table {%
0.0005609622120305 6.420215
0.000653437109663 1.669795
0.00239813854189 0.372075999999999
0.00392095687905 0.102123
0.0639281712417 0.0352379999999999
};
\addplot [thick, black, opacity=0.4, dotted, mark=x, mark size=2, mark options={solid,fill opacity=0}, forget plot]
table {%
0.000264467681055333 89.291871
0.000266457400041 19.253795
0.000432004928221 4.526729
0.00233503884306 1.082113
0.00386412368807 0.29695
0.0641029550006 0.100855
};
\addplot [thick, black, opacity=0.4, dotted, mark=x, mark size=2, mark options={solid,fill opacity=0}, forget plot]
table {%
0.000132059976775 248.00619
0.000134597669583 54.466901
0.000368612154141 12.505615
0.00231851851758 2.823467
0.00384578519125 0.95441
0.0641815074831 0.32309
};
\addplot [thick, black, dash pattern=on 1pt off 3pt on 3pt off 3pt]
table {%
1e-05 9271.77486779502
1.09749876549306e-05 7875.88023928544
1.20450354025878e-05 6689.69789185285
1.32194114846603e-05 5681.78253624554
1.45082877849594e-05 4825.3955892689
1.59228279334109e-05 4097.80182728204
1.74752840000768e-05 3479.67098581935
1.91791026167249e-05 2954.56866823056
2.10490414451202e-05 2508.5232531571
2.31012970008316e-05 2129.65747174137
2.53536449397011e-05 1807.87501204218
2.78255940220713e-05 1534.5939439349
3.05385550883341e-05 1302.51998005133
3.35160265093884e-05 1105.45362878094
3.67837977182863e-05 938.126181022754
4.03701725859655e-05 796.060226250004
4.43062145758388e-05 675.451035115245
4.86260158006535e-05 573.065691963135
5.33669923120631e-05 486.157325434797
5.85702081805666e-05 412.392180935958
6.42807311728432e-05 349.787615398275
7.05480231071865e-05 296.65938125738
7.74263682681127e-05 251.576810366794
8.49753435908644e-05 213.324716019397
9.3260334688322e-05 180.871007768121
0.000102353102189903 153.339163930874
0.000112332403297803 129.984834452655
0.000123284673944207 110.175955517975
0.000135304777457981 93.375849800348
0.000148496826225446 79.1288649210984
0.000162975083462064 67.0481696255281
0.000178864952905744 56.805384122499
0.000196304065004027 48.1217694648903
0.000215443469003188 40.760742042138
0.000236448941264541 34.5215142921604
0.000259502421139973 29.2336925374322
0.00028480358684358 24.7526881908894
0.000312571584968824 20.9558201271783
0.000343046928631492 17.7390043397817
0.000376493580679247 15.0139425865493
0.000413201240011534 12.7057349750569
0.000453487850812858 10.7508527037705
0.000497702356433211 9.09541675191162
0.000546227721768434 7.69373645260261
0.000599484250318941 6.50706880504639
0.000657933224657568 5.50256526479759
0.000722080901838546 4.65237775197883
0.000792482898353917 3.93289986771855
0.000869749002617784 3.32412292152849
0.000954548456661834 2.80908944237938
0.00104761575278967 2.37342945516059
0.00114975699539774 2.00496702114951
0.00126185688306602 1.69338642484965
0.00138488637139387 1.42994899006678
0.00151991108295293 1.20725286787621
0.00166810053720006 1.01902929432681
0.00183073828029537 0.859969797045965
0.00200923300256505 0.725579663471505
0.00220513073990305 0.612053691441661
0.00242012826479438 0.516170844193927
0.00265608778294669 0.435204942498732
0.00291505306282518 0.366848960317998
0.00319926713779738 0.309150858618889
0.00351119173421513 0.260459204636845
0.00385352859371053 0.219377089337645
0.0042292428743895 0.184723081186511
0.00464158883361278 0.155498145635476
0.00509413801481638 0.130857622121879
0.00559081018251223 0.110087488192482
0.00613590727341317 0.0925842573346197
0.00673415065775082 0.0778379563554299
0.00739072203352578 0.0654177123762715
0.00811130830789687 0.054959550972762
0.00890215085445039 0.0461560676201029
0.00977009957299226 0.0387476860365087
0.0107226722201032 0.0325152606454839
0.01176811952435 0.0272738173815181
0.0129154966501488 0.0228672584462893
0.0141747416292681 0.0191638832350945
0.0155567614393047 0.0160526002187573
0.0170735264747069 0.01343972369825
0.0187381742286039 0.0112462655685237
0.0205651230834865 0.00940564597659165
0.0225701971963392 0.00786175841188917
0.0247707635599171 0.00656733464269739
0.0271858824273294 0.00548256328108698
0.0298364724028334 0.00457392284961064
0.0327454916287773 0.00381319623029719
0.0359381366380463 0.00317663846532524
0.0394420605943766 0.00264427418900315
0.0432876128108306 0.00219930462102582
0.047508101621028 0.00182760714206535
0.0521400828799969 0.00151731308992848
0.0572236765935022 0.0012584516301467
0.0628029144183425 0.00104264943031178
0.068926121043497 0.000862877454744097
0.075646332755463 0.000713237539274606
0.0830217568131974 0.000588782542403518
0.0911162756115489 0.000485364830540218
0.1 0.000399508668266867
};
\label{\figlabel-line1}
\addplot [thick, black, dashed]
table {%
1e-05 75.0486661906127
1.09749876549306e-05 67.8289713510176
1.20450354025878e-05 61.2997438250248
1.32194114846603e-05 55.3952834216308
1.45082877849594e-05 50.0561139909785
1.59228279334109e-05 45.2283960806144
1.74752840000768e-05 40.863394827264
1.91791026167249e-05 36.9169979055332
2.10490414451202e-05 33.3492788391229
2.31012970008316e-05 30.1241014191735
2.53536449397011e-05 27.2087613724215
2.78255940220713e-05 24.5736617827813
3.05385550883341e-05 22.1920190972028
3.35160265093884e-05 20.0395968433514
3.67837977182863e-05 18.0944644556445
4.03701725859655e-05 16.3367788500453
4.43062145758388e-05 14.7485866090974
4.86260158006535e-05 13.3136448391118
5.33669923120631e-05 12.0172589431072
5.85702081805666e-05 10.8461357178098
6.42807311728432e-05 9.78825033232587
7.05480231071865e-05 8.83272588143135
7.74263682681127e-05 7.96972432910465
8.49753435908644e-05 7.19034776911701
9.3260334688322e-05 6.48654903028354
0.000102353102189903 5.85105074532332
0.000112332403297803 5.27727208506995
0.000123284673944207 4.75926243480824
0.000135304777457981 4.29164135751605
0.000148496826225446 3.86954425041873
0.000162975083462064 3.48857315711316
0.000178864952905744 3.14475224812979
0.000196304065004027 2.83448752866367
0.000215443469003188 2.55453037376446
0.000236448941264541 2.30194452893235
0.000259502421139973 2.07407624818796
0.00028480358684358 1.8685272726003
0.000312571584968824 1.683130380268
0.000343046928631492 1.51592726412691
0.000376493580679247 1.36514851694984
0.000413201240011534 1.22919552373357
0.000453487850812858 1.10662408053874
0.000497702356433211 0.996129575942102
0.000546227721768434 0.896533586745961
0.000599484250318941 0.806771753616124
0.000657933224657568 0.725882815025119
0.000722080901838546 0.652998689385569
0.000792482898353917 0.587335505681767
0.000869749002617784 0.528185492347914
0.000954548456661834 0.474909642691203
0.00104761575278967 0.426931082901003
0.00114975699539774 0.383729075697381
0.00126185688306602 0.344833599022129
0.00138488637139387 0.309820444925443
0.00151991108295293 0.278306789008039
0.00166810053720006 0.249947185492866
0.00183073828029537 0.224429947269106
0.00200923300256505 0.201473874115826
0.00220513073990305 0.180825295811538
0.00242012826479438 0.16225540000351
0.00265608778294669 0.145557817578239
0.00291505306282518 0.130546440870293
0.00319926713779738 0.117053452396485
0.00351119173421513 0.104927543929209
0.00385352859371053 0.0940323076478734
0.0042292428743895 0.0842447828497027
0.00464158883361278 0.0754541432781168
0.00509413801481638 0.0675605115539575
0.00559081018251223 0.0604738884863128
0.00613590727341317 0.0541131862083403
0.00673415065775082 0.0484053551409733
0.00739072203352578 0.0432845957442355
0.00811130830789687 0.0386916468816214
0.00890215085445039 0.034573143406256
0.00977009957299226 0.0308810362861598
0.0107226722201032 0.0275720692269797
0.01176811952435 0.0246073063304329
0.0129154966501488 0.0219517058512413
0.0141747416292681 0.0195737355897764
0.0155567614393047 0.0174450258867471
0.0170735264747069 0.0155400565743466
0.0187381742286039 0.0138358745892501
0.0205651230834865 0.0123118392702313
0.0225701971963392 0.0109493926501637
0.0247707635599171 0.00973185231166438
0.0271858824273294 0.00864422461026904
0.0298364724028334 0.00767303628113861
0.0327454916287773 0.00680618263706869
0.0359381366380463 0.00603279073891847
0.0394420605943766 0.00534309607627267
0.0432876128108306 0.00472833143777894
0.047508101621028 0.00418062677860976
0.0521400828799969 0.00369291900818177
0.0572236765935022 0.00325887072580756
0.0628029144183425 0.00287279702642079
0.068926121043497 0.00252959958387147
0.075646332755463 0.002224707296406
0.0830217568131974 0.00195402284861764
0.0911162756115489 0.00171387460709052
0.1 0.00150097332381225
};
\label{\figlabel-line3}
\addplot [ultra thick, blue]
table {%
0.000132059976775 248.00619
0.000134597669583 54.466901
0.000266457400041 19.253795
0.000368612154141 12.505615
0.000432004928221 4.526729
0.000653437109663 1.669795
0.00109392185408 0.697092
0.00228054100492 0.242481
0.0025720983848 0.163286
0.00329499256326 0.0579010000000002
0.00461446459391 0.0198119999999997
0.00835749102082 0.0191599999999996
0.00910236800804 0.00796000000000019
0.01378949183 0.00693899999999981
0.0251277204434 0.00659800000000033
0.0372614771795 0.00540699999999994
0.0544651671895 0.00323100000000021
};
\label{\figlabel-line0}
\addplot [ultra thick, green!50.0!black]
table {%
1.44639587272e-05 111.27499
3.78640122421e-05 12.150845
7.69552596262e-05 4.504363
0.000181755811891 1.636557
0.000393103387445 0.804919
0.000877573845432 0.370133000000001
0.00140973705312 0.187424
0.00329158299043 0.0936770000000002
0.00451594669534 0.0559309999999995
0.00905869618308 0.0336440000000002
0.0132448727133 0.0191469999999994
0.0533682590909 0.0106029999999995
0.0544651671895 0.00452799999999964
};
\label{\figlabel-line2}
\addplot [ultra thick, red]
table {%
1.3034378468e-05 39.7433390000003
1.39050009416e-05 39.428929
1.48798437703e-05 37.514525
2.0074898237175e-05 32.112751
2.01188108224714e-05 31.4950820000001
2.03324265740333e-05 35.535266
2.08081775932e-05 30.469507
2.54827506018e-05 29.7638200000001
2.65519824786e-05 29.6559570000002
2.72823539789e-05 29.377575
2.91569805826e-05 28.9404200000001
3.701686162619e-05 9.02992299999995
4.65300789151e-05 8.45554099999998
4.69296625133e-05 8.05957800000001
4.706852640005e-05 7.86874799999999
4.7501381257825e-05 6.95011400000001
5.088955632824e-05 6.05008800000001
5.31329975684e-05 5.96019900000002
7.21409464253e-05 4.69563200000002
7.35018960717e-05 4.59862000000001
7.48626282901e-05 4.32659700000003
7.53449840032e-05 4.08647000000003
9.9333451235125e-05 3.69097900000004
9.99690003499e-05 3.40787400000001
0.000107236989272 3.33596300000001
0.000108754000751 3.26762500000002
0.0001126193183345 3.09652400000001
0.00014446327997 2.97969500000002
0.000196583722616 2.45209000000002
0.000197616639876333 2.09875800000002
0.0001989819586762 1.86485600000002
0.000199789998581 1.79290100000002
0.000212397094372 1.73492100000001
0.000219984164598 1.68358900000002
0.000221974003399 1.69446100000002
0.000227750879287 1.56158600000002
0.000230948709816 1.60147400000002
0.000235090033656 1.49597700000001
0.000327992842683 1.36698300000002
0.000329831117772333 1.17290000000001
0.00034922632046225 1.04703000000001
0.000368473400331 0.993942000000001
0.0003845962244095 0.952699
0.000441431160307 0.917884
0.0005753077757775 0.739084
0.000646608734871 0.698566999999998
0.000653394667458 0.717620000000002
0.000660508158657 0.679904999999996
0.00071630841432 0.627628999999997
0.000724210191333 0.611706999999992
0.000726201391158 0.514577999999994
0.00101875768125 0.485975999999998
0.00102049864232 0.474770000000004
0.00138715636771 0.436615
0.00139885001707 0.472827999999997
0.00148806368052 0.421086
0.00149595611671667 0.348662
0.00166427833724 0.299562
0.00246312064836 0.300625000000001
0.00248308743525 0.271229000000001
0.00249925771548 0.245027
0.00269048596292 0.192085
0.00269649508992 0.204607
0.00280521551337 0.200637000000002
0.00282000929781 0.170570000000002
0.003774509187635 0.143161
0.00463375049753 0.127345000000002
0.00491080915654 0.117097
0.00577584791158 0.102731
0.005798797633845 0.0707090000000001
0.00886445468886 0.0666060000000015
0.013669101956 0.0572540000000012
0.024027090706 0.0128039999999998
0.0243881186303 0.0169570000000006
0.0245109480602 0.0416589999999999
0.02535514762155 0.0243340000000001
};
\label{\figlabel-line4}
\coordinate (legend) at (axis description cs:0.03,0.03);
\end{axis}

\matrix [matrix of nodes,
inner sep=1pt, row sep=1pt,cells={anchor=west},anchor={south west},at={(0.03,0.03)}, anchor=south west, draw=none, fill=none] at (legend) {
\ref{\figlabel-line0} {SL}\\
\ref{\figlabel-line1} {$\epsilon^{-\frac{ 5 }{ 3 }}\log(\epsilon^{-1})$}\\
\ref{\figlabel-line2} {ML}\\
\ref{\figlabel-line3} {$\epsilon^{-1}\log(\epsilon^{-1})$}\\
\ref{\figlabel-line4} {Adaptive ML}\\
};
\end{tikzpicture}

%% file: figures/poisson-kink-3/total-time-vs-error.tex
\begin{tikzpicture}

\begin{axis}[
xlabel={Max Error},
ylabel={Time [s.]},
xmin=1e-05, xmax=0.1,
ymin=0.001, ymax=1000,
xmode=log,
ymode=log,
axis on top,
name=\figlabel,
width=\figurewidth,
height=\figureheight,
xtick={1e-06,1e-05,0.0001,0.001,0.01,0.1,1,10},
xticklabels={,$10^{-5}$,$10^{-4}$,$10^{-3}$,$10^{-2}$,$10^{-1}$,,},
ytick={0.0001,0.001,0.01,0.1,1,10,100,1000,10000},
yticklabels={,$10^{-3}$,$10^{-2}$,$10^{-1}$,$10^{0}$,$10^{1}$,$10^{2}$,$10^{3}$,},
tick pos=both
]
\addplot [thick, black, opacity=0.4, dotted, mark=x, mark size=2, mark options={solid,fill opacity=0}, forget plot]
table {%
0.032164710232 0.00728700000000027
0.0323385638832 0.0244650000000002
0.051035633904 0.00365700000000024
};
\addplot [thick, black, opacity=0.4, dotted, mark=x, mark size=2, mark options={solid,fill opacity=0}, forget plot]
table {%
0.0215099418036 0.0278199999999997
0.0216397120011 0.162038
0.0217242797332 0.00872500000000009
0.0490792003595 0.00579700000000005
};
\addplot [thick, black, opacity=0.4, dotted, mark=x, mark size=2, mark options={solid,fill opacity=0}, forget plot]
table {%
0.0113282255642 0.0385920000000002
0.0114090379439 3.045816
0.0114164151911 0.256722
0.012577796492 0.0106710000000003
0.0496575860038 0.00478000000000045
};
\addplot [thick, black, opacity=0.4, dotted, mark=x, mark size=2, mark options={solid,fill opacity=0}, forget plot]
table {%
0.00700020634185 0.050972
0.00701893900149 0.313565
0.00938574447436 0.012262
0.0506262399478 0.00510299999999986
};
\addplot [thick, black, opacity=0.4, dotted, mark=x, mark size=2, mark options={solid,fill opacity=0}, forget plot]
table {%
0.00339649414451 6.094782
0.00343469517396 0.702338
0.00358106706727 0.110535
0.00773531370794 0.0250140000000001
0.0517209057266 0.00771400000000022
};
\addplot [thick, black, opacity=0.4, dotted, mark=x, mark size=2, mark options={solid,fill opacity=0}, forget plot]
table {%
0.0019931085648 9.61334
0.00206216284501 1.088558
0.0024145667152 0.172686
0.00747873832964 0.0350860000000002
0.0522163261965 0.011585
};
\addplot [thick, black, opacity=0.4, dotted, mark=x, mark size=2, mark options={solid,fill opacity=0}, forget plot]
table {%
0.0009204159105695 22.816235
0.00106861061652 2.939755
0.00177325026714 0.497345
0.00745826609476 0.109129
0.0526202481683 0.031307
};
\addplot [thick, black, opacity=0.4, dotted, mark=x, mark size=2, mark options={solid,fill opacity=0}, forget plot]
table {%
0.000493992319159 54.665137
0.00073802650909 6.966699
0.00164834196358 1.161927
0.00749324983904 0.197729
0.05278678499 0.058357
};
\addplot [thick, black, opacity=0.4, dotted, mark=x, mark size=2, mark options={solid,fill opacity=0}, forget plot]
table {%
0.000233040511938 161.066064
0.000598815200992 20.10836
0.00162332440566 3.150487
0.00752662028315 0.587293
0.0528903900769 0.166705
};
\addplot [thick, black, opacity=0.4, dotted, mark=x, mark size=2, mark options={solid,fill opacity=0}, forget plot]
table {%
0.0001165553314975 443.794976
0.00056522288294 56.633089
0.0016256597444 8.919315
0.0075444453459 1.757166
0.0529371274145 0.528516
};
\addplot [thick, black, dash pattern=on 1pt off 3pt on 3pt off 3pt]
table {%
1e-05 104506.011488895
1.09749876549306e-05 86061.613770343
1.20450354025878e-05 70867.7874112887
1.32194114846603e-05 58352.432167654
1.45082877849594e-05 48044.0117256784
1.59228279334109e-05 39553.8971286922
1.74752840000768e-05 32561.806892457
1.91791026167249e-05 26803.8012521347
2.10490414451202e-05 22062.3829518354
2.31012970008316e-05 18158.3353476336
2.53536449397011e-05 14943.9932488125
2.78255940220713e-05 12297.6952652255
3.05385550883341e-05 10119.2104380387
3.35160265093884e-05 8325.96823849863
3.67837977182863e-05 6849.95097145205
4.03701725859655e-05 5635.1323285737
4.43062145758388e-05 4635.3662178574
4.86260158006535e-05 3812.64680782742
5.33669923120631e-05 3135.67459175147
5.85702081805666e-05 2578.67471431012
6.42807311728432e-05 2120.42323599511
7.05480231071865e-05 1743.44478989335
7.74263682681127e-05 1433.35150099881
8.49753435908644e-05 1178.29832866666
9.3260334688322e-05 968.534355361992
0.000102353102189903 796.033142051519
0.000112332403297803 654.188236568276
0.000123284673944207 537.562366691734
0.000135304777457981 441.680865798794
0.000148496826225446 362.861541037566
0.000162975083462064 298.074564151215
0.000178864952905744 244.827094547389
0.000196304065004027 201.068275209159
0.000215443469003188 165.111009410266
0.000236448941264541 135.567558658077
0.000259502421139973 111.296523531136
0.00028480358684358 91.359198638745
0.000312571584968824 74.9836469144149
0.000343046928631492 61.5351301437943
0.000376493580679247 50.4917729702729
0.000413201240011534 41.4245356449183
0.000453487850812858 33.9807339337429
0.000497702356433211 27.8704789993004
0.000546227721768434 22.8555207923045
0.000599484250318941 18.7400696906936
0.000657933224657568 15.3632462440372
0.000722080901838546 12.5928707520314
0.000792482898353917 10.320355360728
0.000869749002617784 8.4565033222864
0.000954548456661834 6.9280546182028
0.00104761575278967 5.67484559805316
0.00114975699539774 4.64747371177824
0.00126185688306602 3.80537769995394
0.00138488637139387 3.11525948371384
0.00151991108295293 2.54978706570545
0.00166810053720006 2.08652851132507
0.00183073828029537 1.70707593376109
0.00200923300256505 1.39632569336451
0.00220513073990305 1.14188701852197
0.00242012826479438 0.9335961896277
0.00265608778294669 0.76311748777354
0.00291505306282518 0.623615450076395
0.00319926713779738 0.50948572149702
0.00351119173421513 0.416134053432836
0.00385352859371053 0.339794858617451
0.0042292428743895 0.27738226099793
0.00464158883361278 0.226367836787055
0.00509413801481638 0.184680276952669
0.00559081018251223 0.150623051639776
0.00613590727341317 0.122806856038122
0.00673415065775082 0.100094191847165
0.00739072203352578 0.0815539108364285
0.00811130830789687 0.0664239352224561
0.00890215085445039 0.0540806886334641
0.00977009957299226 0.0440140336070702
0.0107226722201032 0.0358067269812234
0.01176811952435 0.0291175815140991
0.0129154966501488 0.0236676674482472
0.0141747416292681 0.0192290071451334
0.0155567614393047 0.0156153139856987
0.0170735264747069 0.0126744072653308
0.0187381742286039 0.0102820009355694
0.0205651230834865 0.00833661833127061
0.0225701971963392 0.00675542958431979
0.0247707635599171 0.00547084500041559
0.0271858824273294 0.00442772769144755
0.0298364724028334 0.00358111338559982
0.0327454916287773 0.00289434554431152
0.0359381366380463 0.00233755049140045
0.0394420605943766 0.00188639085542314
0.0432876128108306 0.00152104677591606
0.047508101621028 0.00122538346626961
0.0521400828799969 0.00098627122089486
0.0572236765935022 0.000793030097851776
0.0628029144183425 0.000636976543056763
0.068926121043497 0.000511053347910716
0.075646332755463 0.000409527712290289
0.0830217568131974 0.000327744953582808
0.0911162756115489 0.000261927669982221
0.1 0.000209012022977789
};
\label{\figlabel-line1}
\addplot [thick, black, dashed]
table {%
1e-05 1602.19092231275
1.09749876549306e-05 1413.2383014726
1.20450354025878e-05 1246.23870468708
1.32194114846603e-05 1098.67659490286
1.45082877849594e-05 968.321028051881
1.59228279334109e-05 853.193832761739
1.74752840000768e-05 751.541309446283
1.91791026167249e-05 661.809063410857
2.10490414451202e-05 582.619628414266
2.31012970008316e-05 512.752574440193
2.53536449397011e-05 451.126826724412
2.78255940220713e-05 396.784952789802
3.05385550883341e-05 348.879200742302
3.35160265093884e-05 306.659095720625
3.67837977182863e-05 269.460422476983
4.03701725859655e-05 236.695440869299
4.43062145758388e-05 207.844197812153
4.86260158006535e-05 182.446814182596
5.33669923120631e-05 160.096638503522
5.85702081805666e-05 140.434171105966
6.42807311728432e-05 123.14167305844
7.05480231071865e-05 107.938383585563
7.74263682681127e-05 94.5762781040565
8.49753435908644e-05 82.8363064928412
9.3260334688322e-05 72.5250578845713
0.000102353102189903 63.4718042070682
0.000112332403297803 55.52587999378
0.000123284673944207 48.5543606931457
0.000135304777457981 42.4400059005554
0.000148496826225446 37.0794376696393
0.000162975083462064 32.3815273820389
0.000178864952905744 28.2659676113272
0.000196304065004027 24.6620080472543
0.000215443469003188 21.5073368865931
0.000236448941264541 18.7470911782751
0.000259502421139973 16.3329814615528
0.00028480358684358 14.2225176818253
0.000312571584968824 12.3783248320294
0.000343046928631492 10.7675380681649
0.000376493580679247 9.36126820347974
0.000413201240011534 8.1341295130065
0.000453487850812858 7.06382269270803
0.000497702356433211 6.13076662810763
0.000546227721768434 5.3177733472018
0.000599484250318941 4.6097611717203
0.000657933224657568 3.99350164833767
0.000722080901838546 3.45739634520911
0.000792482898353917 2.9912800462821
0.000869749002617784 2.5862472725279
0.000954548456661834 2.23449941115329
0.00104761575278967 1.92921004598611
0.00114975699539774 1.66440635901999
0.00126185688306602 1.43486471850501
0.00138488637139387 1.23601878649822
0.00151991108295293 1.0638786715652
0.00166810053720006 0.914959823131451
0.00183073828029537 0.786220515291225
0.00200923300256505 0.675006901887045
0.00220513073990305 0.579004743328758
0.00242012826479438 0.496197010659749
0.00265608778294669 0.424826665341657
0.00291505306282518 0.363363995485904
0.00319926713779738 0.310477962026213
0.00351119173421513 0.265011072680381
0.00385352859371053 0.225957358450075
0.0042292428743895 0.192443077704811
0.00464158883361278 0.163709817345314
0.00509413801481638 0.139099699811756
0.00559081018251223 0.118042439388031
0.00613590727341317 0.100044021881001
0.00673415065775082 0.0846768087907813
0.00739072203352578 0.0715708909486873
0.00811130830789687 0.0604065376506495
0.00890215085445039 0.050907605880815
0.00977009957299226 0.0428357905910804
0.0107226722201032 0.0359856114325554
0.01176811952435 0.0301800440505651
0.0129154966501488 0.0252667152558722
0.0141747416292681 0.0211145912484763
0.0155567614393047 0.0176110967534016
0.0170735264747069 0.0146596105689866
0.0187381742286039 0.0121772897500129
0.0205651230834865 0.0100931805590323
0.0225701971963392 0.00834657951533886
0.0247707635599171 0.00688561243696056
0.0271858824273294 0.00566600338167226
0.0298364724028334 0.00465000891450693
0.0327454916287773 0.00380549621997803
0.0359381366380463 0.00310514628877707
0.0394420605943766 0.00252576578659261
0.0432876128108306 0.00204769329705555
0.047508101621028 0.00165428745706953
0.0521400828799969 0.00133148610215791
0.0572236765935022 0.00106742693947769
0.0628029144183425 0.000852121491069571
0.068926121043497 0.0006771751210991
0.075646332755463 0.000535546897110535
0.0830217568131974 0.000421343853224855
0.0911162756115489 0.000329644937341588
0.1 0.00025635054757004
};
\label{\figlabel-line3}
\addplot [ultra thick, blue]
table {%
0.0001165553314975 443.794976
0.000232761587034 161.066064
0.000493564532628 54.665137
0.000598815200992 20.10836
0.00073802650909 6.966699
0.00106861061652 2.939755
0.00164834196358 1.161927
0.00177325026714 0.497345
0.0024145667152 0.172686
0.00358106706727 0.110535
0.00699311153206 0.050972
0.00747873832964 0.0350860000000002
0.00773531370794 0.0250140000000001
0.00938574447436 0.012262
0.012577796492 0.0106710000000003
0.0217242797332 0.00872500000000009
0.032164710232 0.00728700000000027
0.0490792003595 0.00579700000000005
0.0496575860038 0.00478000000000045
0.051035633904 0.00365700000000024
};
\label{\figlabel-line0}
\addplot [ultra thick, green!50.0!black]
table {%
2.77827474765e-05 857.308716000001
5.75099002429e-05 206.636303
0.000112277003268 12.298732
0.000189863742175 4.885368
0.000630736910547 1.66329
0.000919070948055 0.841051000000001
0.00138045231564 0.397993
0.00252201901249 0.186245
0.00355017263387 0.102517
0.00893802979483 0.0531519999999994
0.0112029278899 0.0320489999999996
0.0491479964041 0.0100389999999995
0.0493803165062 0.0176369999999997
0.051035633904 0.00419800000000015
};
\label{\figlabel-line2}
\addplot [ultra thick, red]
table {%
1.682905020824e-05 311.41457299998
1.70492165828667e-05 306.572412999983
1.70829609509e-05 278.721539999977
1.727061807898e-05 244.833643999975
1.736302509678e-05 223.131715999993
1.7799389140975e-05 194.371228000005
1.786538869205e-05 193.955335999999
1.82194032139e-05 192.523257000003
1.856508115925e-05 192.067118000008
1.87717791392e-05 188.195111000005
1.89474020089e-05 186.368160000006
1.90675213093e-05 184.483863000007
2.3235792567e-05 184.372561999997
2.35003881724667e-05 181.057425000002
2.40564354213e-05 179.395044999998
2.48313255168e-05 180.746955
2.53234413125333e-05 177.331105000007
2.611400090278e-05 169.798549000005
3.02262617698333e-05 162.322601000003
3.07481399056e-05 164.745191
3.11307156634e-05 157.451844000001
3.124116675068e-05 154.185845000005
3.29682610306e-05 159.208551
3.30779291260438e-05 135.675198999993
3.33397167697556e-05 114.954770999993
3.38128293446e-05 112.921348999997
3.922086420675e-05 65.7985259999982
4.040411529026e-05 56.8814990000014
4.050032614895e-05 64.7539969999979
4.11935211211143e-05 62.9130310000005
4.21888423596e-05 63.0575000000021
4.277386761995e-05 60.5711430000014
4.3347078465e-05 57.2197920000012
4.36207134424e-05 58.9558940000024
4.39598908302e-05 58.0069680000017
4.42634295686e-05 61.403620000003
4.47561574026e-05 59.8659700000021
4.66864988026e-05 58.5270340000017
4.8744441424e-05 57.6396590000003
5.106364617055e-05 54.4889900000018
5.19660353981e-05 51.3273560000006
5.22337797112e-05 50.3061219999999
5.25823878421e-05 49.5352120000015
5.7551844229275e-05 49.1336560000015
5.81883983372e-05 49.8849330000005
6.52066751877778e-05 38.1219470000021
6.55487280999769e-05 28.3075350000003
6.61412936914e-05 28.2626309999996
6.64338127461e-05 27.4867939999991
6.69693097881e-05 28.2056369999994
7.435644045065e-05 27.5668879999992
7.44428594431e-05 28.8464919999993
7.56537529264333e-05 28.7764899999989
7.67297344697333e-05 27.637318999999
7.79030973681e-05 27.0358289999994
7.8707419815e-05 27.2006039999991
7.93225039713e-05 27.0358309999994
7.96826933982e-05 28.5968499999995
8.0190482274e-05 26.8047519999995
9.93595104183e-05 29.3306979999998
9.96115338976e-05 25.0092970000001
0.0001013383994725 23.5618489999997
0.000102775627007 23.4178759999999
0.000104142115523 23.2059579999997
0.0001055760732915 23.0544659999997
0.0001271346248165 25.1771329999998
0.000128170533344 24.871662
0.0001296658032005 23.5001720000002
0.0001309227239555 19.6669950000003
0.0001312628377554 15.0713779999999
0.000132778786775 14.9586609999998
0.000139017122603 14.7428119999998
0.000140030996222 14.5466849999995
0.000161864210675727 12.8765039999996
0.00016361695023 12.6897029999997
0.0001647859181715 12.3797989999998
0.000166708480837 12.2106189999996
0.000206367076476 11.8370919999998
0.0002075592408526 11.8189939999998
0.00020959627951575 8.82414700000006
0.00021043711795 8.53376999999994
0.000215839795109 9.08235199999992
0.000218486581616 8.24397299999986
0.000221478502515 8.13356800000004
0.000222819139759 8.24094000000007
0.000233317096064 8.00869200000002
0.000234430351084 7.92386999999998
0.0002373581483655 7.87605999999993
0.000238005333901286 7.01401099999988
0.0002403864397425 6.46234999999993
0.000243655907728 6.55352499999994
0.000261422112111 6.46819299999988
0.0002638372614615 6.34284699999993
0.000265363881059 6.22773999999995
0.000268249969741 6.18965599999995
0.0002713256933005 6.0825329999999
0.00034035304319725 5.98762900000001
0.000342829138345 5.834055
0.000407766433344 4.90760100000002
0.0004099354345794 3.93957500000004
0.000421078262066 3.72514700000004
0.000422786674327 3.60801400000003
0.000479867118785 3.57623600000005
0.000483363282856 3.58345100000002
0.000497531146255286 3.16937300000003
0.000498172384842 3.13280300000002
0.00069860140118 3.04940800000003
0.000701836986693 2.94918999999999
0.0008877216977065 2.876198
0.00090078833417725 2.62356899999999
0.000906178962919857 2.11666699999999
0.000910468585032 1.99963199999999
0.000946374523019 1.75623100000002
0.0009803715917955 1.67862700000001
0.00101185457164667 1.445019
0.00103484144939 1.332526
0.00103877776224 1.35690799999999
0.00110193295374 1.308071
0.0011141376555 1.27826699999999
0.001160295898305 1.18172799999998
0.001260940236745 1.09490999999998
0.00131180501709 1.07426099999999
0.00131695337261 1.04156899999999
0.00135348108296 1.00733499999998
0.00135851087529 0.983511999999986
0.00145913469431 0.907510999999979
0.0016874048423 0.86197699999998
0.00169723382203 0.838648999999985
0.0017326181717 0.782793999999984
0.001849454865255 0.736121999999985
0.00185957730776667 0.663066999999991
0.00186588730191 0.643076999999988
0.00198582708167 0.624894999999992
0.00241314955181 0.603413999999991
0.00245187576688 0.596231999999992
0.00251679456459 0.58778699999999
0.00306595322722 0.552186999999991
0.00307740476944 0.568876999999991
0.00338655805440667 0.510411999999994
0.00340626702522 0.482007999999996
0.003545520928785 0.461705
0.00356141453672 0.413997999999999
0.003601981847475 0.378140999999998
0.004391029690795 0.351951999999999
0.00473438971969 0.303711
0.00476834606068 0.327298
0.00555048966464 0.291315
0.00634597103084 0.249732000000001
0.0063816242189 0.253843000000001
0.0064201347306 0.223800000000002
0.00652480405925 0.109335000000001
0.00655479476002 0.140187
0.00891908090713 0.1027
0.0128186847164 0.0776390000000009
0.0217122569854667 0.0318310000000004
0.021862309007675 0.0425750000000003
0.0223981635367 0.0303530000000003
0.0233249075327 0.0173650000000003
0.0236130554393 0.015924
};
\label{\figlabel-line4}
\coordinate (legend) at (axis description cs:0.03,0.03);
\end{axis}

\matrix [matrix of nodes,
inner sep=1pt, row sep=1pt,cells={anchor=west},anchor={south west},at={(0.03,0.03)}, anchor=south west, draw=none, fill=none] at (legend) {
\ref{\figlabel-line0} {SL}\\
\ref{\figlabel-line1} {$\epsilon^{-2}\log(\epsilon^{-1})$}\\
\ref{\figlabel-line2} {ML}\\
\ref{\figlabel-line3} {$\epsilon^{-1}\log(\epsilon^{-1})^{4}$}\\
\ref{\figlabel-line4} {Adaptive ML}\\
};
\end{tikzpicture}

%% file: figures/poisson-kink-4/total-time-vs-error.tex
\begin{tikzpicture}

\begin{axis}[
xlabel={Max Error},
ylabel={Time [s.]},
xmin=1e-05, xmax=0.1,
ymin=0.001, ymax=10000,
xmode=log,
ymode=log,
axis on top,
name=\figlabel,
width=\figurewidth,
height=\figureheight,
xtick={1e-06,1e-05,0.0001,0.001,0.01,0.1,1,10},
xticklabels={,$10^{-5}$,$10^{-4}$,$10^{-3}$,$10^{-2}$,$10^{-1}$,,},
ytick={0.0001,0.001,0.01,0.1,1,10,100,1000,10000,100000},
yticklabels={,$10^{-3}$,$10^{-2}$,$10^{-1}$,$10^{0}$,$10^{1}$,$10^{2}$,$10^{3}$,$10^{4}$,},
tick pos=both
]
\addplot [thick, black, opacity=0.4, dotted, mark=x, mark size=2, mark options={solid,fill opacity=0}, forget plot]
table {%
0.028850722164 0.0717530000000004
0.0295967561139 0.0129580000000005
0.0564662247388 0.00757999999999992
};
\addplot [thick, black, opacity=0.4, dotted, mark=x, mark size=2, mark options={solid,fill opacity=0}, forget plot]
table {%
0.0192019313784 1.015519
0.0192637300731 0.0846990000000001
0.0208200459224 0.0201250000000006
0.0533431109427 0.00798600000000005
};
\addplot [thick, black, opacity=0.4, dotted, mark=x, mark size=2, mark options={solid,fill opacity=0}, forget plot]
table {%
0.0100960086376 108.389575
0.0101660326209 1.187
0.0102601901521 0.12245
0.0136774813885 0.0226859999999998
0.0518649941625 0.00599800000000017
};
\addplot [thick, black, opacity=0.4, dotted, mark=x, mark size=2, mark options={solid,fill opacity=0}, forget plot]
table {%
0.0061981501287 111.328325
0.00628025704927 1.853047
0.00641119183914 0.170275
0.0114773994563 0.0342469999999999
0.0516988407492 0.0126769999999996
};
\addplot [thick, black, opacity=0.4, dotted, mark=x, mark size=2, mark options={solid,fill opacity=0}, forget plot]
table {%
0.00300619987375 197.410262
0.00311908110281 3.240387
0.00333851386423 0.305088
0.0104501445112 0.0420390000000002
0.0517785643764 0.0175519999999998
};
\addplot [thick, black, opacity=0.4, dotted, mark=x, mark size=2, mark options={solid,fill opacity=0}, forget plot]
table {%
0.00176426575202 215.522559
0.00191623362252 4.546133
0.0022329793161 0.462441
0.0102950449777 0.0629820000000001
0.0518619939716 0.0178560000000001
};
\addplot [thick, black, opacity=0.4, dotted, mark=x, mark size=2, mark options={solid,fill opacity=0}, forget plot]
table {%
0.000814952629474 284.020738
0.00105902604125 11.989475
0.00153439349992 1.27499
0.0102773465982 0.208521
0.0519454543224 0.0416829999999999
};
\addplot [thick, black, opacity=0.4, dotted, mark=x, mark size=2, mark options={solid,fill opacity=0}, forget plot]
table {%
0.000437595506412 426.204314
0.000779824570117 26.455463
0.00134457874406 2.938779
0.010294830759 0.464164
0.0519833560028 0.155849
};
\addplot [thick, black, opacity=0.4, dotted, mark=x, mark size=2, mark options={solid,fill opacity=0}, forget plot]
table {%
0.000206737057054 1132.274854
0.000659063252098 90.08939
0.00127031301023 9.029039
0.0103123898405 1.183068
0.0520078815272 0.331731
};
\addplot [thick, black, opacity=0.4, dotted, mark=x, mark size=2, mark options={solid,fill opacity=0}, forget plot]
table {%
0.00010382158587 3142.399468
0.000625460792001 213.572087
0.00124953797103 22.288365
0.01032191961 2.977998
0.052019174254 0.758581
};
\addplot [thick, black, dash pattern=on 1pt off 3pt on 3pt off 3pt]
table {%
1e-05 2406598.96952747
1.09749876549306e-05 1921338.67711945
1.20450354025878e-05 1533823.19182859
1.32194114846603e-05 1224383.23797227
1.45082877849594e-05 977303.971845199
1.59228279334109e-05 780030.646077453
1.74752840000768e-05 622533.703742001
1.91791026167249e-05 496801.328026675
2.10490414451202e-05 396433.881134684
2.31012970008316e-05 316319.790638917
2.53536449397011e-05 252376.539620611
2.78255940220713e-05 201343.694036271
3.05385550883341e-05 160617.521240128
3.35160265093884e-05 128118.848920966
3.67837977182863e-05 102187.489060944
4.03701725859655e-05 81497.8910113284
4.43062145758388e-05 64991.7586891578
4.86260158006535e-05 51824.2230431855
5.33669923120631e-05 41320.8453554799
5.85702081805666e-05 32943.2740564871
6.42807311728432e-05 26261.8150617143
7.05480231071865e-05 20933.5251990324
7.74263682681127e-05 16684.7176881857
8.49753435908644e-05 13296.9919325413
9.3260334688322e-05 10596.0783406417
0.000102353102189903 8442.93150911752
0.000112332403297803 6726.61906238674
0.000123284673944207 5358.64450917561
0.000135304777457981 4268.41523866852
0.000148496826225446 3399.62491514806
0.000162975083462064 2707.36597706256
0.000178864952905744 2155.82505271971
0.000196304065004027 1716.44374709966
0.000215443469003188 1366.45093264749
0.000236448941264541 1087.69159017695
0.000259502421139973 865.692352250223
0.00028480358684358 688.915966192035
0.000312571584968824 548.166529022746
0.000343046928631492 436.115040880359
0.000376493580679247 346.920967515424
0.000413201240011534 275.930408228267
0.000453487850812858 219.435382469175
0.000497702356433211 174.481875389191
0.000546227721768434 138.716778996568
0.000599484250318941 110.265858313298
0.000657933224657568 87.6364625324969
0.000722080901838546 69.6399707123585
0.000792482898353917 55.3299747437565
0.000869749002617784 43.9530108947778
0.000954548456661834 34.9092964490731
0.00104761575278967 27.721442779213
0.00114975699539774 22.0095269500272
0.00126185688306602 17.4712316424527
0.00138488637139387 13.8660246003165
0.00151991108295293 11.0025573213533
0.00166810053720006 8.7286290282092
0.00183073828029537 6.92319459641526
0.00200923300256505 5.49000089309973
0.00220513073990305 4.35252032626792
0.00242012826479438 3.44991765584216
0.00265608778294669 2.73383973416691
0.00291505306282518 2.16586058584291
0.00319926713779738 1.71544830689882
0.00351119173421513 1.3583474183444
0.00385352859371053 1.07529195052255
0.0042292428743895 0.850981780204163
0.00464158883361278 0.673268483563903
0.00509413801481638 0.53250791600166
0.00559081018251223 0.421045451169121
0.00613590727341317 0.332806758564462
0.00673415065775082 0.262972532072022
0.00739072203352578 0.207719988159774
0.00811130830789687 0.164017461129536
0.00890215085445039 0.129461216387524
0.00977009957299226 0.102145826649196
0.0107226722201032 0.0805612262476732
0.01176811952435 0.063510967658543
0.0129154966501488 0.0500473255875186
0.0141747416292681 0.0394197861263907
0.0155567614393047 0.0310341682740019
0.0170735264747069 0.024420189743834
0.0187381742286039 0.0192057380651279
0.0205651230834865 0.0150964651222539
0.0225701971963392 0.0118596072515296
0.0247707635599171 0.0093111587803203
0.0271858824273294 0.00730570635411625
0.0298364724028334 0.00572837402808977
0.0327454916287773 0.00448844244060772
0.0359381366380463 0.00351429543719209
0.0394420605943766 0.00274941904810059
0.0432876128108306 0.00214923453846778
0.047508101621028 0.00167859236728383
0.0521400828799969 0.00130978971252197
0.0572236765935022 0.00102100265456328
0.0628029144183425 0.000795046677334134
0.068926121043497 0.000618397053474035
0.075646332755463 0.000480414885877113
0.0830217568131974 0.000372735845589451
0.0911162756115489 0.000288787581126161
0.1 0.000223408858078854
};
\label{\figlabel-line1}
\addplot [thick, black, dashed]
table {%
1e-05 7264.78478695498
1.09749876549306e-05 6313.98121875807
1.20450354025878e-05 5486.88894378568
1.32194114846603e-05 4767.49729428287
1.45082877849594e-05 4141.85763083549
1.59228279334109e-05 3597.81895263679
1.74752840000768e-05 3124.79726899365
1.91791026167249e-05 2713.57443713104
2.10490414451202e-05 2356.12271583912
2.31012970008316e-05 2045.4517601894
2.53536449397011e-05 1775.47519809111
2.78255940220713e-05 1540.89429244757
3.05385550883341e-05 1337.0965097312
3.35160265093884e-05 1160.06709271839
3.67837977182863e-05 1006.31197697758
4.03701725859655e-05 872.790601912026
4.43062145758388e-05 756.857351594127
4.86260158006535e-05 656.210521673294
5.33669923120631e-05 568.847849250989
5.85702081805666e-05 493.027765378793
6.42807311728432e-05 427.235637006573
7.05480231071865e-05 370.154358761464
7.74263682681127e-05 320.638736598729
8.49753435908644e-05 277.693176639614
9.3260334688322e-05 240.452254715402
0.000102353102189903 208.163796421057
0.000112332403297803 180.174144851154
0.000123284673944207 155.915334522645
0.000135304777457981 134.893926050297
0.000148496826225446 116.681287600892
0.000162975083462064 100.905136596407
0.000178864952905744 87.2421790750957
0.000196304065004027 75.4117049988015
0.000215443469003188 65.1700160043223
0.000236448941264541 56.3055779763491
0.000259502421139973 48.6348046663053
0.00028480358684358 41.9983906545351
0.000312571584968824 36.2581224789987
0.000343046928631492 31.294105929456
0.000376493580679247 27.0023555044587
0.000413201240011534 23.2926989997873
0.000453487850812858 20.0869562725815
0.000497702356433211 17.3173565199053
0.000546227721768434 14.9251630238539
0.000599484250318941 12.8594783347844
0.000657933224657568 11.076206365945
0.000722080901838546 9.5371509230975
0.000792482898353917 8.20923284956955
0.000869749002617784 7.06381028108468
0.000954548456661834 6.07608851972727
0.00104761575278967 5.22460779095448
0.00114975699539774 4.4907986751575
0.00126185688306602 3.85859633512129
0.00138488637139387 3.3141058183041
0.00151991108295293 2.84531172036009
0.00166810053720006 2.44182637311673
0.00183073828029537 2.09467148316381
0.00200923300256505 1.7960888110171
0.00220513073990305 1.53937605829974
0.00242012826479438 1.31874463270582
0.00265608778294669 1.12919639739891
0.00291505306282518 0.966416891432253
0.00319926713779738 0.826682838137407
0.00351119173421513 0.70678204564733
0.00385352859371053 0.603944053397958
0.0042292428743895 0.515780095467815
0.00464158883361278 0.440231140214973
0.00509413801481638 0.375522929552934
0.00559081018251223 0.32012708359111
0.00613590727341317 0.272727460053991
0.00673415065775082 0.232191065324891
0.00739072203352578 0.197542907258246
0.00811130830789687 0.167944260915935
0.00890215085445039 0.142673888715176
0.00977009957299226 0.121111817526605
0.0107226722201032 0.102725328247241
0.01176811952435 0.0870568593522661
0.0129154966501488 0.0737135658217345
0.0141747416292681 0.0623583094435444
0.0155567614393047 0.0527018865077252
0.0170735264747069 0.0444963249339515
0.0187381742286039 0.0375291054397457
0.0205651230834865 0.0316181809177924
0.0225701971963392 0.026607685144065
0.0247707635599171 0.0223642366288739
0.0271858824273294 0.0187737561500543
0.0298364724028334 0.0157387275315083
0.0327454916287773 0.0131758407771668
0.0359381366380463 0.0110139649363512
0.0394420605943766 0.00919240523184468
0.0432876128108306 0.00765940517459433
0.047508101621028 0.00637085974720528
0.0521400828799969 0.00528921037361429
0.0572236765935022 0.00438249640118246
0.0628029144183425 0.0036235412877351
0.068926121043497 0.00298925468250947
0.075646332755463 0.00246003417958829
0.0830217568131974 0.00201925275985916
0.0911162756115489 0.00165281987023173
0.1 0.00134880575782921
};
\label{\figlabel-line3}
\addplot [ultra thick, blue]
table {%
0.00010382158587 3142.399468
0.000206737057054 1132.274854
0.000437595506412 426.204314
0.000625460792001 213.572087
0.000659063252098 90.08939
0.000779824570117 26.455463
0.00105902604125 11.989475
0.00127031301023 9.029039
0.00134457874406 2.938779
0.00153439349992 1.27499
0.0022329793161 0.462441
0.00333851386423 0.305088
0.00641119183914 0.170275
0.0102601901521 0.12245
0.0102950449777 0.0629820000000001
0.0104501445112 0.0420390000000002
0.0114773994563 0.0342469999999999
0.0136774813885 0.0226859999999998
0.0208200459224 0.0201250000000006
0.0295967561139 0.0129580000000005
0.05178191745585 0.00599800000000017
};
\label{\figlabel-line0}
\addplot [ultra thick, green!50.0!black]
table {%
4.38378183888e-05 1119.346201
8.33039967375e-05 252.305227
0.000171950944015 114.749576
0.000660061750766 4.707762
0.000842738934826 2.070529
0.00167966598481 0.694652000000001
0.00241062258745 0.322607000000001
0.00351338479303 0.166611
0.0107445668526 0.0751949999999995
0.0123733619109 0.049795
0.0509587050092 0.0226240000000002
0.0527863355072 0.0125219999999999
0.0564662247388 0.00512100000000015
};
\label{\figlabel-line2}
\addplot [ultra thick, red]
table {%
5.166463392456e-05 826.398530999836
5.20135367082e-05 878.90814299986
5.7706200994e-05 827.170985999704
5.7980304607975e-05 820.869913999804
6.3070808781975e-05 753.698467999869
6.3629830322e-05 588.175911999865
6.38128119674727e-05 432.484100999825
6.42126193414e-05 758.882064999806
6.45657540918667e-05 387.604139999833
6.514230915655e-05 386.593793999816
6.54609337485e-05 401.404986999825
6.58753575629e-05 384.943788999849
6.64041904747e-05 382.971643999849
6.67000993935e-05 383.920847999904
6.712063362655e-05 382.747271999889
6.79076474708e-05 383.487253999859
6.82759345052e-05 383.675616999867
6.894180324015e-05 382.353180999909
6.92918998478e-05 381.12649799993
6.94118091888e-05 382.237210999903
7.0035513574575e-05 380.735857999903
7.05590631134e-05 380.667598999903
7.08431043756e-05 382.125140999907
7.12845008058333e-05 380.182598999924
7.18697233814e-05 381.215477999945
7.19891405109e-05 379.148866999943
7.26023650848e-05 379.390931999978
7.28651290182e-05 378.869885999969
7.346570295132e-05 377.370247000004
7.38800139614333e-05 376.599408999949
7.48168366828e-05 376.547468999958
7.535829793975e-05 376.599718999981
7.65595305171e-05 377.331090999976
7.934722410826e-05 374.775916000011
8.01180722695e-05 375.684652999936
8.05925289719e-05 374.470680999974
8.11911794664e-05 374.308478999968
8.189152365385e-05 373.705968999963
8.30494103435e-05 373.785219999952
8.3664163077475e-05 288.175898999995
8.38409439523692e-05 289.55830099998
9.45226777655e-05 291.194267000022
9.4886613247625e-05 286.922363000006
9.55524797338e-05 287.584774000007
9.62390114112e-05 284.787907999996
9.69429892956e-05 284.18074400002
9.770873976985e-05 277.256797000014
9.89797120225e-05 279.904126000006
0.000100280593018 277.338860000005
0.000101609151162 276.329955000006
0.000102602044138 276.524563
0.0001042739187785 260.416412000025
0.000105180947027 256.782162000023
0.000106800784959833 244.911297000014
0.0001068196360515 250.471306000019
0.000107892517542 323.791522000022
0.000108445123697 334.727740000021
0.000108923735235 328.013089000005
0.000116201299295 317.233230000032
0.00011718427798 292.276740000022
0.000118660442074 245.308726000008
0.000120080097512 225.771907000009
0.0001210951463305 226.378756000007
0.000121814529012 226.084565000004
0.000122153364157 227.858690999999
0.00013177976245613 149.506512000008
0.000132260211363 320.09499500001
0.000137386962308 155.152416000004
0.000138286429285 146.549664999996
0.000138981426191 148.192632999994
0.0001411174197545 143.219898000005
0.000141416285912 150.109073999987
0.000142950775949 119.467531
0.000143179899814333 135.09347500001
0.000147086520859333 126.858175000008
0.000150913561057 119.500042000001
0.000153682507651 92.6062800000011
0.000161356526235929 73.1216200000047
0.000161854015827 71.6499030000034
0.000162846205595833 71.5008670000036
0.000163483286634 75.2484290000068
0.00016588218708 90.3982730000055
0.000168702923257 87.7668630000051
0.000169870306769 82.9553930000014
0.000171040267601 89.4829820000034
0.000171988023008 75.7033550000035
0.000173328237249 75.0868370000028
0.000190654625461 70.6777530000016
0.000192903627601 73.6512320000041
0.000210698079811429 66.733198000002
0.000211145831562 74.6587880000032
0.0002157782068618 63.2033020000024
0.000216288453707 52.942345000002
0.00022262982363975 46.4967970000018
0.000223019972331333 42.0522770000023
0.00022647550308425 35.0239979999999
0.000231244310738 32.1493600000006
0.000249840669419 31.6651210000009
0.000252717155200533 25.4877569999991
0.000253756249862571 26.2434699999988
0.000256495873754 27.7433419999989
0.000259896649859 25.975036999999
0.000338490328946 24.2035179999997
0.000346592552494333 21.1664079999992
0.000352909739106 20.4827069999991
0.0003542643043605 20.1793229999999
0.000360989619306 19.8305499999997
0.000368701837839 17.3960699999998
0.00037408188956575 16.4671289999996
0.000377758986374 18.4431349999995
0.000416320841375 16.7853329999994
0.00041874780832375 16.9077479999993
0.000422186624728 16.3957599999993
0.0004412406621815 16.0592769999996
0.000448583492883 16.0533549999996
0.000457289197606 15.3812059999994
0.000461717299636 11.3074759999997
0.000463479573614278 11.1560669999997
0.000505691831200555 9.25924499999978
0.0005080297727942 8.51466199999981
0.000517407575865 8.81639199999988
0.000524953307809 8.15773099999977
0.00053291639749475 7.84321600000013
0.000535546653815 7.28106600000004
0.00054133953001 7.20829100000003
0.000541624773118 7.20280400000003
0.000553936667545 6.90091100000008
0.00056297374972 6.28981800000006
0.0005723394315855 6.03530300000003
0.00058063449875 5.52174100000005
0.0005885053909565 5.37156000000002
0.00059960466074 5.24924700000006
0.000608391303682 4.95299400000002
0.0006282243834955 4.77806600000001
0.000637596038142 4.70177199999998
0.000657579434886 4.58105700000001
0.000660368057151 4.51938100000003
0.000662094272679 4.32068600000004
0.000670635922083 4.15837500000007
0.000680752413720667 3.88673000000001
0.00068465449003525 3.89182900000003
0.00068640633154 4.07189100000002
0.000697662462052 3.815286
0.000699097717867 3.74950300000007
0.000709099572131812 3.15652700000002
0.00074009013215 3.04272799999997
0.000757035519925 2.95229099999996
0.000766205202986 2.83692799999998
0.000850270263834 2.74942799999994
0.000867051475432 2.58168499999994
0.000870609530594 2.56189099999991
0.000876805291243 2.49612499999993
0.000885652875215 2.4798919999999
0.000893747462195 2.4372459999999
0.000945793988961 2.36667099999989
0.0009499959849865 2.33570799999989
0.00100439878711571 2.14582799999993
0.00101110823161333 1.98031299999996
0.00101322395762333 1.98765799999994
0.00104314506775 1.90685299999995
0.00105873637318 1.84919399999995
0.00108560765242 1.84169799999994
0.00116815474592 1.85952799999996
0.00118025535796 1.76753299999995
0.00119616523312333 1.63917999999996
0.00121545264835 1.57684799999997
0.00123631120291 1.53011899999996
0.00125651350121 1.52264899999996
0.00127178305999 1.43212699999998
0.00128681520997 1.43459399999997
0.00129831519478 1.35728399999998
0.001337910213155 1.26210699999998
0.00134100728547 1.23644799999998
0.00136116993381 1.12556699999999
0.00150759960148 1.03524699999998
0.0015171688806 1.02876399999999
0.001526964476075 0.963382999999981
0.00165447259295 0.953832999999982
0.0018427350117 0.916516999999984
0.002123076005635 0.868691999999986
0.00212878627448 0.843082999999988
0.00239883444475 0.825886999999995
0.00305150512756 0.79892699999998
0.00307912091924 0.766789999999981
0.00309098586283 0.767190999999989
0.00330844289854 0.740345999999987
0.0033214185031575 0.681819999999995
0.00362578735408 0.652520999999995
0.003640959907864 0.54978999999999
0.00366091114612 0.544129999999994
0.00369424566286 0.643961999999995
0.00437704018641 0.513579999999992
0.004721115804445 0.487772999999992
0.00497970201738 0.446329999999989
0.00507997784354 0.424871999999992
0.00536616648322 0.449137999999993
0.00551412912657 0.133118999999996
0.005540666360772 0.122411
0.0055714440701825 0.112843999999998
0.00709047350622 0.105037999999998
0.0108101508011 0.0985769999999995
0.0191364893705 0.0944070000000001
0.0192290086527 0.0886989999999999
0.0194264722851 0.0806519999999991
0.0195581412143 0.0754479999999995
0.0199436140142 0.0698729999999996
0.0201293567057 0.0619239999999994
0.0205159086314 0.0569559999999993
0.02070218203545 0.0450849999999994
0.0208391967506 0.0278259999999997
0.0210360630119 0.036546
0.0212382346776 0.0225230000000001
0.0218290756358 0.0180039999999999
0.0223567165325 0.0149869999999999
};
\label{\figlabel-line4}
\coordinate (legend) at (axis description cs:0.03,0.03);
\end{axis}

\matrix [matrix of nodes,
inner sep=1pt, row sep=1pt,cells={anchor=west},anchor={south west},at={(0.03,0.03)}, anchor=south west, draw=none, fill=none] at (legend) {
\ref{\figlabel-line0} {SL}\\
\ref{\figlabel-line1} {$\epsilon^{-\frac{ 7 }{ 3 }}\log(\epsilon^{-1})$}\\
\ref{\figlabel-line2} {ML}\\
\ref{\figlabel-line3} {$\epsilon^{-\frac{ 4 }{ 3 }}\log(\epsilon^{-1})^{2}$}\\
\ref{\figlabel-line4} {Adaptive ML}\\
};
\end{tikzpicture}

%% file: figures/poisson-kink-6/total-time-vs-error.tex
\begin{tikzpicture}

\begin{axis}[
xlabel={Max Error},
ylabel={Time [s.]},
xmin=0.0001, xmax=0.1,
ymin=0.001, ymax=10000,
xmode=log,
ymode=log,
axis on top,
name=\figlabel,
width=\figurewidth,
height=\figureheight,
xtick={1e-05,0.0001,0.001,0.01,0.1,1,10},
xticklabels={,$10^{-4}$,$10^{-3}$,$10^{-2}$,$10^{-1}$,,},
ytick={0.0001,0.001,0.01,0.1,1,10,100,1000,10000,100000},
yticklabels={,$10^{-3}$,$10^{-2}$,$10^{-1}$,$10^{0}$,$10^{1}$,$10^{2}$,$10^{3}$,$10^{4}$,},
tick pos=both
]
\addplot [thick, black, opacity=0.4, dotted, mark=x, mark size=2, mark options={solid,fill opacity=0}, forget plot]
table {%
0.0219923706361 0.280572
0.0231485452804 0.032683
0.0473996127045 0.0073700000000001
};
\addplot [thick, black, opacity=0.4, dotted, mark=x, mark size=2, mark options={solid,fill opacity=0}, forget plot]
table {%
0.0146468171491 56.319003
0.0146825404421 0.281625
0.0171486181992 0.0369269999999999
0.0456595617549 0.00841200000000009
};
\addplot [thick, black, opacity=0.4, dotted, mark=x, mark size=2, mark options={solid,fill opacity=0}, forget plot]
table {%
0.00773597291035 71.243986
0.0079395211684 0.479064
0.0129739708107 0.0495500000000002
0.045076536669 0.00707199999999997
};
\addplot [thick, black, opacity=0.4, dotted, mark=x, mark size=2, mark options={solid,fill opacity=0}, forget plot]
table {%
0.00475918310523 52.994155
0.0051763040426 0.503328
0.0120103618342 0.0421529999999999
0.0451498063531 0.00845800000000008
};
\addplot [thick, black, opacity=0.4, dotted, mark=x, mark size=2, mark options={solid,fill opacity=0}, forget plot]
table {%
0.00232937038259 96.979911
0.00321933670459 1.083645
0.0117499134021 0.110281
0.045355625829 0.0184980000000006
};
\addplot [thick, black, opacity=0.4, dotted, mark=x, mark size=2, mark options={solid,fill opacity=0}, forget plot]
table {%
0.00139560009836 115.91719
0.00268287982911 1.983264
0.011787563253 0.200733
0.0454708567658 0.0245000000000002
};
\addplot [thick, black, opacity=0.4, dotted, mark=x, mark size=2, mark options={solid,fill opacity=0}, forget plot]
table {%
0.000712175432305 164.04346
0.00244887368137 4.321099
0.0118688266345 0.397337
0.0455720926749 0.0873309999999998
};
\addplot [thick, black, opacity=0.4, dotted, mark=x, mark size=2, mark options={solid,fill opacity=0}, forget plot]
table {%
0.000477321676295 285.061344
0.0024110537785 10.507983
0.0119135687612 0.877356000000001
0.0456154902222 0.140544
};
\addplot [thick, black, opacity=0.4, dotted, mark=x, mark size=2, mark options={solid,fill opacity=0}, forget plot]
table {%
0.0003730803792 640.047029
0.00240481031792 31.319652
0.0119444028253 2.397552
0.0456429401838 0.375838
};
\addplot [thick, black, opacity=0.4, dotted, mark=x, mark size=2, mark options={solid,fill opacity=0}, forget plot]
table {%
0.000346112906001 1759.361261
0.00240622839308 95.844415
0.011959031944 7.988817
0.0456554324994 1.096945
};
\addplot [thick, black, dash pattern=on 1pt off 3pt on 3pt off 3pt]
table {%
0.0001 423239.328877998
0.000107226722201032 340701.6921795
0.000114975699539774 274244.076807414
0.000123284673944207 220736.68873795
0.000132194114846603 177658.377646415
0.000141747416292681 142978.373492576
0.000151991108295294 115060.984459645
0.000162975083462064 92588.8015041979
0.000174752840000768 74500.818434034
0.000187381742286039 59942.5727609829
0.000200923300256505 48225.9739746426
0.000215443469003188 38796.9384804832
0.000231012970008316 31209.3153141528
0.000247707635599171 25103.8808769865
0.000265608778294669 20191.4180395954
0.00028480358684358 16239.0860817427
0.000305385550883342 13059.4419882983
0.000327454916287773 10501.5977879835
0.000351119173421514 8444.09869718282
0.000376493580679247 6789.18748632729
0.000403701725859656 5458.18548717894
0.000432876128108306 4387.7730419552
0.000464158883361278 3526.99440731665
0.000497702356433211 2834.84614103554
0.000533669923120631 2278.33540732314
0.000572236765935022 1830.91672053959
0.000613590727341318 1471.23343988438
0.000657933224657568 1182.1046627339
0.000705480231071865 949.709713000246
0.000756463327554629 762.93172439459
0.000811130830789688 612.829312925118
0.000869749002617784 492.211369825956
0.000932603346883221 395.294868644612
0.001 317.429496658499
0.00107226722201032 254.876075065581
0.00114975699539774 204.628272694764
0.00123284673944207 164.269163711963
0.00132194114846603 131.855827159449
0.00141747416292681 105.826512663797
0.00151991108295294 84.9259647202143
0.00162975083462064 68.1453579070896
0.00174752840000769 54.6739877217507
0.00187381742286038 43.8604190934022
0.00200923300256505 35.1812433093705
0.00215443469003189 28.2159552585332
0.00231012970008316 22.6267536027607
0.00247707635599171 18.1423004657214
0.00265608778294669 14.5446655369967
0.0028480358684358 11.658831033046
0.00305385550883342 9.34425590542035
0.00327454916287773 7.48809581404042
0.00351119173421513 5.99975433747202
0.00376493580679247 4.80650441509896
0.00403701725859656 3.84997012042086
0.00432876128108306 3.08329997542798
0.00464158883361278 2.46889608512166
0.00497702356433211 1.97658996989634
0.00533669923120632 1.58217736619662
0.00572236765935022 1.26624147027971
0.00613590727341317 1.01320793501471
0.00657933224657568 0.810586054446102
0.00705480231071864 0.648359515605939
0.00756463327554629 0.518497288423507
0.00811130830789687 0.414561005802286
0.00869749002617784 0.33138983315015
0.00932603346883221 0.26484756199189
0.01 0.211619664438999
0.0107226722201032 0.169050457951661
0.0114975699539774 0.135012468582112
0.0123284673944207 0.107801638685976
0.0132194114846603 0.0860532766724825
0.0141747416292681 0.068674651835017
0.0151991108295293 0.0547909449807835
0.0162975083462065 0.0437019143099813
0.0174752840000768 0.0348471570094675
0.0187381742286039 0.0277782654258214
0.0200923300256505 0.0221365126440983
0.0215443469003188 0.0176349720365833
0.0231012970008316 0.0140441918913688
0.0247707635599171 0.0111807200544562
0.0265608778294669 0.008897913034398
0.028480358684358 0.00707857598434937
0.0305385550883342 0.00562906982254237
0.0327454916287773 0.00447459384009732
0.0351119173421513 0.00355540997776119
0.0376493580679247 0.00282382134387065
0.0403701725859656 0.00224175475366278
0.0432876128108306 0.00177882690890076
0.0464158883361278 0.00141079776292666
0.0497702356433211 0.00111833379875714
0.0533669923120631 0.000886019325070111
0.0572236765935022 0.000701566220019842
0.0613590727341318 0.000555182430145048
0.0657933224657568 0.000439067446158306
0.0705480231071865 0.000347009318211629
0.0756463327554629 0.000274062852452426
0.0811130830789687 0.000216292698679454
0.0869749002617784 0.000170568296474342
0.093260334688322 0.000134400255339168
0.1 0.0001058098322195
};
\label{\figlabel-line1}
\addplot [thick, black, dashed]
table {%
0.0001 2958.06499642139
0.000107226722201032 2533.94031177136
0.000114975699539774 2170.37327630492
0.000123284673944207 1858.75045876355
0.000132194114846603 1591.67922140058
0.000141747416292681 1362.81526290182
0.000151991108295294 1166.71444949076
0.000162975083462064 998.705523258718
0.000174752840000768 854.780754519427
0.000187381742286039 731.502015994051
0.000200923300256505 625.920110191836
0.000215443469003188 535.505485474149
0.000231012970008316 458.088737871722
0.000247707635599171 391.809520703052
0.000265608778294669 335.072677525486
0.00028480358684358 286.5105803382
0.000305385550883342 244.950798032762
0.000327454916287773 209.388343110538
0.000351119173421514 178.961850459888
0.000376493580679247 152.933132922859
0.000403701725859656 130.669636556367
0.000432876128108306 111.629385693204
0.000464158883361278 95.3480656700718
0.000497702356433211 81.4279407355213
0.000533669923120631 69.5283473174526
0.000572236765935022 59.3575394960833
0.000613590727341318 50.6656950356365
0.000657933224657568 43.2389174002405
0.000705480231071865 36.8940924392158
0.000756463327554629 31.4744784094651
0.000811130830789688 26.8459251684521
0.000869749002617784 22.8936331161756
0.000932603346883221 19.5193751289983
0.001 16.6391156048703
0.00107226722201032 14.1809700799791
0.00114975699539774 12.083456897487
0.00123284673944207 10.2939992956943
0.00132194114846603 8.76764219552019
0.00141747416292681 7.46595304296641
0.00151991108295294 6.35608041927789
0.00162975083462064 5.40994787127154
0.00174752840000769 4.6035636239434
0.00187381742286038 3.91642959187773
0.00200923300256505 3.3310354695173
0.00215443469003189 2.83242570829301
0.00231012970008316 2.40782892843824
0.00247707635599171 2.04634080581864
0.00265608778294669 1.73865275432465
0.0028480358684358 1.47681982238756
0.00305385550883342 1.25406216382855
0.00327454916287773 1.06459525071853
0.00351119173421513 0.903484688263565
0.00376493580679247 0.766522085289607
0.00403701725859656 0.650118942720255
0.00432876128108306 0.551215958572727
0.00464158883361278 0.467205521783352
0.00497702356433211 0.395865487491264
0.00533669923120632 0.335302600875061
0.00572236765935022 0.283904171788411
0.00613590727341317 0.240296803884752
0.00657933224657568 0.203311154469702
0.00705480231071864 0.171951849099563
0.00756463327554629 0.145371801495314
0.00811130830789687 0.122850297699924
0.00869749002617784 0.103774296176057
0.00932603346883221 0.0876224749540735
0.01 0.0739516249105348
0.0107226722201032 0.0623850464263971
0.0114975699539774 0.0526026564481951
0.0123284673944207 0.0443325555605584
0.0132194114846603 0.0373438411075674
0.0141747416292681 0.0314404835591648
0.0151991108295293 0.0264561099657769
0.0162975083462065 0.022249561129367
0.0174752840000768 0.0187011085991373
0.0187381742286039 0.0157092342518651
0.0200923300256505 0.0131878894488149
0.0215443469003188 0.0110641629230196
0.0231012970008316 0.00927629694190239
0.0247707635599171 0.00777200016704239
0.0265608778294669 0.00650701321480119
0.028480358684358 0.00544388939630109
0.0305385550883342 0.00455095864359323
0.0327454916287773 0.00380144734826769
0.0351119173421513 0.00317273086687891
0.0376493580679247 0.00264569888500741
0.0403701725859656 0.00220421676461474
0.0432876128108306 0.00183466850116662
0.0464158883361278 0.00152556905072115
0.0497702356433211 0.00126723560659859
0.0533669923120631 0.00105150895634419
0.0572236765935022 0.000871517371760993
0.0613590727341318 0.000721476611401018
0.0657933224657568 0.000596520574746176
0.0705480231071865 0.00049255796488746
0.0756463327554629 0.000406151012749154
0.0811130830789687 0.000334412908672761
0.0869749002617784 0.000274921091729391
0.093260334688322 0.000225643976491221
0.1 0.000184879062276337
};
\label{\figlabel-line3}
\addplot [ultra thick, blue]
table {%
0.000346112906001 1759.361261
0.0003730803792 640.047029
0.000477321676295 285.061344
0.000712175432305 164.04346
0.00139560009836 115.91719
0.00232937038259 96.979911
0.00240481031792 31.319652
0.0024110537785 10.507983
0.00244887368137 4.321099
0.00268287982911 1.983264
0.00321933670459 1.083645
0.0051763040426 0.503328
0.0079395211684 0.479064
0.0117499134021 0.110281
0.0120103618342 0.0421529999999999
0.0171486181992 0.0369269999999999
0.0231485452804 0.032683
0.045076536669 0.00707199999999997
};
\label{\figlabel-line0}
\addplot [ultra thick, green!50.0!black]
table {%
0.000381566361572 161.411429
0.000451865796974 110.254917
0.000585790565792 36.67568
0.0024728679842 2.408645
0.00266604551364 1.18852
0.00311350959574 0.573986
0.0113418394376 0.204287
0.0115006079295 0.127699000000001
0.0124560343379 0.0566240000000005
0.0443133452182 0.024389
0.045512938072 0.0135700000000007
0.0473996127045 0.00540899999999978
};
\label{\figlabel-line2}
\addplot [ultra thick, red]
table {%
0.0001960709611192 737.808295999975
0.000196854020298 727.989599999988
0.00019804267732275 703.357109999984
0.000198603139171 692.055464000021
0.000199217701188 686.303994000003
0.000201205139752 666.920267000025
0.000202277559396667 658.144116999918
0.00020299445505 687.917508999979
0.000204172098731 616.318218999908
0.000205601422325667 668.136470999978
0.00020695197751575 650.707283999933
0.000207189973371 640.80303899992
0.000208255156965 640.700297999881
0.000209293310891 633.02301599995
0.000210497834136 626.100400999976
0.0002110656993055 615.800344999896
0.0002121289537415 598.275698999934
0.00021355474089 534.704004999994
0.0002144376853995 473.763580000073
0.000214863960154667 468.744607999966
0.0002162747768775 453.037967999991
0.000217025117762 404.578728000031
0.000218028211424 418.020339000018
0.000218802922344 416.438449999996
0.000219540432973 397.53161
0.000220985433509333 389.251668000006
0.000222027028759 384.714124000029
0.000223273544902 378.819726000025
0.000223641710546 369.962953999996
0.0002246833374775 360.460670999958
0.000226459466349 355.004198999997
0.0002278963871455 349.150304999982
0.000228448407873 345.759503999974
0.0002295748511238 328.664476999985
0.000230658385783 343.060018999947
0.000231432388934 325.868679999957
0.000232840166556667 316.403130999986
0.000233845033977 312.544023999975
0.000235061101871 309.696189999997
0.000236289876817667 275.996198999973
0.000237021541556 270.629156999968
0.000237641899166 264.188015999958
0.000238861588659333 258.673287999934
0.000240236713054667 251.762497999925
0.000241400224986 238.172312999961
0.0002423824746896 193.238509999997
0.000243209009407945 171.328636999989
0.000244510391002583 170.308438000012
0.000245448292627889 169.713555999993
0.0002514116945725 168.710344999979
0.000252312856879 165.042215999967
0.000253078487053067 160.130260999989
0.000254301324113 252.070107000016
0.000255685449247 259.965375000005
0.00025664855162 259.185443000016
0.000258033353835 252.979473000021
0.000259258378160625 156.902840000011
0.000259564185494 190.801435000003
0.000260665249915 158.368874000003
0.000262555296169 155.809023000001
0.000271821357645 156.199467000004
0.000273146157854 194.731583000009
0.0002745402101 204.407146000006
0.000275706888827 204.952800000008
0.000279031961344 211.663478000002
0.000281613775824333 176.205423999999
0.0002825831778344 143.320203000015
0.0002873768395395 125.370210000013
0.000290928390999 118.915649000011
0.000291518843047429 109.903805000014
0.000292507597082333 107.819865000021
0.000295409021421 106.52829200002
0.000296831862402 104.366019000013
0.000299747804516 97.8695310000253
0.00030051986085475 93.9857490000245
0.000302247549054 94.8908050000132
0.000303130251286 86.7264560000199
0.000304198980762 99.4044830000199
0.00031196363135225 82.4881840000194
0.000312784552774 82.4365290000147
0.000316057226580333 81.3968730000047
0.000322939430171 90.5538970000303
0.000325451023899 90.9088230000175
0.000328297896535 91.6993660000185
0.000329646255921 89.7327080000091
0.000331620727146667 87.6943980000138
0.000333927321716 88.7284220000152
0.000335208690935 84.826789000023
0.000337222825823 92.6448950000154
0.0003404374447565 89.2392550000164
0.0003417188045385 86.8321870000123
0.000343748703051 86.5258290000134
0.000346811421293 85.7166620000158
0.000347592665874 79.7095380000142
0.000349832932715333 73.6484740000138
0.000350657361504 76.5035010000179
0.000351790949572 68.4041140000159
0.000355189775868 76.8232230000181
0.0003567716952615 68.2209290000133
0.000359228010425 65.8621950000148
0.000362370974331 67.5965520000141
0.0003638470968518 73.0185990000078
0.0003654546503845 70.5382170000008
0.00036716586277 72.8126120000021
0.000368858490075 69.998957000004
0.0003703217690545 68.9518970000041
0.0003719916421635 61.0345159999993
0.00037341143468 69.1439410000008
0.000375017896589 59.392171
0.000376239087398 57.1620410000065
0.000380563621461 53.3604929999961
0.000383259838497 49.715608999994
0.0003853494346595 50.4747429999948
0.000386195906104 56.6000979999977
0.000396178762489 58.9059259999991
0.000397553316306 53.7578409999983
0.00040059381862 52.0624370000028
0.000403765067384 49.1649820000026
0.000406319199868 50.4132580000056
0.000415648338658 51.5216580000047
0.0004223096414522 50.2579040000042
0.000423775785566 47.2688190000095
0.000425732040548 45.5429130000098
0.0004282857708295 44.9886620000121
0.000430738152901 43.9581370000065
0.000432783000265 45.0125630000089
0.00043434641816 40.9138830000116
0.000437876509574 41.9253240000112
0.000439393788988333 38.1977680000105
0.000442367915847333 32.9127590000112
0.000443522961622 32.0322690000144
0.000445484308653529 31.1379250000105
0.00044702992839225 30.9396380000126
0.0004495685620275 29.7996620000118
0.000451297117691429 29.5381790000138
0.000453790036123 29.5791800000111
0.000455738948988 29.3170920000107
0.000457718227683 30.19855900001
0.000459946523051 28.6124760000117
0.000460720303909 28.5383980000093
0.0004780664734075 28.5216040000101
0.000504233355798 33.163295000009
0.0005060332322185 28.3121280000055
0.000509468836734 27.9349580000053
0.000509686034308 27.935909000004
0.000512997822149 28.0319760000057
0.000544860367978 27.9930950000062
0.000567859081628 27.8581090000062
0.000569800301941375 27.8265840000077
0.000596597141971 27.870156000005
0.000599072725923 28.9226660000077
0.0006276311424835 27.9553890000075
0.000699834225015286 28.0023080000074
0.0007016768976054 26.8475090000061
0.000704293847567667 27.5996560000036
0.000828017262612 26.2921910000051
0.000829293815150857 25.6969290000011
0.0008330600142377 25.9733319999987
0.000837256138046875 28.1570379999976
0.000840622521942857 26.4580189999975
0.000856196724845 27.4307909999971
0.0008590828447325 28.2941499999961
0.0009101524214365 26.8135270000012
0.000914848168306 28.0521579999982
0.0009181191978315 27.3343399999995
0.00095230470152 28.7968520000005
0.000954761222611 27.3446679999998
0.000959790925001 27.3936859999997
0.000981163360295928 14.2786110000002
0.000985524338298875 14.8479650000001
0.000991251030629584 22.6564560000015
0.000994346231230176 15.7006250000001
0.000998799281019742 21.4489660000015
0.00100165426563 16.0921910000004
0.00100749973353 15.8624130000005
0.00101138130559 15.6328850000003
0.00110563288523667 15.165496
0.0012457397417 13.2431530000002
0.00125521263354 15.0831399999999
0.00134898317914 13.4151099999999
0.00135611294534 15.4558780000002
0.00147141347477 13.1815700000005
0.00147854612149 12.9118660000004
0.00151228225734 3.34232000000002
0.00152762647136 3.24622100000004
0.00153169773782 13.2362430000003
0.00154298529365 13.0405530000004
0.0015526802791 12.8866380000003
0.00160323936772857 3.35218899999998
0.00160776268993437 9.75096099999995
0.00161595835644875 8.70159200000012
0.00162365803446333 7.32040800000015
0.00162906210730889 3.40064100000006
0.0016397705949975 3.53028800000006
0.00164430546038833 3.47047100000006
0.0016494392719275 3.77707500000005
0.00166219183754 3.33250100000006
0.001665512593 3.27130800000004
0.0016776084294 3.17432900000002
0.00169713595455 3.21848799999996
0.00172505790934 3.28989099999999
0.00173927036233 3.24402699999998
0.00175821376495 3.16073899999998
0.00185990376843 3.39204699999996
0.001874236136795 3.11851399999996
0.00200837974719 3.00519699999999
0.00202386821951 2.99724899999996
0.00204695500937 3.00066099999995
0.00205922574622 2.89798599999998
0.00207299838124 2.96865799999999
0.00207363653959 3.05860699999998
0.00208760631721667 2.80298499999997
0.00215259003478625 2.41263500000001
0.00217377843189 2.85465699999998
0.00252399607583 2.372277
0.00285229592348 2.35633000000002
0.00288848750212 2.28218900000003
0.00292826364476 2.29672400000001
0.003426858180955 2.51615700000003
0.00393458200465 2.138846
0.00398398449074 2.24905700000001
0.00398607940283 2.212697
0.00403096434561 2.15239300000001
0.0041822005356 2.13866999999999
0.004190465814158 1.910265
0.00421124559292667 1.858181
0.0042509789662 1.83105299999998
0.00427443175164 1.76151999999998
0.004282512831564 1.592591
0.00429752568439 1.647836
0.00432481144988 1.49686699999999
0.00434296669641667 1.38529099999998
0.00440247627817 1.36668799999999
0.00442786081377 1.40409499999999
0.00444601752833 1.302578
0.00447335183174 1.25027299999999
0.00450357449599 1.14247299999999
0.00451522277149667 0.412717999999999
0.00453942048376 0.401827999999999
0.00455911329184333 0.374944
0.004578781292512 0.343254
0.00459771037051 0.322570999999999
0.004622683267108 0.508585999999997
0.00463707588289111 0.554004999999997
0.0046610916286025 0.525983999999996
0.00551433624408 0.307455
0.00786218126462 0.307433
0.00800750761571 0.302148999999999
0.0085054816497 0.271561
0.00888182786324 0.254143999999999
0.00921071739139 0.24361
0.00944671915919 0.235196000000001
0.00974764103562 0.226171
0.00978738616827 0.215927000000001
0.0100046770699 0.206945
0.0102440683766 0.176523
0.01027439170435 0.163981999999999
0.0103644001146 0.186165000000001
0.0105621558774 0.154950999999999
0.0110994895487 0.146933999999999
0.0113329095418 0.132124999999999
0.0117176084402 0.130152
0.0118040737031 0.11099
0.012124501541 0.099470999999999
0.0122354578118 0.0884109999999998
0.0125027666171 0.0281459999999991
0.0125966541808 0.0852890000000004
0.012662267381 0.0588949999999997
0.0127585080144 0.036886999999999
0.0128663968753 0.064945
0.0129423751977 0.0687649999999995
0.0130108248144 0.0742459999999996
0.0160678013788 0.021407
};
\label{\figlabel-line4}
\coordinate (legend) at (axis description cs:0.03,0.03);
\end{axis}

\matrix [matrix of nodes,
inner sep=1pt, row sep=1pt,cells={anchor=west},anchor={south west},at={(0.03,0.03)}, anchor=south west, draw=none, fill=none] at (legend) {
\ref{\figlabel-line0} {SL}\\
\ref{\figlabel-line1} {$\epsilon^{-3}\log(\epsilon^{-1})$}\\
\ref{\figlabel-line2} {ML}\\
\ref{\figlabel-line3} {$\epsilon^{-2}\log(\epsilon^{-1})^{2}$}\\
\ref{\figlabel-line4} {Adaptive ML}\\
};
\end{tikzpicture}

%% file: content/conclusion.tex
{\section{Conclusion}

We have presented a novel multilevel projection method for the approximation of response surfaces using multivariate polynomials and random samples with different accuracies. For this purpose, we have discussed and analyzed various sampling methods for the underlying single-level approximation method. We have then presented theoretical and numerical results on our multilevel projection method for problems in which samples can be obtained at different accuracies. The numerical results show good agreement with the computational gains predicted by our theory. Future work will address the application to problems in uncertainty quantification with infinite-dimensional parameter domains and multi- or infinite-dimensional quantities of interest.}

%% file: content/acknlowledgments.tex
\paragraph{Acknowledgements} F. Nobile received support from the Center for ADvanced MOdeling Science (CADMOS). R. Tempone and S. Wolfers are members of the KAUST SRI Center for Uncertainty Quantification in Computational Science and Engineering. R. Tempone received support from the KAUST CRG3 Award Ref:2281 and the KAUST CRG4 Award Ref:2584.